\documentclass[10pt]{article} 
\usepackage[margin=1in]{geometry}
 
\usepackage{listings}
\usepackage{geometry}

\usepackage{tikz}
\usetikzlibrary{shapes,patterns,calc,arrows.meta}

\usepackage{amsmath,amsthm,amsfonts,amssymb,amscd}
\usepackage[colorlinks=true, pdfstartview=FitV, linkcolor=blue,citecolor=blue, urlcolor=blue]{hyperref}
\usepackage{times}

\usepackage{comment}
\usepackage{mathtools}
\usepackage{mathabx}
\usepackage{float}
\usepackage{makeidx}
\usepackage{stmaryrd}
\usepackage{mathrsfs}
\usepackage{emptypage}
\usepackage{xfrac}
\usepackage{bbm}

\usepackage[doi=false,isbn=false,url=false,eprint=false,maxbibnames=99,uniquename=false,sortcites=true,giveninits=true]{biblatex} 
\DeclareFieldFormat{pages}{#1}

\DeclareFieldFormat[article]{pages}{#1}

\DeclareFieldFormat[article]{volume}{\textbf{#1}}
\DeclareFieldFormat[article]{number}{(#1)}

\AtEveryBibitem{\clearfield{month}}

\DeclareFieldFormat{eprint:arxiv}{arXiv:#1}

\DeclareNameAlias{sortname}{family-given}
\DeclareNameFormat{family-given}{%
  \namepartfamily\addcomma\space
  \namepartgiveni
  \ifthenelse{\value{listcount}<\value{liststop}}
    {\addcomma\space}
    {}%
}

\DeclareFieldFormat[article]{volume}{#1}
\DeclareFieldFormat[article]{number}{(#1)}

\renewbibmacro*{volume+number+eid}{%
  \printfield{volume}%
  \printfield{number}%
  \printfield{eid}}

\DeclareFieldFormat{pages}{%
  \ifinteger{#1}
    {#1}
    {\replacemathtextdashes{#1}}%
}

\newcommand{\replacemathtextdashes}[1]{%
  \ifstrequal{#1}{-}{--}{\replacemathtextdashesaux#1\relax}%
}
\def\replacemathtextdashesaux#1#2\relax{%
  \ifstrequal{#1}{-}{--}{{#1}}%
  \ifx\relax#2\relax
    {}%
  \else
    \replacemathtextdashesaux#2\relax
  \fi
}

\renewcommand{\epsilon}{\varepsilon}

\usepackage{graphicx}
\newcommand{\p}{\ensuremath{\partial}}

\newcommand{\mc}{\ensuremath{\mathcal}}

\definecolor{labelkey}{rgb}{0,0,1}

\def\les{\lesssim}

\def\eps{\varepsilon}

\newcommand{\mb}{\mathbb}

\renewcommand*{\tilde}{\widetilde}
\renewcommand*{\hat}{\widehat}
\renewcommand*{\bar}{\overline}

\newtheorem{theorem}{Theorem}[section]
\newtheorem{lemma}[theorem]{Lemma}

\newtheorem{proposition}[theorem]{Proposition}
\newtheorem{corollary}[theorem]{Corollary}
\theoremstyle{definition}
\newtheorem{definition}[theorem]{Definition}

\newtheorem{remark}[theorem]{Remark}

\numberwithin{equation}{section}

\def\p{\partial}

\def\f1r{{\frac{1}{r}}  }
\def\p{\partial}

\def\f1r{{\frac{1}{r}}  }

\allowdisplaybreaks[4]

\title{Low regularity analysis of the Zakharov--Kuznetsov equation on $\mathbb R \times \mathbb T$}
 \author{G. Cao-Labora\thanks{Courant Institute of Mathematics, New York University, NY 10012, USA. Email: gc2703@nyu.edu}}
 
\date{\vspace{-5ex}}
\begin{document}


\maketitle
 \begin{abstract}
We consider the Cauchy problem for the Zakharov--Kuznetsov equation in the cylinder. We improve the local wellposedness to spaces of regularity $s > 1/2$. The result is optimal in terms of the corresponding bilinear estimate or Picard iteration. Our method is based on an improvement of the understanding of the resonant set, identifying and exploiting its particular geometric properties. We also consider the problem under randomization of the initial data, in which case we obtain solutions for generic data in $H^{s}$ for some $s < 0$. To do so, we consider a novel approach based on lower regularity modifications of the classical $X^{s, b}$ spaces that allow to control concentration of mass in small sets of frequencies. 
 \end{abstract}

\section{Introduction}

\subsection{Historical background}

The Zakharov--Kuznetsov equation is the two-dimensional equation given by
\begin{equation} \label{eq:ZK}
	\begin{cases}
		u_t + \p_{x_1} \Delta u &= \frac12 \p_{x_1} (u^2), \\ 
		u(x, 0) &= u_0 \in H^s .
	\end{cases}
\end{equation}
The dispersion relation and the linear propagator are given as follows
\begin{equation} \label{eq:dispersive_rel}
	\phi(\xi_1, \xi_2) = \xi_1 (\xi_1^2 + \xi_2^2), \qquad \mbox{ and } \qquad \widehat{S(t)u_0} (\xi_1, \xi_2) = e^{i t \phi (\xi_1, \xi_2)} \hat u_0 (\xi_1, \xi_2).
\end{equation}

The equation was introduced by Zakharov and Kuznetsov in \cite{ZK} as a model for propagation of ion-sound waves. It was independently derived by Laedke
and Spatschek as a two-dimensional model from the equations of motion for hydrodynamics \cite{laedke}, which was further justified by Lannes - Linares - Saut \cite{lannes}. We also note that in one dimension this takes the form of the well-known KdV equation, thus the Zakharov--Kuznetsov can also be seen as an extension of the KdV equation in two dimensions.

Equation \eqref{eq:ZK} has the following conserved quantities:
\begin{equation} \label{eq:cons}
	M(u) = \int u^2 dx , \qquad \mbox{ and }\qquad E(u) = \frac12 \int |\nabla u|^2 + \frac13 \int u^3.
\end{equation}
Moreover, assuming $u$ has sufficient regularity and decay in $x_1$, one can also write the following conservation law:
\begin{equation*}
	\p_t \int_{\mathbb R} u(x)  dx_1 = - \int_{\mathbb R} \p_{x_1} \Delta u dx_1 + \frac12 \int_{\mathbb R} \p_{x_1} (u^2) dx_1 = 0
\end{equation*}
for any fixed $x_2$. In Fourier space, $\hat u (0, \xi_2, t )$ is formally invariant under the flow for any $\xi_2$. That is, the Fourier modes of $u$, whose first component vanish, remain constant. One can also observe that the equation admits the following rescaling $u_\lambda (x, t) = \lambda^2 u(\lambda x, \lambda^3 t)$. In particular, the critical regularity is $s_c = -1$, given that $\| u_\lambda \|_{\dot{H}^{-1}} = \| u \|_{\dot{H}^{-1}}$. 

There is extensive mathematical work in understanding the Zakharov--Kuznetsov equation. We will focus our discussion on the works on the original ZK equation (two dimensions and quadratic nonlinearity). We refer to the works \cite{RibaudGZK, LinaresPastorGZK, GrunrockGZK, FarahGZK, BiagioniGZK, MolinetGZK, Kato} for a study of the generalized ZK equation, which considers higher order nonlinearities. We also refer to the works \cite{Ribaud3d, LinaresSaut3D, Grunrock3D, Correia, KinoshitaDD, Linares3D, Kato} for a discussion of the higher dimensional ZK equation.

The mathematical analysis of the Zakharov--Kuznetsov equation was initiated by the work of Faminskii \cite{faminskii1995}, where local-wellposedness is established in $H^1(\mathbb R^2)$, and global-wellposedness follows from conservation of energy. The local-wellposedness proof is based on the approach developed by Kenig, Ponce and Vega in the context of the KdV equation \cite{Kenig-Ponce-Vega}. This was later improved to $H^s (\mathbb R^2)$ with $s > 3/4$ by Linares and Pastor \cite[Theorem 1.6]{linares2009}. The range $s > 1/2$ was achieved independently in the works of Gr\"unrock and Herr \cite{grunrock2014}, and Molinet and Pilod \cite{molinet2015}. Recently, Kinoshita \cite{kinoshita2021} established local wellposedness for $s > -1/4$ and demonstrated that the Picard iteration fails for $s < -1/4$. In particular, global wellposedness in $L^2$ follows directly from Kinoshita result given the conservation of mass.

In the case of the periodic setting $\mathbb T^2$, Linares, Panthee, Robert and Tzvetkov \cite{linares2019} obtained local wellposedness in $H^s(\mathbb T^2)$ for $s > 5/3$. This was later improved by Schippa \cite{schippa2020} to $s > 3/2$ and by Kinoshita-Schippa \cite{kinoshita2021a} to $s > 1$.

Our work is devoted to the case of $\mathbb R \times \mathbb T$. The study of \eqref{eq:ZK} in that setting was initiated by Linares, Pastor, Saut \cite{linares2010}, where they showed local wellposedness in $H^s (\mathbb R \times \mathbb T)$ for $s \geq \frac32$. This was later improved to $s \geq 1$ by Molinet and Pilod \cite{molinet2015}. Recently, Osawa \cite{osawa2022} was able to cover $s > 9/10$. Also, in a joint work with Takaoka \cite{osawa2024}, they showed global wellposedness for $s > 29/31$. 

The first part of this paper gives a new proof of the local wellposedness that allows to reach $s > 3/4$ unconditionally, and also $s > 1/2$ under the natural low-frequency condition $\hat u_0(\xi, n) = O(\xi)$ for $|\xi| \ll 1$. The main difficulty is dealing with the resonant set of frequencies characterized by $\Delta = \phi (\nu, k) + \phi (\zeta, m) + \phi (\xi, n)$ where $\nu + \zeta + \xi = 0$ and $k + m + n = 0$. While the direct analysis of the set where $\Delta$ is small is unfeasible, we are able to carry out an analysis of the set where $\Delta$ and all its directional derivatives are small. As we will see, that resonant contribution corresponds to frequencies located around the triangle of vertices $(0, p)$, $(-p/2, -p/2)$ and $(p/2, -p/2)$, for some $p \in \{ k, m, n \}$. The low frequency condition $| \hat u_0 ( \xi, n) | = O(\xi)$ arises naturally when looking at that resonant contribution. 

We also show that both results ($s > 1/2$ with low-frequency condition or $s > 3/4$ without it) are optimal in terms of the Picard iteration, except for the possible endpoint case. The optimality contrasts with the criticality of the problem, given that one usually expects local wellposedness for $s > s_c = -1$. This gap between the Picard method and criticality was already observed in $\mathbb R^2$ (given that Kinoshita \cite{kinoshita2021} showed $s > -1/4$ is optimal for Picard), however, the gap is even larger for the problem in $\mathbb R \times \mathbb T$, due to the absence of dispersion in the second variable.

The second part of the paper consists of studying \eqref{eq:ZK} under random initial data. We will show that for generic initial data, we have local existence of solutions in $H^{s}$, for any $s > \frac{-1}{26}$. In particular, we remark that the result is below $L^2$, and therefore it is the first step towards the construction of a Gibbs measure. The local wellposedness proof is done by randomizing the initial data with some independent and identically distributed Gaussians, and subtracting its linear evolution to the equation. Then, we prove that the remainder is smoother so that it can be treated at the deterministic local wellposedness regularity level.

This type of argument goes back to the pioneering work of Bourgain \cite{bourgain1993, bourgain_2Dinvariant}, where he showed probabilistic local wellposedness for periodic cubic defocusing NLS under Wick reordering for $d = 1, 2$, below the critical threshold in the $d=2$ case. There is a vast literature in PDE where the local wellposedness regularity can be improved by assuming randomness or generic initial data  such as NLS \cite{CLS, NS, DNY, DNY2, oh2020, nahmod2012}, or wave equations \cite{burq2008, burq2008b, LM, bringmann2024}. 

Our approach however differs from the generic approach for probabilistic wellposedness in several points. First of all, usually randomness is used to obtain results below the criticality threshold, while in our case the probabilistic wellposedness is still subcritical, due to the gap between criticality and the deterministic wellposedness theory. We follow the technique initiated by Bourgain \cite{bourgain1993} by subtracting the linear evolution of the rough initial data and showing a higher regularity for the remainder. However, the main estimate, which arises from the interaction of the rough linear evolution and the smoother remainder, is done completely differently. Instead of using Bourgain's approach, which is based on the $TT^\ast$ trick, we take advantage of the better control on $\| \hat u \|_{L^\infty}$ due to randomness. 

A fundamental part of our work is to combine the $L^2$-based $X^{s, b}$ norm, which is usual in dispersive equations, with a \textit{lower regularity} term which is $L^\infty$-based in frequency. Such type of modifications of $X^{s, b}$ have been used in the dispersive literature before (see \cite{grunrock2004, grunrock2008, nahmod2012}), and they basically correspond to measuring the $X^{s, b}$ norm in Fourier space in $L^p$ for some other $p \neq 2$, being $p=2$ the standard $X^{s, b}$ norm. However, to the best of our knowledge, this is the first case where these norms appear as \textit{lower regularity auxiliary term} combined with a standard $X^{s, b}$. The auxiliary part of the norm is irrelevant for most of the analysis, due to its lower regularity. However, it helps in very resonant regions, in order to avoid scenarios where $|\hat u (\xi, n)|$ concentrates in a set of measure $O(1/N^2)$ (being $N\approx | (\xi, n) |$). In those scenarios, the $L^\infty$ control gives a control that is $N$ times better than the $L^2$ control, and this will be enough to compensate for the fact that the $L^\infty$ control happens at a lower regularity level. We think that this type of idea may be widely applicable to other dispersive problems when one needs to rule out very resonant contributions, both in the deterministic or probabilistic settings.

\subsection{Deterministic result}

As usual, for functions $u(x, t)$ we define the space $X^{s, b}$ (cf. \cite{bourgain1993}) by the norm below. We also define the $Y^{s, b}$ space identically, but with a singular weight $\frac{\langle \xi \rangle}{| \xi | }$ that enforces vanishing of $\hat u (\xi, n_2, \tau) $ around $\xi \approx 0$:
\begin{align*}
	\| u \|_{X^{s, b}}^2 &= \int_\mathbb{R}  \sum_{n_2 \in \mathbb Z} \int_\mathbb{R} | \hat u(\xi, n_2, \tau) |^2 \langle \tau - \phi (\xi, n_2) \rangle^{2b} \langle (\xi, n_2) \rangle^{2s} d\xi d\tau,\\
    	\| u \|_{Y^{s, b}}^2 &= \int_\mathbb{R}  \sum_{n_2 \in \mathbb Z} \int_\mathbb{R} |\hat u(\xi, n_2, \tau)|^2 \langle \tau - \phi(\xi, n) \rangle^{2b} \langle (\xi, n_2) \rangle^{2s} \frac{ \langle \xi \rangle }{ |\xi | } d\xi d\tau,
\end{align*}
where $\hat u$ is the Fourier Transform in both space and time and $\langle y \rangle := \sqrt{ 1 + |y|^2 }$ for any $y \in \mathbb R^d$. 

The intuition behind the $Y^{s, b}$ norm is the following. Recall that $\hat u (0, n_2, t)$ is formally conserved along the flow. Therefore, when $u$ is periodic in the first variable, one can remove its average on the $x_1$ direction and enforce $\hat u(0, n_2, t) = 0$ always. In that case, the $Y^{s, b}$ norm and the $X^{s, b}$ are equivalent. In the case of nonperiodic data along $x_1$, the spaces $Y^{s, b}$ should be understood as a technique for replacing the average removal procedure for non-periodic data. We also note that the $X^{s, b}$ and $Y^{s, b}$ norms are equivalent for data supported on $| \xi | \geq c$, for any small $c$.

We define $\tilde H^s (\mathbb R \times \mathbb T )$ to be the corresponding space of initial data, with the norm
\begin{equation} \label{eq:tildeH}
	\| f \|_{\tilde H^s}^2 = \sum_{n_2 \in \mathbb Z} \int_\mathbb{R} | \hat f(\xi, n_2) |^2 \langle (\xi, n_2) \rangle^{2s} \frac{ \langle \xi \rangle }{ |\xi | } d\xi.
\end{equation}
As we see in the following remark, restricting ourselves to $\tilde H^s (\mathbb R \times \mathbb T ) $ is not a significant restriction. $H^s$ data with the corresponding cancellation and with reasonable decay will always be in $\tilde H^s$.

\begin{remark} Let $s \geq 0$. The space $\tilde H^s (\mathbb R \times \mathbb T )$ contains all functions $f\in H^s (\mathbb R \times \mathbb T)$ with the decay condition $x_1  f(x_1, x_2) \in L^2 (\mathbb R \times \mathbb T)$  together with the cancellation $\int f(x_1, x_2) dx_1 = 0$ for all $x_2 \in \mathbb T$. 

This is proved from the following argument. Let $| \xi | \leq 1$, and consider some function $g(x)$ with Fourier Transform $\hat g(\xi)$ such that $\hat g(0) = 0$. Then
\begin{equation*}
\int_{-1}^1 \frac{ | \hat g (\xi) |^2 }{ | \xi | } d\xi \leq 2\| \hat g (\xi) \|_{C^{1/2}_\xi}^2 \les \| \hat g (\xi) \|_{H^1_\xi}^2 = \| \langle x \rangle g(x) \|_{L^2_x}^2 
\end{equation*}

	As an interesting particular case, we notice that $\tilde H^s(\mathbb R \times \mathbb T)$ contains all functions $f \in H^s(\mathbb R \times \mathbb T)$ that are odd on $x_1$ and satisfy the decay condition $x_1  f(x_1, x_2) \in L^2$. 
\end{remark}

As usual, one can also define the short-time $X^{s, b}_T$ norm of a function $u$ defined over $t\in [0, T]$ by taking the infimum of the global norm over all possible extensions of $u$ outside $[0, T]$. We do the same for $Y^{s, b}$ spaces:
\begin{equation} \label{eq:localized_spaces}
	\| u \|_{X^{s, b}_T} = \inf_{ \substack{ u' = u \\ \forall t\in [0, T]} } \|  u'  \|_{X^{s, b}}, \qquad \mbox{ and } \qquad
	\| u \|_{Y^{s, b}_T} = \inf_{ \substack{ u' = u \\ \forall t\in [0, T]} } \|  u'  \|_{Y^{s, b}}.
\end{equation}

\begin{remark} \label{rem:embedding} Let us recall the well-known fact that $X^{s, b}_T$ are continuously embedded in $C([0, T], H^s)$ for $b > 1/2$. The same is true for $Y^{s, b}_T \xhookrightarrow{} C([0, T], \tilde H^s)$. The proof is standard in the case of $X^{s, b}$ spaces and can be found in textbooks such as \cite{tao_dispersive}, \cite{bourgain_ias}. The proof in $Y^{s, b}$ spaces is completely analogous to the $X^{s, b}$ case.
\end{remark}

We are now ready to state our theorems regarding the wellposedness of \eqref{eq:ZK} in the deterministic setting.

\begin{theorem} \label{th:deterministic_Xs_space} Let $s > 3/4$ and fix $u_0 \in H^s (\mathbb R \times \mathbb T )$. There exist $T$ and $b > 1/2$ such that the Zakharov--Kuznetsov equation \eqref{eq:ZK} is locally wellposed in $X^{s, b}_T$. That is, for any $u_0 \in H^s (\mathbb R \times \mathbb T)$ there exists a unique local solution $u \in X^{s, b}_{T}$ and the map $u_0 \to u$ is continuous from $H^s$ to $X^{s, b}_T$. Moreover, Remark \ref{rem:embedding} yields $u \in C([0, T]; H^s (\mathbb R \times \mathbb T))$.
\end{theorem}

\begin{theorem} \label{th:deterministic_Ys_space} Let $s > 1/2$ and fix $u_0 \in \tilde H^s (\mathbb R \times \mathbb T )$. There exist $T$ and $b > 1/2$ such that the Zakharov--Kuznetsov equation \eqref{eq:ZK} is locally wellposed in $Y^{s, b}_T$. That is, for any $u_0 \in \tilde H^s (\mathbb R \times \mathbb T)$ there exists a unique local solution $u \in Y^{s, b}_{T}$, and the map $u_0 \to u$ is continuous from $\tilde H^s$ to $Y^{s, b}_T$. Moreover, Remark \ref{rem:embedding} yields $u \in C([0, T]; \tilde H^s (\mathbb R \times \mathbb T))$. 
\end{theorem}

Using the Duhamel formula to formulate the fixed-point problem
\begin{equation} \label{eq:fixed_point}
u = S(t)u_0 + \frac12 \int_0^t S(t-s) \p_{x_1}(u^2) ds 
\end{equation}
one can reduce the proof of Theorems \ref{th:deterministic_Xs_space} and \ref{th:deterministic_Ys_space} to the bilinear estimates
\begin{equation} \label{eq:one1}
	\| \p_{x_1} ( u  v) \|_{X^{s, b-1+\delta}} \les \|  u \|_{X^{s, b}} \|  v \|_{X^{s, b}}, \qquad \mbox{ and } \qquad \| \p_{x_1} ( u  v) \|_{Y^{s, b-1+\delta}} \les \|  u \|_{Y^{s, b}} \|  v \|_{Y^{s, b}}
\end{equation}
for some $\delta > 0$ that is sufficiently small depending on $s$. The argument to deduce Theorems \ref{th:deterministic_Xs_space} and \ref{th:deterministic_Ys_space} from \eqref{eq:one1} is standard in the dispersive literature (see \cite{bourgain1993, Kenig-Ponce-Vega, tao_dispersive}), and in the paper we will focus on proving the estimate \eqref{eq:one1}.

Moreover, we also show that our range of $s$ is optimal (except up to possibly the endpoint) for \eqref{eq:one1}.
\begin{theorem} \label{th:det_fail} The $X^{s, b}$ inequality from \eqref{eq:one1} does not hold for any $s < 3/4$ and the $Y^{s, b}$ inequality does not hold for any $s < 1/2$. 
\end{theorem}

\begin{remark} The fact that inequality \eqref{eq:one1} fails for a certain $s$ also has a direct implication in terms of wellposedness. While the problem \eqref{eq:ZK} could be wellposed below the regularity needed to do a Picard iteration, the examples showing the failing of \eqref{eq:one1} also show that the data-to-solution map will not be $C^2$. This is the same situation as in Kinoshita result for $\mathbb R^2$ \cite{kinoshita2021}. See \cite[Section 6]{BourgainPicardC2} for this argument, and a precise definition of the data-to-solution map being $C^2$.
\end{remark}

\subsubsection*{Outline of the obstructions and the proof} 

For simplicity, when discussing the proof of \eqref{eq:one1} we omit the subindices $2$ and simply denote by $k, m, n$ the second components of the frequency. Let us now do a brief discussion regarding the proof of the above theorems and the obstructions faced. We recall $\Delta := \phi (\nu, k) + \phi (\zeta, m) + \phi (\xi, n)$, where we have that $\xi = -\nu-\zeta$ and $n = -k-m$. The function $\Delta$ corresponds to the temporal frequency that one obtains in the Duhamel formula when looking at the interaction of two waves at frequencies $(\nu, k)$ and $(\zeta, m)$. Thus, when $| \Delta |$ is large, one expects gains from the oscillation, and most of our work will consist on obtaining bounds on the measure of the set of frequencies for which $| \Delta |$ is small (resonant set).

While the problem of bounding such resonant set is common in dispersive equations, it is extremely equation-dependent. Some cases, like KdV, the resonant set admits an easy analysis because the corresponding $\Delta$ can be factorized in first-degree polynomials (and therefore $| \Delta |$ is small when one factor is small, significantly simplifying the problem). In other dispersive equations one does not have those factorization properties, and needs to deal with $\Delta$ which are second-degree polynomials. In our case $\Delta$ is a third-degree polynomial of four variables, with no factorization or simplification. Moreover, some of the variables of $\Delta$ are discrete (due to the periodic second component of the PDE). It is generally much more difficult to bound the resonant set under such conditions. In particular, a direct analysis of the level sets of $\Delta$ in our case does not seem feasible.


Instead, we will use the analysis of the derivatives of $\Delta$ as an important auxiliary tool. For example, we have that expressing $\Delta$ in $  \nu, k_2,   \zeta, m_2$ variables:
\begin{equation*}
	\p_{\nu} \Delta = (3   \nu^2 + k^2) - (3 (-  \nu -   \zeta)^2 + (-k-m)^2 ) = (3   \nu^2 + k^2) - (3  \xi^2 +n^2) =: \theta_1^2 - \theta_3^2,
\end{equation*}
where we defined $\theta_1 = \sqrt{3   \nu^2 + k^2}$ and $\theta_3 = \sqrt{3  \xi^2 +n^2}$. Roughly speaking, this will tell us that for fixed $k, n, \xi$, the condition $| \Delta | \leq C$ happens in a set of $\nu$ of measure at most $\frac{2C}{| \theta_1^2 - \theta_3^2 |}$. While this is not exactly true because $\theta_i$ depend themselves on the frequencies, the argument can be made rigorous using an appropriate dyadic partition on $| \theta_1 - \theta_3 |$.

Specifically, we will have a dyadic-type decomposition that at each step compares the size of $| \Delta |$ with the sizes of $| \theta_i - \theta_j |$ and decides on what bounds should be used based on that comparison. This will be accompanied by a rigidity result showing that when $| \Delta | \leq N^2 M$ and all $| \theta_i - \theta_j | \leq M$ (being $N$ the size of the maximum frequency and $M \leq N$ arbitrary), then the frequencies are within distance $O(M)$ from a triangle of frequencies of the type $(0, p)$, $(-p/2, -p/2)$, $(p/2, -p/2)$.

We will stop the dyadic decomposition at $M = 1$ and call the remaining frequencies 'bad frequencies' (or $S_{\rm{bad}}$). That is, $S_{\rm{bad}}$ corresponds to configurations of frequencies within $O(1)$ distance from the triangle configuration described above. One can see such configuration in Figure \ref{fig:triangle-diagram}. Under such precise configuration, we will be able to analyze explicitly $\Delta$ when the frequencies lie on $S_{\rm{bad}}$. That contribution will be the only one requiring $s>3/4$ (instead of $1/2$) in the $X^{s, b}$ approach, or requiring the $\frac{ \langle \xi \rangle }{ | \xi | }$ weight in the $Y^{s, b}$ approach. Moreover, that type of configuration will also generate the examples that we need to prove Theorem \ref{th:det_fail}.

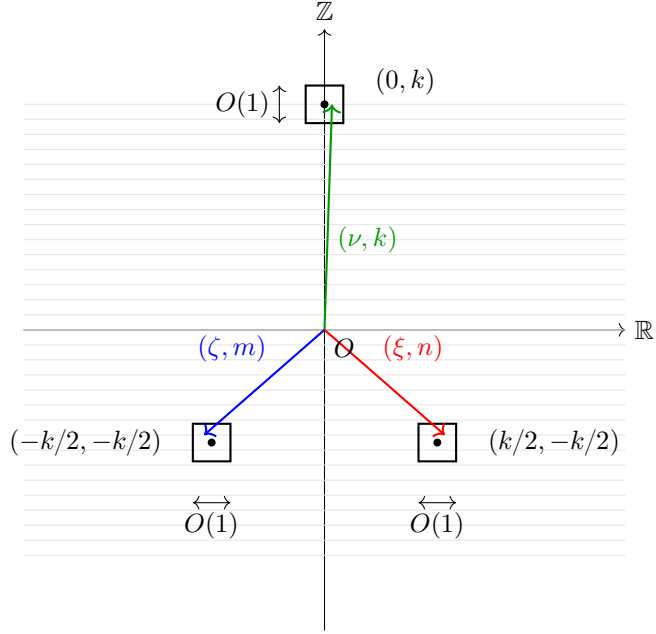
\begin{figure}[htbp]
    \centering
    \begin{tikzpicture}
    \def\k{3} 
    \def\boxsize{0.5} 
    
    \draw[->] (-4,0) -- (4,0) node[right] {$\mathbb{R}$};
    \draw[->] (0,-4) -- (0,4) node[above] {$\mathbb{Z}$};
    
    \foreach \y in {-3,-2.8,...,3}
        \draw[gray!20] (-4,\y) -- (4,\y);
    
    \draw[thick] (-\boxsize/2,\k-\boxsize/2) rectangle (\boxsize/2,\k+\boxsize/2);
    \node[right=0.3cm] at (\boxsize/2,\k+0.3) {$(0,k)$};
    \fill (0,\k) circle (1.5pt);
    \draw[<->] (-0.6,\k-\boxsize/2) -- (-0.6,\k+\boxsize/2) node[midway,left] {$O(1)$};
    
    \draw[thick] (-\k/2-\boxsize/2,-\k/2-\boxsize/2) rectangle (-\k/2+\boxsize/2,-\k/2+\boxsize/2);
    \node[left=0.3cm] at (-\k/2-\boxsize/2,-\k/2) {$(-k/2,-k/2)$};
    \fill (-\k/2,-\k/2) circle (1.5pt);
    \draw[<->] (-\k/2-\boxsize/2,-\k/2-0.8) -- (-\k/2+\boxsize/2,-\k/2-0.8) node[midway,below] {$O(1)$};
    
    \draw[thick] (\k/2-\boxsize/2,-\k/2-\boxsize/2) rectangle (\k/2+\boxsize/2,-\k/2+\boxsize/2);
    \node[right=0.3cm] at (\k/2+\boxsize/2,-\k/2) {$(k/2,-k/2)$};
    \fill (\k/2,-\k/2) circle (1.5pt);
    \draw[<->] (\k/2-\boxsize/2,-\k/2-0.8) -- (\k/2+\boxsize/2,-\k/2-0.8) node[midway,below] {$O(1)$};
    
    \draw[->,green!60!black,thick] (0,0) -- (0.1,\k) node[right,pos=0.4] {$(\nu,k)$};
    
    \draw[->,blue,thick] (0,0) -- (-\k/2-0.1,-\k/2+0.1) node[above left,pos=0.4] {$(\zeta,m)$};
    
    \draw[->,red,thick] (0,0) -- (\k/2+0.1,-\k/2+0.1) node[above right,pos=0.4] {$(\xi,n)$};
    
    \node[below right] at (0,0) {$O$};
\end{tikzpicture}
    \caption{Diagram showing frequencies in $S_{\rm{bad}}$, which corresponds to the main resonance. Each frequency vector should lie in the corresponding box, and the picture is up to rearrangement of the frequencies. In such configuration, we have that $|\Delta | = |\phi (\nu, k) + \phi (\zeta, m) + \phi (\xi, n) |\les 1$. The resonance occurs around the triple cancellation $\Delta = \p_\nu \Delta = \p_\zeta \Delta = 0$.}
    \label{fig:triangle-diagram}
\end{figure}

\subsection{Probabilistic result} \label{subsec:probabilistic}

Now, we look to further improve the deterministic local wellposedness results by considering the case of randomized initial data. We define a smooth even function $\chi_0 : \mathbb R \to [0, 1]$, supported on $[-1, 1]$, such that $\chi_0 (x) + \chi_0 (x-1) = 1$ for all $x \in [0, 1]$. Therefore, we have the pointwise equality $$1 = \sum_{n \in \mathbb Z} \chi_0 (x + n).$$

For any function $f : \mathbb R \times \mathbb Z \to \mathbb C$ and any $k \in \mathbb Z^2$, we define its projection to frequencies around $k = (k_1, k_2)$ as
\begin{equation} \label{eq:Pk}
	\widehat{ P_k f } (\xi, n_2) = \chi_0 (\xi - k_1) \mathbbm{1}_{n_2 = k_2} \hat f (\xi, n_2)
\end{equation} 

\begin{definition}
We say that $u_0$ is \textit{generic} in $H^{(\alpha - 1)-}$ if there exists some $C > 0$ such that
\begin{equation}
	\| \widehat{P_k u_0} \|_{L^\infty} \leq \frac{C}{\langle k \rangle^\alpha}, \qquad \forall k \in \mathbb Z^2.
\end{equation}
\end{definition}

Now, for a generic $u_0$ we consider its randomization
\begin{equation} \label{eq:u01}
	u_{0}^\omega (x) =  \sum_{k\in \mathbb Z^2} g_k( \omega ) P_k u_0 
\end{equation}
where $g_k(\omega)$ are independent and identically distributed complex Gaussians and $\omega$ lives in our probability space. See \cite{LM, Oh2019} for similar randomizations in the fully euclidean setting. In order to enforce real initial data, we take 
\begin{equation}\label{eq:make_real} u_{0, r}^\omega = \text{Re} (u_{0}^\omega) = \frac12 \left( u_{0}^\omega + \overline {u_{0}^\omega} \right) \end{equation} as initial data. From computing $\mathbb E_\omega ( \|u_{0, r}^\omega \|_{H^s}^2 )$ one can check that $u_{0, r}^\omega$ is almost surely on any $H^s$ with $s < \alpha-1$. Now, we present our probabilistic local wellposedness theorem.

\begin{theorem} \label{th:probabilistic} For any $s' > \frac{-1}{26}$, there exist $b > 1/2$ and $s > 1/2$ such that the following holds for every $u_0$ generic on $H^{s'-}$. We consider the Zakharov--Kuznetsov equation \eqref{eq:ZK} with randomized real initial data $u_{0, r}^\omega$ obtained from \eqref{eq:u01}--\eqref{eq:make_real}. Then, almost surely, there exists $T>0$ and a solution on $[0, T]$ of the form 
	\begin{equation*}
		u = S(t)u_{0, r}^\omega + v, \qquad \mbox{ with } v\in Y^{s, b}_T.
	\end{equation*}
	In particular, by Remark \ref{rem:embedding}, we have $v \in C([0, T], \tilde H^s (\mathbb R \times \mathbb T ))$. 
\end{theorem}

For convenience, we will work with $\alpha = s' + 1 > \frac{25}{26}$. We stress that the remainder $v$ is in $Y^{s, b}_T$, with $s > 1/2$, and in particular it is smoother than the linear evolution of the initial data (which just has the almost sure $H^{(\alpha-1)-}$ regularity of the initial data). The uniqueness of $v$ will be guaranteed in a stronger space $Z^{s, b}_T \subset Y^{s, b}_T$, which we define later in \eqref{eq:Zsb_norm}.

\begin{remark} 
	
	Let us justify the assumption of generic initial data. This type of generic initial data is the natural extension of the random initial data considered in the literature for fully periodic cases to our $\mathbb R \times \mathbb T$ space. In the fully periodic case (see \cite{bourgain_2Dinvariant, burq2008, burq2008b, DNY2}) one simply considers $\hat u_0^\omega (n) =  \frac{g_n^\omega}{\langle n \rangle^{\alpha}}$ where the regularity and $\alpha$ are linked by $\alpha = \frac{d}{2} + s + \eps$ (in our case, $d=2$), since that $\alpha$ will generate that $u_0^\omega \in H^s$ almost surely.
	
	The analogy to the above works would be to pick $\hat{u_0^\omega} (\xi, n_2) = \frac{g_n^\omega}{\langle n \rangle^\alpha }$ being $n_1$ some integer with $| \xi - n_1 | \les 1$. However, in the euclidean setting this avoids distinguishing frequencies that are close to each other, so we have preferred to keep more generality in the $\mathbb R \times \mathbb T$ case and allow for any $P_k u_0$ which has the corresponding decay in $L^\infty$. This does not make any significant difference in terms of the analysis.
	
\end{remark}

\begin{remark}
	The fact that we use Gaussians in our randomization is just a convention in the probabilistic local wellposedness literature. However, the result would follow for any random variable with exponentially decaying tails with the same argument. The crucial property is that the random variables $g_k^\omega$ randomizing the initial data are \emph{independent}.
\end{remark}


\subsubsection*{Outline of the proof and obstructions}

Going back to the pioneering works of Bourgain \cite{bourgain1993, bourgain_2Dinvariant}, an important idea to show wellposedness of PDE with random initial data has been to subtract the linear evolution. Then, the remainder will gain some regularity from the randomness on the initial data, so that one can work at a higher regularity level with the equation for the remainder. We take $v = u - S(t) u_{0, r}^\omega$, which satisfies the equation
\begin{equation} \label{eq:pert_PDE}
	\begin{cases}
		(\p_t + \p_{x_1} \Delta ) v &= \frac12 \p_{x_1}(v^2) + \frac12 \p_{x_1} ((S(t)u_{0, r}^\omega)^2) + \p_{x_1} (v S(t) u_{0, r}^\omega ) \\
		v(x, 0) &= 0.
	\end{cases}
\end{equation} 

Similarly to the deterministic setting, local wellposedness for $v$ would be proved by bounding the right hand side in $X^{s, b-1+\delta}$ norm in terms of $\| v \|_{X^{s, b}}$ (or the same with $Y^{s, b}$ norms). However, that direct approach will not work for any $\alpha < 1$, due to a \emph{new problematic high-low type resonance} appearing in the term $\p_{x_1} (S(t)u_{0, r}^\omega ) v$.

As an example, consider that $\hat v$ is concentrated at frequencies $(\zeta, m_2)$ which are of size $O(1)$ and are at a distance $O(N^{-2})$ from the line $\sqrt{3} \zeta =  m_2$, for some large $N$. Then, consider its interaction with frequencies $(\nu, k_2)$ of $\p_x S(t) u_{0, r}^\omega$ which are of size $O(N)$ and located at distance $O(1)$ from the line $\sqrt{3} \nu = - k_2$. One can check that in those cases $| \Delta | \les 1$, so it is not difficult to show that $\| \p_{x_1} ( S(t) u_{0, r}^{\omega} ) v \|_{Y^{s, b}}$ is not bounded in terms of $\| v \|_{Y^{s, b}}$ when $\alpha < 1$.

However, that scenario happens when $\hat v (\zeta, m_2)$ is concentrated in a very small set of low frequencies. Since the set has small measure, an $L^\infty$-based norm on frequency space will give us a better control than the $L^2$-based $Y^{s, b}$-type spaces. Moreover, since the example we want to control happens at low frequencies of $v$, we can still get a better control even when the $L^\infty$-based norm is set at a lower regularity.

We thus define
\begin{equation} \label{eq:Zsb_norm}
\| v \|_{Z^{s, b}}^2 =  \left\| v \right\|_{Y^{s, b}}^2 +  \left\| \frac{\langle \xi \rangle}{|\xi|}\left\| \langle \tau - \phi(\xi, n_2) \rangle^{b} \hat v \right\|_{L^2_{\tau}} \right\|_{L^\infty_{\xi, n_2}}^2
\end{equation}
and its corresponding localized version as in \eqref{eq:localized_spaces}. While the first term of the norm has regularity $s > 1/2$, the second one is at the $L^\infty$ level.

We note that such type of $L^p$-based $X^{s, b}$ has been considered before. First, Gr\"unrock \cite{grunrock2004} considered them with the same $p$ on the temporal and spatial frequency, in the context of the KdV equation. Then, Gr\"unrock and Herr \cite{grunrock2008} considered the spaces in more generality when studying the derivative NLS equation, allowing for $L^p$-based spatial frequencies and $L^q$-based temporal frequencies, with $p \neq q$. In the context of the probabilistic local wellposedness, they have also been used by Nahmod, Oh, Rey-Bellet and Staffilani \cite{nahmod2012} to study the probabilistic wellposedness and Wiener measure of derivative NLS. The novelty of our approach consists on using these norms as a \textit{lower regularity auxiliary terms} with a standard $L^2$-based norm at the top regularity level.

Using a standard Duhamel fixed point argument like in the deterministic case, we can reduce the local wellposedness of \eqref{eq:pert_PDE} for $v \in Z^{s, b}_T$ to the set of estimates:
\begin{align} \label{eq:beck1} \begin{split}
		&\left\| \p_{x_1}(v w) \right\|_{Z^{s, b-1+\delta}_T} \les  \| v \|_{Z^{s, b}_T} \| w \|_{Z^{s, b}_T}, \qquad
		\left\|  \p_{x_1}(v S(t)u_{0, r}^\omega )  \right\|_{Z^{s, b-1+\delta}_T} \les \| v \|_{Z^{s, b}_T}, \\
		&\left\|  \p_{x_1}( (S(t)u_{0, r}^\omega)^2 )  \right\|_{Z^{s, b-1+\delta}_T} \les 1, 
		\qquad \forall v, w \in Z^{s, b}_T,
\end{split} \end{align}
where the implicit constant in $\les$ is allowed to depend on the parameters $\delta, s, b$; but not on $T$, $v$ or $w$. Thus, the time of existence of the solution may also depend on $\delta, s, b, \omega$. 

Therefore, our proof will consist on showing \eqref{eq:beck1}. That yields the existence and uniqueness for the problem \eqref{eq:pert_PDE} with $v \in Z_{T}^{s, b}$, and as a consequence $v \in Y_T^{s, b}$ and $u(\cdot, t) \in S(t)u_{0, r}^\omega + \tilde H^s (\mathbb R \times \mathbb T)$ for all $t \in [0, T]$.

We also remark that Bourgain's technique of subtracting the linear evolution of the initial data was greatly generalized in the recent works of Deng, Nahmod, and Yue \cite{DNY, DNY2}. While the techniques they use differ from the ones presented here, their method also concerns improving the regularity that Bourgain's method would yield, specifically regarding resonant high-low interactions where $S(t)u_0^\omega$ is at high frequency. Therefore, it is reasonable to expect that their set of techniques (possibly combined with $Z^{s, b}$ norms) could allow to lower the regularity. In contrast with the deterministic case, we do not believe that the regularity $s' > -\frac{1}{26}$ in Theorem \ref{th:probabilistic} is optimal.

\subsection{Organization of the paper}

In Section \ref{sec:deterministic} we prove the inequalities in \eqref{eq:one1}, which yield the proofs of Theorems \ref{th:deterministic_Xs_space} and \ref{th:deterministic_Ys_space}. We also include the proof of Theorem \ref{th:det_fail} in Subsection \ref{subsec:det_fail}.

In Section \ref{sec:prob} we prove Theorem \ref{th:probabilistic}. Concretely, in Subsection \ref{subsec:inho} we perform the $Y^{s, b}$ bounds on $\p_{x_1} (S(t) u_{0, r})^2$, and in Subsection \ref{subsec:mixed} the $Y^{s, b}$ bounds corresponding to $ \p_{x_1} (S(t) u_{0, r}^\omega v )$, while the $Y^{s, b}$ bounds on $\p_{x_1}(vw)$ follow directly from the deterministic bounds. Lastly, in Subsection \ref{subsec:Linfty}, we perform the auxiliary $L^\infty$ bounds that allow us to upgrade the $Y^{s, b}$ estimates to $Z^{s, b}$.

\subsection{Notation} \label{subsec:not}
We will always have three frequencies on $\mathbb R \times \mathbb Z$. The first frequency is denoted by $(\nu, k_2)$, the second one by $(\zeta, m_2)$ and the third one by $(\xi, n_2)$ and they will satisfy that $\nu + \zeta + \xi = 0$ and $k_2 + m_2 + n_2 = 0$. Thus, a frequency configuration can be parametrized by $4$ parameters. In the probabilistic part, we will also consider $\tilde \nu, \tilde \zeta, \tilde \xi \in (-1, 1)$ corresponding to arguments of functions like $P_k u$ (which we recall from \eqref{eq:Pk}), and we will define integer values $k_1, m_1, n_1$ so that $\xi = \tilde \xi + n_1$, $\zeta = m_1 + \tilde \zeta $ and $\nu = k_1 + \tilde \nu$. Thus, $k, m, n \in \mathbb Z^2$ are integer points within distance $1$ from the actual frequency $(\nu, k_2)$, $(\zeta, m_2)$ or $(\xi, n_2)$.

In the deterministic part (Section \ref{sec:deterministic}), the notation is slightly different, since we will not need to approximate the first component of the frequency by an integer. Thus, in that section, $k, m, n$ refers to the \emph{second component of the frequency}, so that the frequencies are $(\nu, k)$, $(\zeta, m)$ and $(\xi, n)$ in that case. In the probabilistic part, we keep $k, m, n \in \mathbb Z^2$, with $k_1 \in \{ \lfloor \nu \rfloor, \lceil \nu \rceil \}$ (analogously for $m_1, n_1$).

We will abbreviate $$\phi_1 := \phi(\nu, k_2), \qquad \phi_2 := \phi (\zeta, m_2), \qquad \phi_3 := \phi (\xi, n_2), \qquad \Delta := \phi_1 + \phi_2 + \phi_3.$$ We will denote the time-frequencies by $\mu$, $\eta$ and $\tau$, which also add up to zero. We also define $$\tilde \mu := \mu - \phi_1, \qquad \tilde \eta := \eta - \phi_2, \qquad \tilde \tau := \tau - \phi_3, \qquad \tilde \mu + \tilde \eta + \tilde \tau = -\Delta$$
where the last equation follows from $\mu + \eta + \tau = 0$ and the definition of $\Delta$.

We will also consider dyadic values of different frequencies. To that purpose, denote $d(x)$ to be the dyadic value of $x$ (for $x \in \mathbb R$ or $x\in \mathbb R^2$). That is, $d(x)$ is the integer power of $2$ such that $d(x) \leq |x| < 2 d(x)$ (or simply $|x| < 2$ when $d(x) = 1$). We define $$ N_1 = d(\nu, k_2), \quad N_2 = d (\zeta, m_2), \quad N_3 = d(\xi, n_2), \quad N_{\ast 1} = d(\nu), \quad N_{\ast 2} = d(\zeta ), \quad N_{\ast 3} = d(\xi) $$ 
$$L_1 = d(\tilde \mu), \quad L_2 = d(\tilde \eta), \quad L_3 = d(\tilde \xi)$$
\begin{align*}
	L_{\rm{min}} = \min \{ L_1, L_2, L_3 \} \qquad L_{\rm{med}} = \mathrm{med} \{ L_1, L_2, L_3 \} \qquad L_{\rm{max}} = \max \{ L_1, L_2, L_3 \} \\
	N_{\rm{min}} = \min \{ N_1, N_2, N_3 \} \qquad N_{\rm{med}} = \mathrm{med} \{ N_1, N_2, N_3 \} \qquad N: = N_{\rm{max}} = \max \{ N_1, N_2, N_3 \}
\end{align*}
 Notice that we also denote $N_{\rm{max}}$ simply by $N$. Here, the function $\text{med}$ of three arguments refers to the one that is neither the maximum nor the minimum (if some value is repeated, $\text{med}$ gives that repeated value). It should be noted that $N_{\rm{med}} \geq N_{\rm{max}}/4$ from triangular inequality, since the sum of frequencies is zero. In the case of $L_i$, since $\tilde \mu + \tilde \eta + \tilde \tau = - \Delta$, one has $L_{\rm{max}} \les L_{\rm{med}} + | \Delta | \les L_{\rm{med}} + N^3$. We also define the projections 
 $$\widehat{P_N u}(\xi, n_2, \tau) = \hat u (\xi, n_2, \tau) \mathbbm{1}_{d(\xi, n_2) = N}, \qquad \mbox{ and } \qquad
 \widehat{Q_L u}(\xi, n_2, \tau) = \hat u (\xi, n_2, \tau) \mathbbm{1}_{d(\tau - \phi(\xi, n_2) = L}$$

We denote $$\theta_1 := \sqrt{ 3  \nu + k_2^2}, \qquad \theta_2 := \sqrt{ 3 \zeta^2 + m_2^2 }, \qquad \theta_3 := \sqrt{ 3 \xi ^2 + n_2^2 }.$$ We will define the dyadic parameters $M_i$ measuring $|\theta_j - \theta_k |$. Contrary to the case of $L_i$, $N_i$ and $N_{\ast i}$, these parameters $M_i$ may be smaller than $1$. We define the smallest possible value of $M_i$ to be:
\begin{align*}
M_{i,\rm{min}} = \min \{ 2^q \; : \; q \in \mathbb Z, \mbox{ and } 2^q \geq N_{\ast i}^{1/2}/N \}.
\end{align*}
Letting $d'(x)$ be the biggest power of two smaller or equal than $|x|$ (either positive power or negative power), we define
\begin{equation*}
M_1 := \max \{ d'(\theta_2 - \theta_3), M_{1, \rm{min}} \}, \quad 
M_2 := \max \{ d'(\theta_1 - \theta_3), M_{2, \rm{min}} \}, \quad 
M_3 := \max \{ d'(\theta_1 - \theta_2), M_{3, \rm{min}} \}. 
\end{equation*}
The reason behind this truncation is that the corresponding estimates will not improve when splitting at smaller scales than $M_{i, \rm{min}}$.

We will denote by $S$ the set of all possible tuples describing the frequencies:
\begin{equation*}
	S = \{(\nu, \zeta, \xi, k_2, m_2, n_2, \mu, \eta, \tau) \quad \mbox{ s.t. } \quad \nu + \zeta + \xi = 0, \quad k_2+m_2+n_2 = 0, \quad \mu + \eta + \tau = 0\}.
\end{equation*}
We denote one of its elements by $\sigma = (\nu, \zeta, \xi, k, m_2, n_2, \mu_2, \eta, \tau)$, and $d\sigma = d\nu d\zeta dk_2 dm_2 d\mu d\eta$, where $dk_2$ and $dm_2$ are discrete measures over $\mathbb Z$. Here, we determine $\xi, n_2, \tau$ via the other frequencies, but one can parametrize $\sigma$ by other 6-dimensional combinations of frequencies. We will use $\bar S$ and $\bar \sigma$ whenever we are including just the spatial frequencies. We will also denote tuples of dyadic numbers by $\iota = (N_i, N_{\ast i}, L_i, M_i)_{i=1}^3$, and we will denote by $S_\iota$ or $\bar S_{\iota}$ to the subset of frequencies that satisfy the corresponding restrictions. Moreover, if we want to fix some of the frequencies, we indicate it with superindices on $S, \bar S$. For example $\bar S^{\xi, n_2}_\iota$ describes the possible frequencies $\nu, \zeta, k_2, m_2$ satisfying the dyadic restrictions from $\iota$, for some fixed $\xi, n_2$. 

We use the bracket notation $\langle k \rangle = \sqrt{1+|k|^2}$.  We will denote $\tilde u = \mathcal F^{-1}| \hat u |$, that is, the function that is obtained by taking absolute value in frequency space. We also use the notation $A \les B$ to indicate that $A \leq C B$ for some universal constant $C$ independent of all the parameters. We use $A \approx B$ to indicate $A \les B$ and $B \les A$. We use $A \les_\beta B$ to indicate that the implicit constant may depend on $\beta$. We reserve the letter $C$ for constants that may change between different arguments.

\subsection*{Acknowledgements}

Except for small parts of the writing phase, this work has been done while I was completing my PhD at MIT under the supervision of Gigliola Staffilani. I would like to thank her for suggesting the problem, for extremely helpful discussions about it, and for being an extraordinary supervisor.

GCL has been supported by NSF under grants DMS-2052651 and DMS-2306378, and partially supported by the MICINN (Spain) research grant number PID2021–125021NA–I00.

\section{Proof of the Theorems \ref{th:deterministic_Xs_space}, \ref{th:deterministic_Ys_space} and \ref{th:det_fail}}  \label{sec:deterministic} \label{sec:maindet}

We recall that the proof of Theorems \ref{th:deterministic_Xs_space} and \ref{th:deterministic_Ys_space} follows from the inequalities in \eqref{eq:one1}. This section is organized as follows. The introductory subsection \ref{subsec:det1} separates the frequency space between 'good' and 'bad' interactions, and reduces the proof of \eqref{eq:one1} to Proposition \ref{prop:maindet}. Then, subsection \ref{subsec:det2} separates the corresponding integrals dyadically and deals with the high-low interactions, while leaving the most difficult case where all frequencies are of comparable sizes. In order to deal with that case, we will introduce new bilinear estimates, an analysis of the resonant set and a partition of the frequency space in \ref{subsec:det3}. Using those tools, we deal with the remaining 'good' interactions in \ref{subsec:det4}. We treat the 'bad' interactions in \ref{subsec:det5}, and lastly, we prove Theorem \ref{th:det_fail} in \ref{subsec:det6}.

\subsection{Reduction of the proof of 
 \eqref{eq:one1} }  \label{subsec:det1} 

In order to show \eqref{eq:one1}, it is useful to use duality and show the estimate by testing against any $w$ in the dual space. Let us note that the dual space of $Y^{s, b}$ with respect to the $L^2$ inner product $\langle f, g \rangle = \int fg$ is given by $Y^{-s, -b}_\ast$, where
\begin{equation*}
	\| f \|_{Y^{s', b'}_\ast}^2 = \int_\tau \int_\xi \sum_{n} | \hat f (\xi, n, \tau) |^2 \langle \tau - \phi( \xi, n) \rangle^{2b'} \langle |(\xi, n)|\rangle^{2s'} \frac{| \xi |}{\langle \xi \rangle}.
\end{equation*}
The dual space of $X^{s, b}$ with respect to the $L^2$ inner product is just given by $X^{-s, -b}$. We take $b = 1/2+\delta$, so that we can reduce to show the estimates 
\begin{align} \begin{split} \label{eq:twoo}
		\langle \p_{x_1}(u  v),  w \rangle_{L^2}  &\les \|  u \|_{X^{s, \frac12 + \delta}} \|  v \|_{X^{s, \frac12 + \delta}} \| w \|_{X^{-s, \frac12 - 2\delta}}, \quad \forall u, v \in X^{s, \frac12 + \delta}, \quad \forall w \in X^{-s, \frac12 - 2\delta}, \quad s > \frac34, \\
		\langle \p_{x_1}(u  v), w \rangle_{L^2}  &\les \|  u \|_{Y^{s, \frac12 + \delta}} \| v \|_{Y^{s, \frac12 + \delta}} \|  w \|_{ Y^{-s, \frac12 - 2\delta}_\ast}, \qquad \forall u, v \in Y^{s, \frac12 + \delta}, \quad \forall w \in Y^{-s, \frac12 - 2\delta}_\ast, \quad s > \frac12.
\end{split} \end{align}

It is normally more convenient to explicitly extract the weights $u, v, w$ have due to their spaces, and do the estimates with $u, v, w$ in $L^2$. We can achieve this defining
\begin{align*}
	\hat u_{\rm{new}} (\nu, k, \mu) &= \langle |(\nu, k) | \rangle^{-s} \langle \phi(\nu, k) - \mu \rangle^{-\frac12 - \delta} \left( \frac{|\nu|}{\langle \nu \rangle} \right)^\gamma \hat u (\nu, k, \mu ), \\
	\widehat{v}_{\rm{new}} (\zeta, m, \eta) &= \langle |(\zeta, m) | \rangle^{-s} \langle \phi(\zeta, m) - \eta \rangle^{-\frac12 - \delta} \left( \frac{|\zeta|}{\langle \zeta \rangle}  \right)^\gamma \widehat{v} (\zeta, m, \eta), \\
	\widehat{w}_{\rm{new}} (\xi, n, \tau) &= \langle |(\xi, n)| \rangle^s \langle \phi(\xi , n ) - \tau \rangle^{-\frac12 + 2\delta} \left( \frac{\langle \xi \rangle}{| \xi |} \right)^{\gamma}\widehat{w} (\xi, n, \tau),
\end{align*}
where $\gamma = \frac12$ if we are working in $Y^{s, b}$ spaces and $\gamma = 0$ for $X^{s, b}$ spaces. Note also that $\| u_{\rm{new}} \|_{L^2} = \| u \|_{X^{s, 1/2+\delta}}$ (for $\gamma = 0$) and $\| u_{\rm{new}} \|_{L^2} = \| u \|_{Y^{s, 1/2+\delta}}$ (for $\gamma = 1/2$). The same happens with $v$, and with $w$ we have $\| w_{\rm{new}} \|_{L^2} = \| w \|_{X^{-s, 1/2-2\delta}}$ (for $\gamma = 0$) and $\| w_{\rm{new}} \|_{L^2} = \| w \|_{Y^{s, 1/2-2\delta}_\ast}$ (for $\gamma = 1/2$).
Note that the normalisation of $w$ is different because $w \in Y^{-s, \frac12 - 2\delta}_\ast$. Then, estimate \eqref{eq:twoo} reads
\begin{align} \begin{split}\label{eq:three}
		&\left| \int_{S} \left( \frac{\langle \xi \rangle |\nu| | \zeta | }{| \xi | \langle \zeta \rangle \langle \nu \rangle } \right)^\gamma \frac{  | \xi |\langle (\xi, n) \rangle^s}{\langle (\nu, k) \rangle^s \langle (\zeta, m) \rangle^s} \frac{ \hat{u}_{\rm{new}}(\nu, k, \mu) \hat{v}_{\rm{new}} (\zeta, m, \eta) \hat{w}_{\rm{new}}(\xi, n, \tau)  }{\langle \tau - \phi(\xi, n) \rangle^{\frac12 - 2\delta}\langle \eta - \phi(\zeta, m) \rangle^{\frac12 + \delta} \langle \mu - \phi(\nu, k) \rangle^{\frac12 + \delta} } d\sigma \right|  \\
		&\qquad \les \| u_{\rm{new}}  \|_{L^2} \| v_{\rm{new}} \|_{L^2} \| w_{\rm{new}} \|_{L^2},
\end{split} \end{align}
where we refer to the notation subsection \ref{subsec:not} for the definition of the set of frequencies $S$ and the measure $d\sigma$.

We moreover partition $S$ with the disjoint union $S = S_{\rm{good}} \cup S_{\rm{bad}}$. We define $S_{\rm{bad}}$ to be the subset of $S$ satisfying:
\begin{align} \begin{split} \label{eq:defSbad}
		|\nu | &\leq \frac{1}{10}, \qquad \left| \zeta + \frac{k}{2} \right|, \left| \xi - \frac{k}{2} \right| \leq \frac{1}{10}, \qquad m = n = \frac{-k}{2}, \\ 
		&\mbox{ and } \qquad | \mu - \phi(\nu, k)|, | \eta - \phi(\zeta, m) |, | \tau - \phi(\xi, n) | \leq 2048|k|^3
\end{split} \end{align}
up to rearrangement of the triplets $(\nu, k, \mu)$, $(\zeta, m, \eta)$ and $(\xi, n, \tau)$.  We define $S_{\rm{good}} = S \setminus S_{\rm{bad}}$. 

The use of $Y^{s, b}$ spaces is only relevant in the region $S_{\rm{bad}}$, which is very small. Thus, in $S_{\rm{good}}$, we can use the bound $\left( \frac{\langle \xi \rangle |\nu| | \zeta | }{| \xi | \langle \zeta \rangle \langle \nu \rangle } \right)^\gamma |\xi|\les \langle \xi \rangle$ for both $\gamma = 0$ and $\gamma = 1/2$. This reduces the proof of \eqref{eq:three} (and therefore, the proof of Theorems \ref{th:deterministic_Xs_space} and \ref{th:deterministic_Ys_space}) to proving the following proposition.
\begin{proposition} \label{prop:maindet} We have that for any $s > 1/2$:
	\begin{equation} \label{eq:maindet1}
		\left| \int_{S_\mathrm{good}} \frac{  \langle \xi  \rangle \cdot \langle (\xi, n_2) \rangle^s}{\langle (\nu, k_2) \rangle^s \langle (\zeta, m_2) \rangle^s} \frac{ \hat{u}(\nu, k_2, \mu) \hat{v} (\zeta, m_2, \eta) \hat{w}(\xi, n_2, \tau)  d\sigma }{\langle \tau - \phi(\xi, n_2) \rangle^{\frac12 - 2\delta}\langle \eta - \phi(\zeta, m_2) \rangle^{\frac12 + \delta} \langle \mu - \phi(\nu, k_2) \rangle^{\frac12 + \delta} } \right|  \les \| u  \|_{L^2} \| v \|_{L^2} \| w \|_{L^2}.
	\end{equation}
	Moreover, we have that for $s > 1/2$:
	\begin{align} \label{eq:maindet2} \begin{split} 
			&\left| \int_{S_\mathrm{bad}} \frac{| \xi |^{1/2} |\nu|^{1/2} | \zeta |^{1/2} }{\langle \xi \rangle^{1/2} \langle \zeta \rangle^{1/2} \langle \nu \rangle^{1/2} } \frac{  \langle \xi \rangle \cdot \langle (\xi, n_2) \rangle^s}{\langle (\nu, k_2) \rangle^s \langle (\zeta, m_2) \rangle^s} \frac{ \hat{u}(\nu, k_2, \mu) \hat{u} (\zeta, m_2, \eta) \hat{w}(\xi, n_2, \tau) d\sigma }{\langle \tau - \phi(\xi, n_2) \rangle^{\frac12 - 2\delta}\langle \eta - \phi(\zeta, m_2) \rangle^{\frac12 + \delta} \langle \mu - \phi(\nu, k_2) \rangle^{\frac12 + \delta} } \right| \\
			\qquad &\les \| u  \|_{L^2} \| v \|_{L^2} \| w \|_{L^2},
	\end{split} \end{align}
	and for any $s > 3/4$:
	\begin{equation} \label{eq:maindet3}
		\left| \int_{S_\mathrm{bad}} \frac{  \langle \xi  \rangle \cdot \langle (\xi, n_2) \rangle^s}{\langle (\nu, k_2) \rangle^s \langle (\zeta, m_2) \rangle^s} \frac{ \hat{u}(\nu, k_2, \mu) \hat{v} (\zeta, m_2, \eta) \hat{w}(\xi, n_2, \tau) d\sigma }{\langle \tau - \phi(\xi, n_2) \rangle^{\frac12 - 2\delta}\langle \eta - \phi(\zeta, m_2) \rangle^{\frac12 + \delta} \langle \mu - \phi(\nu, k_2) \rangle^{\frac12 + \delta} } \right|  \les \| u  \|_{L^2} \| v \|_{L^2} \| w \|_{L^2}
	\end{equation}
\end{proposition}

The advantage of introducing $Y^{s, b}$ spaces can be seen from the statement of this proposition. Note that in $X^{s, b}$ spaces we need the estimates \eqref{eq:maindet1} and \eqref{eq:maindet3} (which are actually the same estimate, but in regions $S_{\rm{good}}$ and $S_{\rm{bad}}$). Thus, the only obstruction to obtain $X^{1/2+, b}$ local wellposedness is the region $S_{\rm{bad}}$. The introduction of $Y^{s, b}$ spaces, motivated by the standard average removal procedure of the periodic case, gives the extra factor $\frac{| \xi |^{1/2} |\nu|^{1/2} | \zeta |^{1/2} }{\langle \xi \rangle^{1/2} \langle \zeta \rangle^{1/2} \langle \nu \rangle^{1/2} }$, and since one of the frequencies $\nu, \zeta, \xi$ will be close to $0$ in $S_{\rm{bad}}$ (see \eqref{eq:defSbad}), this will give a gain and allow to conclude the estimate for $s > \frac12$. 

\subsection{Reduction to the case of comparable frequencies} \label{subsec:det2}

We start the proof of Proposition \ref{prop:maindet}, where the main part will be showing \eqref{eq:maindet1}. First of all, we formulate a version of estimate \eqref{eq:maindet1} projected to dyadic frequencies and we show that such estimate suffices to conclude \eqref{eq:maindet1}. Concretely, let us show that \eqref{eq:maindet1} is implied by the estimate
\begin{align} \begin{split} \label{eq:maindet11}
&\frac{N_3^{1+s}N_1^{-s} N_2^{-s}}{ L_1^{\frac{1+\delta}{2} } L_2^{\frac{1+\delta}{2} } L_3^{ \frac12 - 3\delta}} \int_{S_\mathrm{good}} \left| \widehat{P_{N_1} Q_{L_1} u} (\nu, k, \mu) \widehat{P_{N_2} Q_{L_2} v} (\zeta, m, \eta)  \widehat{P_{N_3} Q_{L_3} w} (\xi, m, \tau) \right| d\sigma \\
&\quad \les_{s, \delta, \eps} N_{\rm{min}}^{-(s-1/2) + \eps}  \| u \|_{L^2} \| v \|_{L^2} \| w \|_{L^2}, 
\end{split} \end{align}
uniformly on the dyadic numbers $N_i, L_i$, for all $s > 1/2$ and for $\delta, \eps > 0$ sufficiently small (depending on $s$). We will generally omit $s$, $\delta$, $\eps$ on $\les_{s, \delta, \eps}$ and it should be assumed that implicit constants are allowed to depend on those (but not on anything else). Morevoer, the order of determining those parameters is $\delta \ll \eps \ll s-1/2$.

Let us show how to deduce \eqref{eq:maindet1} from \eqref{eq:maindet11}. Let us call $\mathcal I_{N_i, L_i}$ to the left hand side of \eqref{eq:maindet11}. Clearly, \eqref{eq:maindet11} implies that 
\begin{equation} \label{eq:corrubedo}
\mathcal I_{N_i, L_i} \les  N_{\rm{min}}^{-\eps} \| P_{N_1} Q_{L_1 }u \|_{L^2} \| P_{N_2} Q_{L_2} v \|_{L^2} \|P_{N_3} Q_{L_3} w \|_{L^2},
\end{equation}
since substituting $u$ by $P_{N_1}Q_{L_1} u $ (and analogously $v$ and $w$) does not change $\mathcal I_{N_i, L_i}$.

Decomposing the integral in \eqref{eq:maindet1} over dyadic projections, and using \eqref{eq:corrubedo}, we obtain that
\begin{align}
 \mathcal I &:= \left| \int_{S_\mathrm{good}} \frac{\langle \xi \rangle \cdot \langle (\xi, n) \rangle^s}{\langle (\nu, k) \rangle^s \cdot \langle (\zeta, m) \rangle^s} \cdot \frac{ \hat u (\nu, k, \mu) \hat v (\zeta, m, \eta) \hat w (\xi, m, \tau)  d\sigma }{\langle \mu - \phi(\nu, k) \rangle^{\frac12 + \delta} \cdot \langle \eta - \phi(\zeta, m) \rangle^{\frac12 + \delta}  \cdot \langle \xi - \phi(\xi, n) \rangle^{\frac12 - 2\delta} } \right| \nonumber \\
 &\les \sum_{N_i, L_i}  \frac{1}{ L_1^{\frac{\delta}{2}} L_2^{\frac{\delta}{2}} L_3^{\delta} }\mathcal I_{N_i, L_i}
 \les \sum_{N_i}  N_{\rm{min}}^{-\eps} \sum_{L_1} \frac{\| P_{N_1} Q_{L_1 }u \|_{L^2}}{L_1^{\frac{\delta}{2}} }\sum_{L_2} \frac{\| P_{N_2} Q_{L_2} v \|_{L^2}}{ L_3^{\frac{\delta}{2}}} \sum_{L_3} \frac{\|P_{N_3} Q_{L_3} w \|_{L^2}}{L_3^{\delta}} \nonumber \\
 &\les \sum_{N_i} N_{\rm{min}}^{-\eps} \| P_{N_1} u \|_{L^2}  \| P_{N_2} v \|_{L^2}  \| P_{N_3} w \|_{L^2}. \label{eq:corrubedo2}
\end{align}

In the last line we used Cauchy-Schwarz together with the fact that $\sum_{L_i} \frac{1}{L_i^{\delta/2}} \les 1$.

Now, since the sum in \eqref{eq:corrubedo2} is symmetric with respect to $u$, $v$, $w$, it suffices to bound the sum when $N_1\leq N_2 \leq N_3$. From $(\xi, n) = -(\zeta, m) - (\nu, k),$ we obtain $N_3 \in \{ N_2, 2N_2, 4N_2 \}$. Using Cauchy-Schwarz, we have
\begin{align*}
\mathcal I &\les \sum_{N_1} \frac{\| P_{N_1} u \|_{L^2} }{N_1^{\eps}} \sum_{N_2} \| P_{N_2} v \|_{L^2} \left( \| P_{N_2} w \|_{L^2} + \| P_{2N_2} w \|_{L^2} + \| P_{4N_2} w \|_{L^2} \right) \\ 
&\les \| u \|_{L^2}  \left( \sum_{N_2} \| P_{N_2} v \|_{L^2}^2 \right)^{1/2} \left( \sum_{N_2}  \left( \| P_{N_2} w \|_{L^2} + \| P_{2N_2} w \|_{L^2} + \| P_{4N_2} w \|_{L^2} \right)^2  \right)^{1/2} \\ 
&\leq \| u \|_{L^2} \| v\|_{L^2}\left( 3\sum_{N_2}\left( \| P_{N_2} w \|_{L^2}^2 + \| P_{2N_2} w \|_{L^2}^2 + \| P_{4N_2} w \|_{L^2}^2 \right) \right)^{1/2} \les \| u \|_{L^2} \| v\|_{L^2} \| w \|_{L^2}.
\end{align*}
Thus, we have concluded the proof of \eqref{eq:maindet1} from \eqref{eq:maindet11} for all the dyadic numbers $N_i, L_i$. From now on, we focus on showing \eqref{eq:maindet11}. We consider the following cases: \begin{enumerate}
\item \label{case:low} All $N_1, N_2, N_3 \leq 128$.
\item \label{case:Llarge} $N_{\rm{max}} \geq 256$ and $L_{\rm{max}} \geq 64 N_{\rm{max}}^3$.
\item \label{case:Ndif} $N_{\rm{max}} \geq 256$, $L_{\rm{max}} \leq 32 N_{\rm{max}}^3$ and $N_{\rm{max}} \geq 8N_{\rm{min}}$.
\item \label{case:Neq} All $N_i$ are between $N_{\rm{min}}$ and $4N_{\rm{min}}$. Moreover, $N_{\rm{min}} \geq 64$ and $L_{\rm{max}} \leq 2048 N_{\rm{min}}^3$.
\end{enumerate}

It is clear that those cases cover all the possibilities. The first case is trivial and follow from a standard analysis, common to any dispersive equation. The second and third cases follows from the techniques of Molinet-Pilod \cite{molinet2015}. The more challenging will be the last case, since that it contains the resonances, we will need to introduce new techniques to conclude our bounds. Let us start showing \eqref{eq:maindet11} for the three first cases.

\begin{remark} \label{rem:freq_loc} In the reduction from \eqref{eq:maindet1} to \eqref{eq:maindet11} we have not used the properties of $S_{\rm{good}}$ at any moment. Thus, this reduction is also valid for \eqref{eq:maindet2} and \eqref{eq:maindet3}. That is, for \eqref{eq:maindet2} it suffices to show
\begin{align}\label{eq:maindet21}
& \frac{N_3^s N_1^{-s} N_2^{-s} }{ L_3^{\frac12 - 3\delta}L_1^{\frac12 + \delta/2} L_2^{\frac12 + \delta/2}} \int_{S_\mathrm{bad}} \left| \frac{\langle \xi \rangle^{1/2} |\nu|^{1/2} | \xi |^{1/2} | \zeta |^{1/2} }{ \langle \zeta \rangle^{1/2} \langle \nu \rangle^{1/2} }  \widehat{ P_{N_1} Q_{L_1} u }(\nu, k, \mu) \widehat{ P_{N_2} Q_{L_2} v} (\zeta, m, \eta) \widehat{ P_{N_3} Q_{L_3} w}(\xi, n, \tau) d\sigma \right|  \\
 \qquad &\les N_{\rm{min}}^{-(s-1/2) + \eps} \| u  \|_{L^2} \| u \|_{L^2} \| w \|_{L^2}, \notag
\end{align}
for $s > 1/2$ and $\delta > 0$ sufficiently small. Similarly, instead of showing \eqref{eq:maindet3}, it suffices to show
\begin{align} \begin{split} \label{eq:maindet31}
 \frac{N_3^s N_1^{-s} N_2^{-s} }{ L_3^{\frac12 - 3\delta}L_1^{\frac12 + \delta/2} L_2^{\frac12 + \delta/2}} & \int_{S_\mathrm{bad}} \left|   \langle \xi  \rangle   \widehat{ P_{N_1} Q_{L_1} u }(\nu, k, \mu) \widehat{ P_{N_2} Q_{L_2} v} (\zeta, m, \eta) \widehat{ P_{N_3} Q_{L_3} w}(\xi, n, \tau)   d\sigma \right| \\
 & \les N_{\rm{min}}^{-(s-1/2) + \eps} \| u  \|_{L^2} \| v \|_{L^2} \| w \|_{L^2},
\end{split} \end{align}
for all $s > 3/4$ and $\delta > 0$ sufficiently small. In both cases, the bounds should be uniform in $N_i, L_i$.
\end{remark}

\subsubsection{Case \ref{case:low}}
We want to show \eqref{eq:maindet11} for $1\leq N_1, N_2, N_3 \leq 128$. That is, we want to show
\begin{equation} \label{eq:obj_ss1}
\int_{S_\mathrm{good}}  \frac{ \left|  \widehat{P_{N_1} Q_{L_1} u}(\nu, k, \mu) \widehat{P_{N_2} Q_{L_2} v} (\zeta, m, \eta) \widehat{ P_{N_3} Q_{L_3}  w}(\xi, n, \tau)  \right| d\sigma }{L_3^{\frac{1}{2} - 2\delta} L_2^{\frac12 + \delta} L_1^{\frac12 + \delta} }  \les  \| u  \|_{L^2} \| v \|_{L^2} \| w \|_{L^2}.
\end{equation}
First of all, note that this estimate is trilinear, so we can assume without loss of generality that $\| u \|_{L^2} = \| v \|_{L^2} = \| w \|_{L^2} = 1$. The geometric-arithmetic mean inequality (applied to $\sqrt{|ab|}$, $\sqrt{|bc|}$ and $\sqrt{|ac|}$) tells us that $|a b c| \les |ab|^{3/2} + |ac|^{3/2} + |bc|^{3/2}$. Using that, we get that
\begin{align}
& \int_{S_{\rm{good}}} \frac{  \left| \widehat{P_{N_1} Q_{L_1} u}(\nu, k, \mu) \widehat{P_{N_2} Q_{L_2} v} (\zeta, m, \eta) \widehat{ P_{N_3} Q_{L_3}  w}(\xi, n, \tau) 
 \right|   d\sigma }{ L_1^{\frac{1+\delta}{2}}  L_2^{\frac{1+\delta}{2}}L_3^{ \frac12 - 3\delta } }\label{eq:morcef1} \\
&\les \int_{S_{\rm{good}}} \Bigg( 
\frac{ |\widehat{P_{N_1} Q_{L_1} u}(\nu, k, \mu)|^{3/2} |\widehat{P_{N_2} Q_{L_2} v} (\zeta, m, \eta) |^{3/2} }{ L_1^{\frac34 (1+\delta)} L_2^{\frac34 (1+\delta)}  }
+  \frac{ |\widehat{P_{N_1} Q_{L_1} u}(\nu, k, \mu)|^{3/2} | \widehat{ P_{N_3} Q_{L_3}  w}(\xi, n, \tau)|^{3/2}  }{L_1^{\frac34 (1+\delta)} L_3^{\frac34 - \frac92 \delta} } \nonumber \\
&\qquad \qquad + \frac{ | \widehat{ P_{N_3} Q_{L_3}  w}(\xi, n, \tau)|^{3/2}  |\widehat{P_{N_2} Q_{L_2} v} (\zeta, m, \eta) |^{3/2}  }{ L_2^{\frac34 (1+\delta)} L_3^{\frac34 - \frac92 \delta} } 
\Bigg) d\sigma \nonumber \\
&\les A_u A_v + A_u A_w + A_v A_w, \nonumber
\end{align}
where
\begin{equation*}
A_u = \int_{-256}^{256} \sum_{k = -256}^{256}  \int_{\mathbb R} \frac{\mathbbm{1}_{|\mu - \phi(\nu, k)| < 2L_1}}{L_1^{\frac34 - \frac92 \delta}}|\hat{u}(\nu, k, \mu)|^{3/2} d\mu d\nu ,
\end{equation*}
together with the analogous definitions for $A_v, A_w$. Here we are using that $N_{i} \leq 128$ (and thus any frequency is smaller than $256$). We also used that $$|d\sigma |  = | d\nu d\zeta dk dm d\mu d\eta | = | d\nu d\xi dk dn d\mu d\tau | = | d\zeta d\xi dm dn d\eta d\tau |. $$

Finally, we have that
\begin{equation*}
    A_u \leq \int_{-256}^{256} \sum_{k=-256}^{256} \left( \int_{\mathbb R} | \hat u (\nu, k, \mu)|^2 d\mu \right)^{3/4} \frac{(4L_1)^{\frac14}}{L_1^{\frac34 - \frac92 \delta}}\les \| u \|_{L^2}^{3/2} = 1,
\end{equation*}
and plugging this into \eqref{eq:morcef1} we conclude the proof of \eqref{eq:obj_ss1}.

\subsubsection{Case \ref{case:Llarge}}

Before treating this case, let us state the following bilinear estimates from \cite{molinet2015}.

\begin{proposition}[Molinet - Pilod, \cite{molinet2015}] \label{prop:molinet} Let $N_1, N_2, L_1, L_2$ be dyadic numbers. We have
\begin{equation} \label{eq:molinet1}
\|  P_{N_1} Q_{L_1} u P_{N_2} Q_{L_2} v \|_{L^2} \les \min \{ N_1, N_2 \} \cdot \min \{ L_1^{1/2}, L_2^{1/2} \} \cdot \|  P_{N_1}Q_{L_1}u \|_{L^2} \|  P_{N_2} Q_{L_2} v \|_{L^2}.
\end{equation}

 Let us assume $N_1 \geq 4N_2$. Then, we have that 
\begin{equation} \label{eq:molinet2}
\|  P_{N_1} Q_{L_1} u  P_{N_2} Q_{L_2} v \|_{L^2} \les \frac{N_2^{1/2}}{N_1} L_1^{1/2} L_2^{1/2} \|  P_{N_1}Q_{L_1}u \|_{L^2} \| P_{N_2} Q_{L_2} v \|_{L^2}.
\end{equation}
\end{proposition}

We want to show \eqref{eq:maindet11} for the case $N_{\rm{max}} \geq 256$ and $L_{\rm{max}}\geq 64N_{\rm{max}}^3$. That is, we want to show
\begin{align} \begin{split} \label{eq:manzanas}
\Xi_{N_i, L_i} &:= \int_{S_{\rm{good}}} \frac{  N_3^{1+s}}{N_1^s N_2^s} \cdot \frac{ \left| \widehat{P_{N_1} Q_{L_1} u}(\nu, k, \mu) \widehat{P_{N_2} Q_{L_2} v} (\zeta, m, \eta) \widehat{P_{N_3} Q_{L_3} w} (\xi, n, \tau)  \right|  d\sigma }{L_1^{\frac12 (1+\delta)}L_2^{\frac12 (1+\delta)} L_3^{\frac12 - 3\delta} }  \\
&\les N_{\rm{min}}^{-(s-1/2) + \eps} \| u  \|_{L^2} \| v \|_{L^2} \| w \|_{L^2}. 
\end{split} \end{align}

Recall $L_{\rm{max}} \geq 64 N_{\rm{max}}^3$. Observe that
\begin{align} \label{eq:pera0} \begin{split}
N_{\rm{max}} &< 2N_{\rm{min}} + 2N_{\rm{med}} \leq 4 N_{\rm{med}}\\
L_{\rm{max}} &\leq |\mu | + | \phi (\nu, k) | \leq | \tau - \phi(\xi, n) | + |\phi(\xi, n)| + | \eta - \phi(\zeta, m) | + |\phi(\zeta, m) |  + | \phi(\nu, k) | \\
&\leq 2L_{\rm{med}} + 2L_{\rm{min}} + (2N_{\rm{max}})^3 + (2N_{\rm{med}})^3+ (2N_{\rm{min}})^3 \\
&\leq 4 L_{\rm{med}} + 24 N_{\rm{max}}^3,
\end{split} \end{align}
where in the second line we assumed $L_{\rm{max}} = | \mu - \phi(\nu, k) |$ but the inequality would follow symmetrically for any other case. With the second inequality and the assumption $L_{\rm{max}} \geq 64N_{\rm{max}}^3$, we have $\left(1 - \frac{24}{64}\right) L_{\rm{max}} \leq 4L_{\rm{med}}$, and therefore: 
\begin{equation} \label{eq:peru1} L_{\rm{max}} < 8L_{\rm{med}}. \end{equation}

We can reexpress \eqref{eq:manzanas} in physical space as:
\begin{align} 
\Xi_{N_i, L_i} &=  \frac{   N_3^{1+s}}{N_1^s N_2^s L_1^{\frac12 (1+\delta)}L_2^{\frac12 (1+\delta)} L_3^{\frac12 - 3\delta} }  \int_{\mathbb R} \int_{\mathbb R \times \mathbb T}   P_{N_1} Q_{L_1} \tilde u (x, t)  P_{N_2} Q_{L_2} \tilde v (x, t)    P_{N_3} Q_{L_3} \tilde w(x, t)  dx dt, \label{eq:fresas1}
\end{align} 
where we use $\tilde u$ to denote $\mathcal F^{-1} \left| \hat u\right|$. We will apply Cauchy-Schwarz to the integral in \eqref{eq:fresas1} and then use the bilinear estimate \eqref{eq:molinet1}. We can choose which pair of functions from $\tilde u$, $\tilde v$, $\tilde w$ is estimated with the bilinear estimate, so we take the one corresponding to $P_{N_{\rm{min}}}$ and the one that corresponds to $Q_{L_{\rm{min}}}$ (if they are the same, pick any other second function). We obtain
\begin{align*}
\Xi_{N_i, L_i} &\les \frac{  N_{\rm{max}}^{1+s}}{N_{\rm{med}}^s N_{\rm{min}}^s L_{\rm{min}}^{\frac12 (1+\delta)}L_{\rm{med}}^{\frac12 (1+\delta)} L_{\rm{max}}^{\frac12 - 3\delta} } \left( N_{\rm{min}} L_{\rm{min}}^{1/2} \right) \|    P_{N_1} Q_{L_1} \tilde u \|_{L^2} \|  P_{N_2} Q_{L_2} \tilde v \|_{L^2} \|  P_{N_3} Q_{L_3} \tilde w \|_{L^2}  \\
&\les \frac{  N_{\rm{max}} N_{\rm{min}}^{1-s} }{L_{\rm{max}}^{1 - 3\delta}  }  \|    \tilde u \|_{L^2} \|  \tilde v \|_{L^2} \|   \tilde w \|_{L^2}  \leq \frac{N_{\rm{max}}^{3/2}}{N_{\rm{min}}^{s-1/2} N_{\rm{max}}^{3-9\delta}} \|  u \|_{L^2} \| v \|_{L^2} \|  w \|_{L^2} ,
\end{align*}
where in the second inequality we used $N_{\rm{med}} \gtrsim N_{\rm{max}}$ from \eqref{eq:pera0} and $L_{\rm{med}} \gtrsim L_{\rm{max}}$ from \eqref{eq:peru1}. In the third inequality, we used the assumption $L_{\rm{max}} \geq 64N_{\rm{max}}^3$ and the fact that $\| \tilde u \|_{L_2} = \| | \hat u | \|_{L^2} = \| u \|_{L^2}$. Since $s > 1/2$ and $\delta$ is sufficiently small, we conclude \eqref{eq:manzanas}. 

\subsubsection{Case \ref{case:Ndif}}

We write the integral from \eqref{eq:maindet11} in physical space as:
\begin{align}
\Xi_{N_i, L_i} &:= \frac{N_3^{1+s}N_1^{-s} N_2^{-s}}{ L_1^{\frac{1+\delta}{2} } L_2^{\frac{1+\delta}{2} } L_3^{ \frac12 - 3\delta}} \int_{S_\mathrm{good}} \left| \widehat{P_{N_1} Q_{L_1} u} (\nu, k, \mu) \widehat{P_{N_2} Q_{L_2} v} (\zeta, m, \eta)  \widehat{P_{N_3} Q_{L_3} w} (\xi, m, \tau) \right| d\sigma \nonumber \\
&\leq \frac{N_3^{1+s}N_1^{-s} N_2^{-s}}{ L_1^{\frac{1+\delta}{2} } L_2^{\frac{1+\delta}{2} } L_3^{ \frac12 - 3\delta}} \int_{\mathbb R} \int_{\mathbb R \times \mathbb T}   P_{N_1} Q_{L_1} \tilde u (x, t)  P_{N_2} Q_{L_2} \tilde v (x, t)    P_{N_3} Q_{L_3} \tilde w(x, t)  dx dt. \label{eq:arcadia}
\end{align}

Recall that $N_{\rm{max}} \geq N_{\rm{min}}$, $N_{\rm{max}} \geq 256$ and $L_{\rm{max}} \leq 32 N_{\rm{max}}^3$ in this case. We consider two subcases: \begin{itemize}
\item Subcase A. Here, $\frac{N_1}{N_2} \in \{ 1/2, 1, 2 \}$. Given that $N_{\rm{max}} \geq 8N_{\rm{min}}$, we deduce that either $N_3 = N_{\rm{max}}$ or $N_3 = N_{\rm{min}}$. However, we know that $N_3 \leq 2N_1 + 2N_2$ (from $(\xi,n) = -(\nu, k) - (\zeta, m)$), and using that $N_1/N_2 \in \{1/2, 1, 2\}$, we deduce $N_3 \leq 6 \min \{ N_1, N_2 \}$. Thus, it cannot be that $N_3 = N_{\rm{max}}$, since this would contradict the assumption $N_{\rm{max}} \geq 8 N_{\rm{min}}$. Therefore, in subcase A we have that $N_3 = N_{\rm{min}}$ and $N_1, N_2 \geq \frac12 N_{\rm{max}} \geq 4N_3$.
\item Subcase B. Here, either $N_1 \geq 4N_2$ or $N_2 \geq 4N_1$.
\end{itemize}

\textbf{Subcase A.} \\

Given that $N_1 \geq 4N_3$, we can apply equation \eqref{eq:molinet2} to $\|  P_{N_1}Q_{L_1} \tilde u  P_{N_3} Q_{L_3} \tilde w \|_{L^2}$. Using also that $N_1, N_2 \geq \frac12 N_{\rm{max}}$ and Cauchy-Schwarz, from \eqref{eq:arcadia}, we obtain:
\begin{align*}
\Xi_{N_i, L_i} &\les \frac{N_{\rm{max}}^{1-s}}{L_1^{1/2+\delta/2 } L_2^{1/2+\delta/2} L_3^{1/2 - 3\delta} } \frac{N_{\rm{min}}^{1/2}}{N_{\rm{max}}} L_1^{1/2} L_3^{1/2}
\|  P_{N_1}Q_{L_1} \tilde u \|_{L^2} \| P_{N_2} Q_{L_2} \tilde v \|_{L^2} \|  P_{N_3} Q_{L_3} \tilde w \|_{L^2}  \\
&\les \frac{N_{\rm{max}}^{1/2-s}}{L_1^{\delta/2 } L_2^{1/2+\delta/2} L_3^{- 3\delta} }  \| u \|_{L_2} \| v \|_{L^2} \| w \|_{L^2} \les  \frac{N_{\rm{max}}^\eps}{N_{\rm{max}}^{s-1/2}}\| u \|_{L_2} \| v \|_{L^2} \| w \|_{L^2} 
\end{align*}
In the last inequality, we used the assumption $L_{\rm{max}} \leq 32 N_{\rm{max}}^3$, which implies $L_3^{3\delta} \les N_{\rm{max}}^{9\delta} \les N_{\rm{max}}^{1/2-s}$ (since $s>1/2$ and that $\delta$ is sufficiently small). \\

\textbf{Subcase B.} \\

In this case, $N_1 \geq 4N_2$ or $N_2 \geq 4N_1$, so we are in the hypothesis to apply \eqref{eq:molinet2} to $\|  P_{N_1}Q_{L_1} \tilde u  P_{N_2} Q_{L_2} \tilde v \|_{L^2}$. Using that with Cauchy-Schwarz in \eqref{eq:arcadia}, we get
\begin{align*}
\Xi_{N_i, L_i}  &\les \frac{ N_{\rm{max}}^{1+s}  }{ L_1^{1/2 + \delta/2} L_2^{1/2 + \delta/2} L_3^{1/2 - 3\delta} N_1^s N_2^s  } \frac{\min \{ N_1, N_2\}^{1/2} L_1^{1/2} L_2^{1/2} }{ \max \{ N_1, N_2 \} } \| P_{N_1}Q_{L_1} \tilde u \|_{L^2} \| P_{N_2} Q_{L_2} \tilde  v \|_{L^2} \| P_{N_3} Q_{L_3} \tilde w \|_{L^2} \notag \\ 
&\les \min \{ N_1,  N_2 \}^{1/2 - s}
\|  u \|_{L^2} \| v \|_{L^2} \| w \|_{L^2}  \les 
\frac{1}{N_{\rm{min}}^{s-1/2}}\|  u \|_{L^2} \| v \|_{L^2} \| w \|_{L^2},
\end{align*}
where in the last line we used that $N_{\rm{max}} \les \max \{N_1, N_2 \}$ and $s > 1/2$. 

\subsection{Bilinear estimates and frequency decomposition} \label{subsec:det3}

\subsubsection{New bilinear estimates}
Before showing \eqref{eq:maindet11} for case \ref{case:Neq}, we will introduce some new tools that will be the bulk of the proof. Let us start showing new bilinear estimates, which are a refinement of Proposition \ref{prop:molinet}. Even though we did not give a proof of Proposition \ref{prop:molinet} (which can be found in the original paper \cite{molinet2015}), it is also straightforward to conclude such proof as a particular case of the proof below.

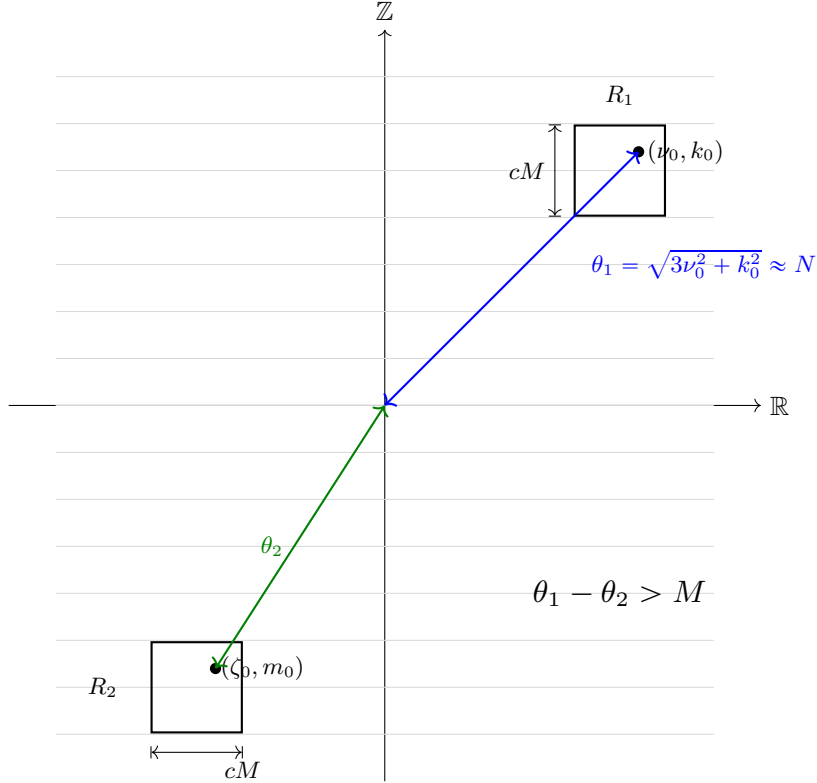
\begin{figure}[htbp]
    \centering
    \begin{tikzpicture}[
    scale=1.25,
    box/.style={
        rectangle,
        draw,
        minimum width=1.2cm,
        minimum height=1.2cm,
        thick
    },
    annotation/.style={
        font=\small
    },
    point/.style={
        circle,
        fill,
        inner sep=1.5pt
    }
]
    \draw[->] (-4,0) -- (4,0) node[right] {$\mathbb{R}$};
    \draw[->] (0,-4) -- (0,4) node[above] {$\mathbb{Z}$};
    
    \foreach \y in {-3.5,-3,-2.5,-2,-1.5,-1,-0.5,0,0.5,1,1.5,2,2.5,3,3.5} {
        \draw[gray!30, thin] (-3.5,\y) -- (3.5,\y);
    }
    
    \node[box] (R1) at (2.5,2.5) {};  
    \node[box] (R2) at (-2,-3) {};    
    
    \node[annotation] at (2.5,3.3) {$R_1$};
    \node[annotation] at (-3,-3) {$R_2$};
    
    \node[point] (p1) at (2.7,2.7) {};
    \node[point] (p2) at (-1.8,-2.8) {};
    
    \node[annotation] at (3.2,2.7) {$(\nu_0,k_0)$};
    \node[annotation] at (-1.3,-2.8) {$(\zeta_0,m_0)$};
    
    \draw[<->, blue, thick] (0,0) -- (2.7,2.7);
    \draw[<->, green!50!black, thick] (0,0) -- (-1.8,-2.8);
    
    \node[annotation, blue, text width=4cm] at (3.8,1.5) 
        {$\theta_1 = \sqrt{3\nu_0^2 + k_0^2} \approx N$};
    \node[annotation, green!50!black] at (-1.2,-1.5) {$\theta_2$};
    
    \node[annotation, scale=1.3] at (2.5,-2) {$\theta_1 - \theta_2 > M$};
    
    \draw[|<->|] ($(R1.north west)-(0.2,0)$) -- ($(R1.south west)-(0.2,0)$) 
        node[midway,left,annotation] {$cM$};
    
    \draw[|<->|] ($(R2.south west)-(0,0.2)$) -- ($(R2.south east)-(0,0.2)$) 
        node[below,annotation] {$cM$};
    
\end{tikzpicture}  
    \caption{Frequency space representation of Proposition \ref{prop:refined}.}
    \label{fig:frequency-space}
\end{figure}

\begin{proposition} \label{prop:refined} Let $M, N, L_1, L_2$ be dyadic numbers with $1 \les M \leq \frac{N}{128}$. Let $R_1$ and $R_2$ be some squares of side smaller than $c M$ (for some small constant $c$) intersected with $\mathbb{R} \times \mathbb Z$. Let $P_1$ and $P_2$ be projections to $R_1$ and $R_2$. Moreover, let us assume that for any $(\nu, k) \in R_i$ we have that $| (\nu, k) |$ is between $\frac{9}{10} N$ and $\frac{10}{9} \cdot 8N$. Then, we have
\begin{equation} \label{eq:refined1}
\| P_1 Q_{L_1} u P_2 Q_{L_2} v \|_{L^2} \les M \min \{  L_1 , L_2 \}^{1/2} \| P_1 Q_{L_1} u \|_{L^2} \| P_2 Q_{L_2} v \|_{L^2}. 
\end{equation}
Let us suppose there exist $(\nu_0, k_0) \in R_1$ and $(\zeta_0, m_0) \in R_2$ such that $\theta_1 = \sqrt{3 \nu_0^2 + k_0^2}$ and $\theta_2 = \sqrt{3 \zeta_0^2 + m_0^2}$ satisfy $|\theta_1 - \theta_2| \geq M$. In that case, we have
\begin{equation} \label{eq:refined2}
\| P_1 Q_{L_1} u P_2 Q_{L_2} v \|_{L^2} \les \frac{1}{N^{1/2} } L_1^{1/2} L_2^{1/2} \| P_1 Q_{L_1} u \|_{L^2} \| P_2 Q_{L_2} v \|_{L^2}. 
\end{equation}
\end{proposition}
\begin{proof} We have that
\begin{align*}
&\widehat{P_1 Q_{L_1} u P_2 Q_{L_2} v }(-\xi, -n, -\tau)  \\
&=  \int_{\mathbb R} \sum_{k\in \mathbb Z} \int_{\mathbb R} \mathbbm{1}_{(\nu, k) \in R_1} \mathbbm{1}_{| \phi (\nu, k) - \mu | \leq 2L_1} \hat{u}(\nu, k, \mu) \mathbbm{1}_{(-\xi - \nu, -n-k) \in R_2} \mathbbm{1}_{| \phi (-\xi - \nu, -n-k) - (-\tau - \mu) | \leq 2L_2} \hat v (-\xi - \nu, n-k, -\tau - \mu) d\nu d\mu
\end{align*}

Using Cauchy-Schwarz, we see that
\begin{align*}
\| P_1 Q_{L_1} u P_2 Q_{L_2} v \|_{L^2}  \les \sup_{\tau \in \mathbb R, (\xi, n) \in -R_1 - R_2} |A_{\tau, \xi, n} |^{1/2} \| P_1 Q_{L_1} u \|_{L^2} \| P_2 Q_{L_2} v \|_{L^2},
\end{align*}
where $A_{\tau, \xi, n}$ is the set of $(\nu, k, \mu)$ such that all the conditions of the $\mathbbm{1}$ inside the integral are satisfied. The rest of the proof will consist on estimating $|A_{\tau, \xi, n} |^{1/2}$. Clearly, $\mu$ can range over at most $O(\min \{ L_1,   L_2 \})$ values, since $|\mu - \phi_1| \leq 2L_1$ and $|(-\tau - \mu) - \phi_2| \leq 2L_2$. Without loss of generality, assume $L_1 \leq L_2$, so that $|A_{\tau, \xi, n} | \les L_1 |\tilde A_{\tau, \mu, \xi, n}|$ where 
\begin{equation*}
\tilde A_{\tau, \mu, \xi, n} = \{ (\nu, k) \in R_1 \quad \text{s.t.} \quad (\xi + \nu, n+k) \in -R_2, \quad |\phi (\xi + \nu, n+k ) - (\tau + \mu)| \leq 2 L_2, \quad \mbox{and} \quad |\mu - \phi(\nu, k)| \leq 2L_1 \}.
\end{equation*}
The trivial estimate $| \tilde A_{\tau, \mu, \xi, n} | \leq | R_1 | \leq c^2 M^2$ yields \eqref{eq:refined1}. Let us now show \eqref{eq:refined2}. 

Using that $|\mu - \phi(\nu, k) | \leq 2L_1 \leq 2L_2$, summing the last two conditions of $\tilde A_{\tau, \mu, \xi, n}$ and applying triangular inequality, we note that $\tilde A_{\tau, \mu, \xi, n} \subset \tilde B_{\tau, \xi, n}$ where
\begin{equation*}
\tilde B_{\tau, \xi, n} = \{ (\nu, k) \in R_1 \quad \text{s.t.} \quad (\xi + \nu, n+k) \in -R_2 \quad \text{ and } \quad |\phi (\xi + \nu, n+k ) - \phi (\nu, k) - \tau| \leq 4 L_2 \}.
\end{equation*}
Now, for any fixed $(\xi, n)$, we compute
\begin{align}
D(\xi, n) &= \inf_{(\nu, k) \in R_1} \left| \frac{\p}{\p \nu}\left( \phi (\xi + \nu, n+k ) - \phi (\nu, k) \right) \right| \notag \\
&=  \inf_{(\nu, k) \in R_1}  \left|  (n+k)^2 + 3(\xi + \nu)^2 - 3 \nu^2 - k^2  \right| = \inf_{(\nu, k) \in R_1} \left| n^2+ 2 kn  + 3 \xi^2 + 6 \xi \nu \right| \label{eq:chocolate}
\end{align}
Let us now study $| \tilde B_{\tau, \xi, n} |$. There are $\les M$ possible values for $k$ (recall the hypothesis $M \gtrsim 1$). Let us fix one of them. We know that the set of $\nu$ such that $(\nu, k) \in R_1$ and $(\xi + \nu, n+k) \in -R_2$ is some connected interval $I_k$ because the regions $R_1$ and $-R_2$ are squares. We see that $(\nu, k) \in \tilde B_{\tau, \xi, n} $ if and only if $\nu \in I_k$ and $|\phi (\xi + \nu, n+k ) - \phi (\nu, k) - \tau| \leq 4 L_2$. As the derivative of this quantity with respect to $\nu$ is greater than $D(\xi, n)$, we have that the set of $\nu$ such that $(\nu, k) \in  \tilde B_{\tau, \xi, n}$ has measure at most $\frac{8L_2}{D(\xi, n)}$. As there are $\les M$ possible values for $k$, we obtain that
\begin{equation*}
\sup_{\tau \in \mathbb R, (\xi, n) \in -R_1 - R_2} | \tilde B_{\tau, \xi, n} | \les  L_2 M \sup_{ (\xi, n) \in -R_1 - R_2}\frac{1}{D(\xi, n)}.
\end{equation*}
Therefore, if we show that $\inf_{(\xi, n) \in -R_1 - R_2} D(\xi, n) \gtrsim  NM$, we are done. Let us do that.

For any $(\xi, n) \in -R_1 - R_2$, we let $(\zeta, m) = -(\xi, n) - (\nu, k)$. We have that $(\zeta, m)$ is in $(R_1 + R_2) - R_1 = R_2 + [-cM, cM]\times [cM, -cM]$. We call that region $\tilde R_2$, and note that it is a square of side $3cM$ centered at the same center as $R_2$.

Rewriting the expression from \eqref{eq:chocolate} in terms of $\nu, \zeta, k, m$, we have
\begin{align*}
\inf_{(\xi, n) \in R_1+R_2} D(\xi, n) &=  \inf_{\substack{ (\nu, k) \in R_1 \\ (\zeta, m) \in \tilde R_2}} \left| ((k+m)^2 + 3 (\nu + \zeta)^2 - 2k(k+m) - 6\nu (\nu + \zeta) \right| \\
&=  \inf_{\substack{ (\nu, k) \in R_1 \\ (\zeta, m) \in \tilde R_2}} \left| (3 \zeta^2 + m^2) - (3 \nu^2 + k^2) \right|
\end{align*}

Now recall every frequency on $R_1, R_2$ is bounded by $\frac{10}{9} \cdot 8N$. As $\tilde R_2$ is a box centered at the same place as $R_2$ and of side $3cM \leq \frac{3c}{128}N$ and $c\leq 1$, we have that any frequency of $\tilde R_2$ is bounded by $10N$. Using that the point $(\nu_0, k_0, \zeta_0, m_0)$ satisfies $|\theta_1 - \theta_2| \geq M$, and noting that $| \nu - \nu_0| \leq cM$ (and the same for other variables) we have
\begin{align*}
\inf_{(\xi, n) \in R_1 + R_2} D(\xi, n) &\geq |\theta_1^2 - \theta_2^2| - \sup_{(\nu, k)\in R_1} \left| 3 \nu^2 + k^2 - 3 \nu_0^2 - k_0^2 \right| - \sup_{(\zeta, m) \in \tilde R_2} \left| 3 \zeta^2 + m^2 - 3 \zeta_0^2 - m_0^2 \right| \\
&\geq  M \cdot (\theta_1 + \theta_2) - cM\sup_{(\nu, k)\in R_1} | 3\nu + 3 \nu_0| - cM\sup_{(\nu, k)\in R_1} |  k + k_0| \\
&\qquad- 3cM\sup_{(\zeta, m)\in R_2} |  m + m_0| - 3cM\sup_{(\zeta, m)\in R_2} | 3\zeta + 3 \zeta_0| \\
&\geq N M - 32cM \cdot 10 N \geq  \frac12 NM,
\end{align*}
assuming that $c < 1/640$. We also used that $\theta_1 + \theta_2 \geq N$ because $\theta_i \geq \sqrt{\nu^2 + k^2} \geq \frac{9 N}{10}$.
\end{proof}

\subsubsection{Analysis of the resonant set}

Let us recall that since we are in case \ref{case:Neq}, we have that $ N_{\rm{min}}\leq N_i \leq 4N_{\rm{min}}$, together with $N_{\rm{min}} \geq 64$ and $L\leq 2048 N_{\rm{min}}^3$. Our objective in this subsection is to subdivide the frequency space corresponding to those restrictions. We have seen in Proposition \ref{prop:refined} the quantities $|\theta_i - \theta_j|$ are of special relevance, since the bilinear estimates depend on their size. They arise as directional derivatives of $\Delta := \phi (\nu, k) + \phi (\zeta, m) + \phi( - \nu - \zeta, -k - m)$. A difficult case is when all the directional derivatives of $\Delta$ are small (one can also see this difficulty from trying to apply the non-stationary phase principle to the Duhamel formulation of the equation). Another relevant case is when $\Delta$ itself is small, which limits the gain one expects when integrating the solution due to its time oscilations. We will be able to divide the frequencies in two cases: one where $\Delta \gtrsim N_{\rm{min}}^2 M$ and another one where some $|\theta_i - \theta_j | \gtrsim N_{\rm{min}}M$. These regions will also depend on the dyadic parameter$M \gtrsim 1$ and will have area $\les N_{\rm{min}} M^3$ in the frequency space ($\nu, k, \zeta, m$). In one case, we will exploit the fact that we are far from space-like resonances (using our bilinear estimates), and in the other one, the fact that we are far from time-like resonances.

First of all, let us study the $\sigma$ of this case under the extra condition $| \theta_i - \theta_j | \leq \frac{N}{1000}$. For all this section, we consider $N_{\rm{min}} \geq 64$ to be fixed and consider:
\begin{align}  \begin{split}
S_0 &=  \left\{ \sigma \in S_{\rm{good}} \quad \mbox{s.t.}\quad N_{\rm{min}} \leq | (\nu, k)|, |(\zeta, m)|, |(\xi, n)| \leq 8 N_{\rm{min}} \right\},\\
S_1 &=  \left\{ \sigma \in S_{\rm{good}} \quad \mbox{s.t.}\quad N_{\rm{min}} \leq | (\nu, k)|, |(\zeta, m)|, |(\xi, n)| \leq 8N_{\rm{min}} , \quad |\theta_i - \theta_j | \leq \frac{N_{\rm{min}}}{1000} \qquad \forall i,j \in \{1, 2, 3\} \right\}. \label{eq:propertiesS4}
\end{split} \end{align}
where we recall $\theta_1(\sigma) = \sqrt{3 \nu^2 + k^2}$ (and similarly $\theta_2, \theta_3$). We note that $N = N_{\rm{max}} \approx N_{\rm{min}}$, however we will work mostly with $N_{\rm{min}}$ since it will be more convenient.

\begin{lemma} \label{lemma:localization} Let $w_1, w_2, w_3 \in \mathbb R^2$ such that $w_1 + w_2 + w_3 = 0$. Moreover, assume they are almost of the same size: $$ |w_1| - \mathcal E \leq |w_i| \leq |w_1| + \mathcal E$$ for every $i$ and some $|\mathcal E| < |w_1|/100$.

Then, we have that
\begin{equation} \label{eq:localization}
w_2 \in R_{\pm 2\pi/3} (w_1) + B, \qquad \mbox{ and } \qquad
w_3 \in R_{\mp 2\pi/3} (w_1) + B.
\end{equation}
where $\pm \in \{ 1, -1 \}$, $\mp = - (\pm)$ and $B = B_0(5\mathcal E)$ is a ball of radius $5\mathcal E$ centered at the origin. We have denoted with $R_{\varphi}$ the rotation of $\varphi$ radians on $\mathbb R^2$ centered at the origin.
\end{lemma}

\begin{figure}[htbp]
    \centering
    \begin{tikzpicture}[
    scale=1.5,
    vector/.style={->,thick},
    ball/.style={draw,dashed,circle}
]
    \draw[vector,blue] (0,0) -- (3,0) node[right] {$w_1$};
    
    \draw[vector,red] (0,0) -- ({-1.35},{2.65}) node[above] {$w_2$};  
    \draw[vector,green!50!black] (0,0) -- ({-1.62},{-2.52}) node[below] {$w_3$}; 
    
    \draw[dashed,gray] (0,0) -- ({-1.5},{2.598076}) node[right=1cm] {$R_{2\pi/3}(w_1) + B$};
    \draw[dashed,gray] (0,0) -- ({-1.5},{-2.598076}) node[right=1cm] {$R_{-2\pi/3}(w_1) + B$};
    
    \draw[ball] ({-1.5},{2.598076}) circle (0.45);
    \draw[ball] ({-1.5},{-2.598076}) circle (0.45);
    
    \draw[->] (0.8,0) arc(0:120:0.8) node[above left] {$\frac{2\pi}{3}$};
    \draw[->] (0.8,0) arc(0:-120:0.8) node[below left] {$-\frac{2\pi}{3}$};
    
    \node[below right] at (1,-2) {$w_1 + w_2 + w_3 = 0$};
    
    \node[above right] at (1,2) {$|w_1| - \mathcal{E} \leq |w_i| \leq |w_1| + \mathcal{E}$};
    
\end{tikzpicture}  
    \caption{Illustration of Lemma~\ref{lemma:localization}. The vectors $w_1$, $w_2$ and $w_3$ (shown in blue, red and green respectively) add up to zero and have approximately equal lengths. $w_2$ and $w_3$ must lie in the dashed balls centered at $R_{\pm 2\pi/3}(w_1)$.}
    \label{fig:localization_figure}
\end{figure}
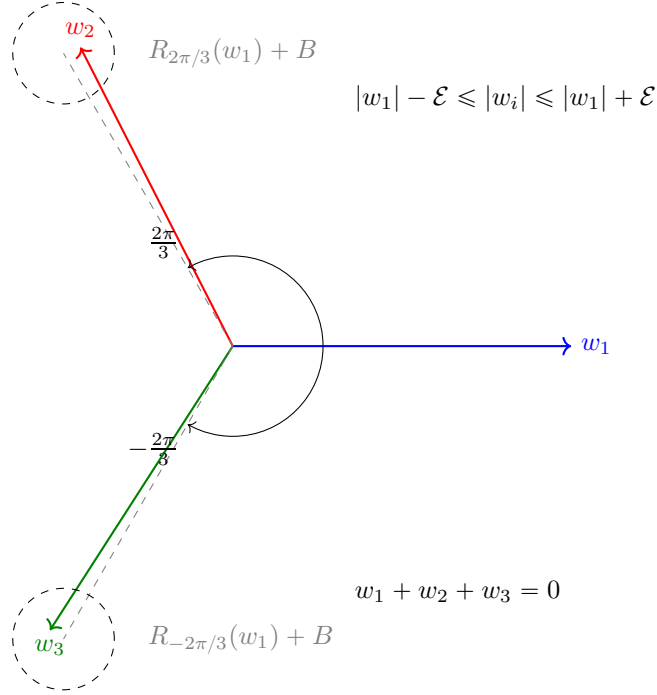


\begin{proof}
From our assumptions, we see that
\begin{equation*}
|w_3|^2 = |-w_1-w_2|^2 = |w_1|^2 + |w_2|^2 + 2 w_1 \cdot w_2
\end{equation*}
We pass $|w_1|^2$ subtracting to the left hand side and use that $\left| |w_1|^2 - |w_3|^2 \right| = \left| |w_1| - |w_3| \right| \cdot ( |w_1| + | w_3 | )\leq \frac52 |w_1|\mathcal E$. Dividing by $2|w_1||w_2|$, we obtain that
\begin{equation*}
\left| \frac{w_1 \cdot w_2}{|w_1| |w_2|} + \frac{|w_2|}{2|w_1|} \right| \leq \frac54 \cdot \frac{\mathcal E} {|w_2|} \leq \frac32 \frac{\mathcal E}{|w_1|}
\end{equation*}
where in the last inequality we use $| |w_1|-|w_2| | \leq |w_1|/100$.

Using that $| |w_1| - |w_2|  | \leq \mathcal E$, we get that
\begin{equation*}
\left| \cos \measuredangle (w_1, w_2) - \frac{-1}{2} \right| \leq 2 \frac{\mathcal E}{|w_1|}
\end{equation*}
Now, we take the function $\arccos :[-1, 1] \to [0, \pi]$ on both terms of the absolute value above. Note that both $ \cos \measuredangle (w_1, w_2)$ and $-1/2$ are in the interval $[-4/10, 6/10]$ (using the hypothesis $\mathcal E \leq |w_1|/100$). In that interval, we have $| \arccos '(x)| \leq 2$ so taking $\arccos$ can only amplify the distance between $\cos \measuredangle (w_1, w_2)$ and $-1/2$ by a factor of $2$. Therefore,
\begin{equation} \label{eq:abelcaballero}
\left| \measuredangle (w_1, w_2) - \frac{2\pi}{3} \right| \leq 4 \frac{\mathcal E}{|w_1|}
\end{equation}
Note that this argument assumes $\measuredangle (w_1, w_2) \in [0, \pi]$ in order to have $\arccos$ well-defined, so does not discriminate if the angle is clockwise or counterclockwise. For some choice of sign $\pm$, we have that
\begin{equation*}
\left| w_2 - R_{\pm 2\pi/3} w_1 \right| \leq \left| \frac{w_2}{|w_2|} |w_1| - R_{\pm 2\pi/3} w_1 \right| + \left| w_2 - \frac{w_2}{|w_2|} |w_1| \right| = |w_1| \left| \measuredangle (w_1, w_2) - \frac{2\pi}{3} \right| + \left| |w_2| - |w_1| \right|
\end{equation*}
Thus, for the appropriate choice of $\pm$, and using \eqref{eq:abelcaballero}, we get
\begin{equation*}
\left| w_2 - R_{\pm 2\pi/3} w_1 \right| \leq 5 \mathcal E
\end{equation*}
The same can be trivially deduced for $w_3$ since the original problem is symmetric. However, note that the choice of signs in $\pm$ for $w_2$ and $w_3$ must be different since $w_1 + 2 R_{\pm 2\pi /3} w_1  \neq 0$ unless $w_1 = 0$. This concludes the desired result.\end{proof} 
\begin{lemma} \label{cor:coro}Let $C = 2^{12}$. Assume that $\sigma \in S_1$, and moreover that $\min \{ | \nu |, | \zeta |, | \xi |\} \geq C \max \{ | \theta_i - \theta_j | \}$. Then we have that 
\begin{equation} \label{eq:nuez}
| \phi (\nu, k) + \phi ( \zeta, m) + \phi ( \xi, n) | \geq  \frac{N_{\rm{min}}^2}{100} \min \{ | \nu |, | \zeta |, | \xi |\}.
\end{equation}

\end{lemma}
\begin{proof}
Note that from $\sigma \in S_1$, we can apply Lemma \ref{lemma:localization} with vectors $w_i$ given by $(\sqrt{3} \nu, k)$, $(\sqrt{3} \zeta, m)$ and $(\sqrt{3} \xi, n)$ and $\mathcal{E} = \frac{N_{\rm{min}}}{1000}$. Therefore, $w_2$ and $w_3$ correspond to rotations of $\pm \frac{2\pi}{3}$ radians of $w_1$, up to an error smaller than $5\mathcal E$. 

In the $(x, y)$ plane, the regions $x \geq |w_1|/2$ and $x \leq -|w_1|/2$ cover each of them $\frac{2\pi}{3}$-radians of the circunference of radius $|w_1|$. Therefore, of the three points $w_1$, $R_{2\pi/3}w_1$ and $R_{-2\pi/3}w_1$, at least one of them is on each area.\footnote{There can be two of them if they are on the boundary of the region} Thus, at most one of those three points is in the region $|x| < |w_1|/2$. If we now consider the points $w_1, w_2, w_3$, which may differ up to an error of $5\mathcal E \leq \frac{5}{100} |w_1|$ from the previous set of three points considered, we see that at most one of those three points is in the region $|x| < \frac{45}{100} |w_1|$. We obtain that \begin{equation} \label{eq:pelota} \text{med} \{ | \nu |, |\zeta |, |\xi| \} \geq \frac{45}{100 \sqrt{3}} |w_1| \geq \frac{2N_{\rm{min}}}{10} \end{equation}

Now, note that 
\begin{align} \label{eq:pera}
\left| \phi( \nu, k ) + \phi( \zeta, m) + \phi ( \xi, n ) \right| &=  \left| \nu ( \theta_1^2 - 2 \nu^2 ) + \zeta (\theta_2^2 - 2 \zeta^2) + \xi (\theta_3^2 - 2 \xi^2) \right|  \\
&= \left| 6 \nu \zeta (\nu + \zeta) + \zeta ( \theta_2^2 - \theta_1^2) + \xi (\theta_3^2 - \theta_1^2 ) \right| \notag \\
&\geq 6 \cdot \frac{2N_{\rm{min}}}{10} \frac{2N_{\rm{min}}}{10} \min \{ |\nu |, | \zeta |, |\xi | \} - 16 N_{\rm{min}} \max | \theta_i + \theta_j | \max | \theta_i - \theta_j | \notag \\
&= \frac{6}{25} N_{\rm{min}}^2 \min \{ |\nu |, | \zeta |, |\xi | \} - 512 N_{\rm{min}}^2 \max | \theta_i - \theta_j |  \notag
\end{align}
In the second equality, we used $\nu\theta_1^2 = -\zeta \theta_1^2 - \xi \theta_1^2$ together with $-2\xi^3 = 2(\zeta + \nu)^3$. In the third line we used \eqref{eq:pelota}. Finally, using our assumption on $\max | \theta_i - \theta_j |$, we get
\begin{equation*}
\left| \phi( \nu, k ) + \phi( \zeta, m) + \phi ( \xi, n ) \right|  \geq \left(\frac{6}{25} - \frac{512}{C}\right) N_{\rm{min}}^2 \min \{ |\nu |, | \zeta |, |\xi | \}, 
\end{equation*}
and plugging in $C = 2^{12}$ we are done. \end{proof}

\begin{remark} One can see from the second line of \eqref{eq:pera} that we also have $| \phi (\nu, k) + \phi ( \zeta, m) + \phi ( \xi, n) | \les  N_{\rm{min}}^2 \rm{min} \{ |\nu|, |\zeta|, |\xi| \}$. We will not use that fact, but it may be helpful to the reader noticing that $N^2 \min \{ | \nu |, |\zeta |, | \xi | \}$ is basically equal (up to a constant) to $\Delta$ under the assumption that we are far from space-like resonances (assumption $\min \{ | \nu |, | \zeta |, | \xi |\} \geq C \max \{ | \theta_i - \theta_j | \}$). 
\end{remark}

The objective of the above lemma is the following. It is difficult to directly study the cubic polynomial $\phi (\nu, k) + \phi (\zeta, m) + \phi (-\nu - \zeta , -k -m)$ in $\mathbb R^2 \times \mathbb Z^2$. However, we only need to lower-bound the absolute value of that cubic, in the case where $|\theta_i - \theta_j |$ is small. In that case, Lemma \ref{cor:coro} tells us that we can look at the much simpler quantity $\min \{ |\nu |, | \zeta |, | \xi | \}$ instead. 

Now, we split the set of frequencies $S_0$ with the a dyadic partition. The idea is to consider a process where we divide $M$ by two at every step until one of the following two stopping conditions is met: either $|\theta_i - \theta_j| \geq 2CM$ or $| \theta_i - \theta_j | \geq M$. We say that a certain frequency is an $M$-interaction if the process stops at $M$. If the first stopping condition is satisfied, we call it interaction of first kind; otherwise, it is of second kind. If no condition is satisfied, we keep halving $M$ until one of them is satisfied or we reach a certain minimum value $\beta$. We give the rigorous definition below.

\begin{definition} \label{def:Minteraction} Let us define an auxiliary constant $C_2 = 512C$. For some small parameter $\beta$ to be fixed later and any dyadic number $M$ with $\beta \leq M < \frac{N_{\rm{min}}}{C_2}$, we say that some element of $S_1$ is an $M$-interaction if: \begin{itemize}
\item \textit{(Halving process did not stop before $M$)} We have that $\min \{ | \nu |, | \zeta |, | \xi | \} < 4CM$ and $ | \theta_i - \theta_j | < 2M$ for all $i,j$.
\item \textit{(Halving process stops at $M$)} We have that either $\min \{ | \nu |, | \zeta |, | \xi | \} \geq 2CM$ or $| \theta_i - \theta_j | \geq M$ for some $i, j$. 
\end{itemize}
We say the interaction is of first kind if $\min \{ | \nu |, | \zeta |, | \xi | \} \geq 2CM$ and of second kind otherwise.\footnote{In particular, note that if both stopping conditions are satisfied simultaneously, that is, $\min \{ | \nu |, | \zeta |, | \xi | \} \geq 2CM$ and $| \theta_i - \theta_j | \geq M$, we consider the interaction to be of first kind.} In the extremal case of the biggest possible $M$,  ($\frac{N_{\rm{min}}}{C_2}$-interactions) we do not require the first item. Moreover, we also define any element of $S_0 \setminus S_1$ to be an $\frac{N_{\rm{min}}}{C_2}$-interaction of second kind. 

We let $S_{M} \subset S_0$ denote the set of $M$-interactions of $S_0$. We let $S_{M}^1$ and $S_{M}^2$ denote the sets of $M$-interactions of first and second kind respectively. 
\end{definition}

As we will see afterwards, these $M$-interactions cover the whole set $S_0$ since the cases where $\min \{ | \nu |, | \zeta |, | \zeta | \} < 2C\beta$ and $|\theta_i - \theta_j| < \beta$ will correspond precisely to $S_{\rm{bad}}$ (the set of resonant frequencies that is treated differently). First, let us note the following Remark and Corollary.

\begin{remark} Note that for interactions of first kind we can apply Lemma \ref{cor:coro}, since $\min \{ | \nu |, | \zeta |, | \xi | \} \geq 2CM$ and $| \theta_i - \theta_j | < 2M$. Therefore, we have \eqref{eq:nuez}. For interactions of second kind we will apply our bilinear estimates from Proposition \ref{prop:refined}, as a consequence of our next corolalry.
\end{remark}


\begin{corollary} \label{cor:bigbox} Suppose that $\sigma$ is a $\frac{N_{\rm{min}}}{C_2}$-interaction or that it satisfies the first condition of Definition \ref{def:Minteraction}. In particular, this assumption covers the case $\sigma\in S_{M}$ for all dyadic $M$ with $\beta \leq M \leq \frac{N_{\rm{min}}}{C_2}$. Then we have
\begin{equation*}
|\nu | = | (\nu, k) - (0, k) | \les M, \qquad |(\zeta, m) - (-k/2, -k/2) | \les M, \qquad | (\xi, n) - (k/2, -k/2) | \les M,
\end{equation*}
up to some rearrangement of the triplets $(\nu, k)$, $(\zeta, m)$ and $(\xi, n)$. 
\end{corollary}
\begin{proof}
First, let us note that the result is trivial for $\frac{N_{\rm{min}}}{C_2}$-interactions, since all the frequencies are bounded by $8N_{\rm{min}}$ and $M \approx N_{\rm{min}}$ in this case. Thus, let us suppose that the first condition in Definition \ref{def:Minteraction} is satisfied.

Without loss of generality, assume that $| \nu | = \min \{ | \nu |,  | \zeta |,  | \xi | \} \les M$. Now, note that Lemma \ref{lemma:localization} can be applied taking $w_i$ to be $(\sqrt{3} \nu, k)$, $(\sqrt{3} \zeta, m)$, $(\sqrt{3} \xi, n)$. We obtain that
\begin{equation*}
\left| (\zeta, m) - (-k/2, -k/2 )\right| \les \left| (\sqrt{3} \zeta, m) - R_{-2\pi/3} (0, k) \right| \les \left| (\sqrt{3} \zeta, m) - R_{-2\pi/3} (\sqrt{3} \nu, k) \right| + | \nu | \les M
\end{equation*}
where in the last inequality we used the consequence of Lemma \ref{lemma:localization}. We also assumed a sign $-$ on $-2\pi/3$. We would obtain the same bound with $+2\pi/3$ rearranging $(\sqrt{3} \zeta, m)$ and $(\sqrt{3} \xi, n)$. An analogous reasoning gives the bound for $(\xi, n)$. \end{proof}

Now, we fix $\beta$ sufficiently small so that all $S_0$ is covered by the $M$-interactions from Definition \ref{def:Minteraction}. The only case that is not covered in Definition \ref{def:Minteraction}, is the case where $\min \{ |\nu |, |\zeta |, |\xi | \} < 2C \beta$ and $|\theta_i - \theta_j | < \beta$. Those cases are covered by the hypothesis of Corollary \ref{cor:bigbox} with $M = \beta/2$. Therefore, there exists some small constant $\beta$ such that, for any $\sigma \in S_0$ that is not an $M$-interaction, we have:
\begin{equation} \label{eq:ravioli}
|\nu | \leq \frac{1}{10}, \qquad |\zeta + k/2 | \leq \frac{1}{10}, \qquad  | \xi - k/2 | \leq \frac{1}{10}, \qquad \mbox{ and } \qquad m = n = -k/2,
\end{equation}
up to rearranging the pairs $(\nu, k)$, $(\zeta, m)$ and $(\xi, n)$. However, this implies that $\sigma \in S_{\rm{bad}}$, so it is not in $S_0$ (which is defined as a subset of $S_{\rm{good}}$). From now on, we just fix $\beta$ so that \eqref{eq:ravioli} holds and we have that 
\begin{equation}
S_0 = \bigcup_{M} S_{M}
\end{equation}
where $M$ is assumed to run over powers of $2$ between $\beta$ and $\frac{N_{\rm{min}}}{C_2}$

\subsubsection{Division of frequency}

The purpose of this section is to divide the set of $M$-interactions, $S_M$, in products of boxes for which ones we can either apply Proposition \ref{prop:refined} or every point satisfies the hypothesis of Lemma \ref{cor:coro}. From Corollary \ref{cor:bigbox}, we know that any $\sigma \in S_M$ satisfies that the three elements $(\nu, k)$, $(\zeta, m)$ and $(\xi, n)$ are all at distance $\les M$ from one of the lines $(0, t)$, $(-t/2, -t/2)$ or $(t/2, -t/2)$ (parametrized by $t$). Those pairs also have norm between $N_{\rm{min}}$ and $8N_{\rm{min}}$. Let us call $\mathcal R_M$ to that neighbourhood of those three lines, restricted to the circular annulus defined by having norm between $N_{\rm{min}}$ and $8N_{\rm{min}}$. Thus, $(\nu, k)$, $(\zeta, m)$, $(\xi, n)$ are all in $\mathcal R_M$. We depict $\mathcal R_M$ in Figure \ref{fig:RM}.

\begin{figure}[htbp]
    \centering
    \begin{tikzpicture}[
    scale=1.2,
    strip/.style={
        fill=blue!15,
        opacity=0.7
    }
]
    \begin{scope}
        \clip (0,0) circle (4);
        
        \fill[strip] (-0.3,-4.5) rectangle (0.3,4.5);
        \fill[strip] ({-4.5*cos(45)-0.3*cos(-45)},{-4.5*sin(45)-0.3*sin(-45)}) -- 
                    ({4.5*cos(45)-0.3*cos(-45)},{4.5*sin(45)-0.3*sin(-45)}) --
                    ({4.5*cos(45)+0.3*cos(-45)},{4.5*sin(45)+0.3*sin(-45)}) --
                    ({-4.5*cos(45)+0.3*cos(-45)},{-4.5*sin(45)+0.3*sin(-45)}) -- cycle;
        \fill[strip] ({-4.5*cos(135)-0.3*cos(-135)},{-4.5*sin(135)-0.3*sin(-135)}) -- 
                    ({4.5*cos(135)-0.3*cos(-135)},{4.5*sin(135)-0.3*sin(-135)}) --
                    ({4.5*cos(135)+0.3*cos(-135)},{4.5*sin(135)+0.3*sin(-135)}) --
                    ({-4.5*cos(135)+0.3*cos(-135)},{-4.5*sin(135)+0.3*sin(-135)}) -- cycle;
    \end{scope}
    
    \fill[white] (0,0) circle (2);
    
    \foreach \x in {-0.4,-0.2,0,0.2} {
        \foreach \y in {-4.1,-3.9,-3.7,-3.5,-3.3,-3.1,-2.9,-2.7,-2.5,-2.3,-2.1} {
            \draw[gray, opacity=0.3] (\x,\y) rectangle (\x+0.2,\y+0.2);
        }
    }
    
    \draw[dashed] (0,-4) -- (0,4);
    \draw[dashed] ({-4*cos(45)},{-4*sin(45)}) -- ({4*cos(45)},{4*sin(45)});
    \draw[dashed] ({-4*cos(135)},{-4*sin(135)}) -- ({4*cos(135)},{4*sin(135)});
    
    \draw[dashed] (0,0) circle (2);
    \draw[dashed] (0,0) circle (4);
    
    \draw[<->,line width=1pt] (-0.1,0) -- (2,0) node[midway, above] {$N_{\text{min}}$};
    
    \draw[->,line width=1pt] (0,0) -- ({4*cos(45)},{4*sin(45)});
    \draw[<-,line width=1pt] (0,0) -- ({0.05*cos(45)},{0.05*sin(45)});
    \node[right] at ({2.5*cos(45)},{2.5*sin(45)}) {$8N_{\text{min}}$};
    
    \draw[|<->|,line width=1pt] (0.3,2.5) -- (-0.3,2.5) node[midway,above] {$O(M)$};
    
    \node[above] at (0,4) {$(0,t)$};
    \node[above] at (-2.5,2.5) {$(-t/2,-t/2)$};
    \node[above=0.3cm] at (2.5,2.5) {$(t/2,-t/2)$};
    
    \node[blue, below right] at (2,-2) {$\mathcal{R}_M$};
    
\end{tikzpicture}  
    \caption{The region $\mathcal{R}_M$ (shaded) consists of $O(M)$-neighborhoods of three lines $(0,t)$, $(-t/2,-t/2)$, and $(t/2,-t/2)$, restricted to the annulus between radii $N_{\rm{min}}$ and $8N_{\rm{min}}$. We also have covered $\mathcal R_M$ with squares of size $cM$. In order to improve the clarity of the picture, it is not done at scale, and the covering in squares of length $cM$ is only done for one of the six connected components of $\mathcal R_M$}
    \label{fig:RM}
\end{figure}
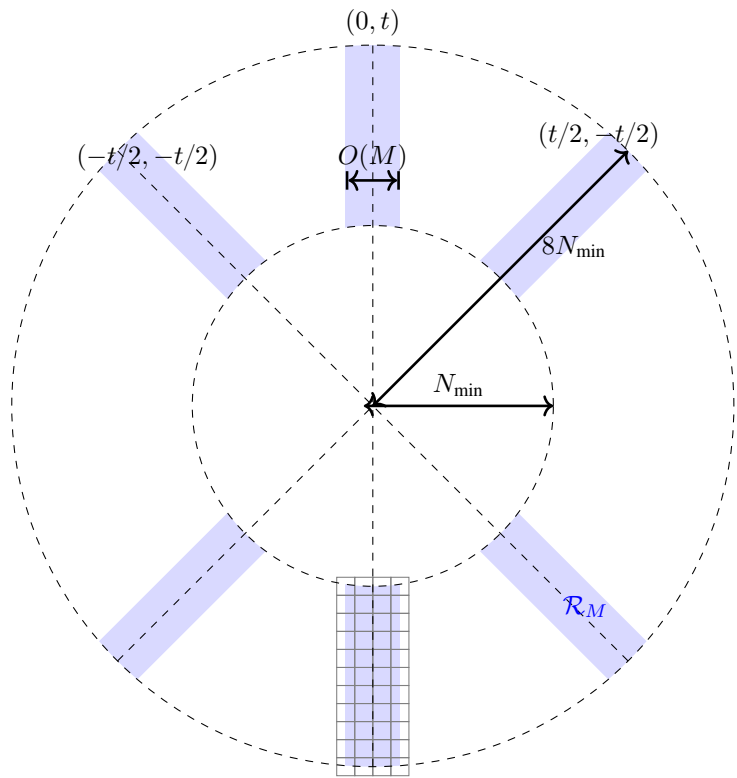

We can cover $\mathcal R_M$ with a family of $O(N_{\rm{min}}/M)$ squares of size $cM\times cM$ that we call $\mathcal B$ organized with a grid-like structure, so that any point is just in one square (up to a measure-zero set). Here $c$ is the small constant from Proposition \ref{prop:refined}. We construct a covering of $S_M$ as follows. For any triplet $B^{(1)}, B^{(2)}, B^{(3)} \in \mathcal B$, we check if there is any $\sigma \in S_M$ with $(\nu, k) \in B^{(1)}$, $(\zeta, m) \in B^{(2)}$ and $(\xi, n) \in B^{(3)}$. If that is the case, we include $(B^{(1)}, B^{(2)}, B^{(3)} )$ in our covering, otherwise not. That way we construct the covering:
\begin{equation} \label{eq:covering}
S_M = \bigcup_{\alpha \in \Upsilon_M} B_\alpha^{(1)} \times B_\alpha^{(2)} \times B_\alpha^{(3)} \times \mathcal S_t.
\end{equation}
Here $\mathcal S_t$ refers to the triples $(\mu, \eta, \tau)$ with $\mu + \eta + \tau = 0$, since we recall that the tuples of $S_M \subset S_0$ have also temporal frequencies as their last components. However, the temporal frequencies will play no role in our decomposition.

We claim that for any box $B^{(1)}_\alpha$ there are finitely many options for $B_\alpha^{(2)} \times B_\alpha^{(3)}$ so that  $B_\alpha^{(1)}\times B_\alpha^{(2)} \times B_\alpha^{(3)} \times \mathcal S_t$ is in the covering. Indeed, let us call $C$ the implicit constant in Corollary \ref{cor:bigbox}. Once $B_\alpha^{(1)}$ is fixed we have that $k$ is in some interval of size $cM$ determined by $B_\alpha^{(1)}$. Call that interval $I_k$. Then, by Corollary \ref{cor:bigbox}, $\zeta, \xi, m, n$ are all on some interval of size $2CM$ around $\pm k/2$, which itself is in $I_k/2$. Thus, every coordinate of $\zeta \in I_\zeta$, $\xi \in I_\xi$, $m \in I_m$, $n\in I_n$, where $I_\zeta, I_\xi, I_m, I_n$ are intervals of size $2CM+cM$ which depend only on $B^{(1)}_\alpha$. Therefore, $I_\zeta \times I_m$ or $I_\xi \times I_n$ are covered by some $\lceil 2C/c + 1\rceil \times \lceil 2C/c+1 \rceil$ square of $cM \times cM$ boxes, meaning there is some absolute constant $C_3 =\lceil 2C/c+1 \rceil^4$ bounding the quantity of pairs of boxes $B^{(2)} \times B^{(3)}$ that have a non-empty intersection with $I_\zeta \times I_m \times I_\xi \times I_n$. Thus, for any box $B^{(1)} \in \mathcal B$, the quantity of times $B^{(1)}$ appears in the covering \eqref{eq:covering} is bounded by the absolute constant $C_3$. By an analogous reasoning, the same happens with $B^{(2)}$ or $B^{(3)}$. 

Lastly, we classify each set of the partition as follows. If there exists some $\sigma \in B^{(1)}_\alpha \times B^{(2)}_\alpha \times B^{(3)}_\alpha \times \mathcal S_t \cap S_0$ such that $|\theta_i - \theta_j | \geq M$, we say that set is an interaction of second-kind and we will be in the hypothesis of Proposition \ref{prop:refined}. Otherwise, it is an interaction of first kind. Note that for all $\sigma \in S_0 \setminus S_1$ we have that $|\theta_i - \theta_j | \geq \frac{N_{\rm{min}}}{1000}$ for some $i,j$, so we see that all the boxes intersecting $S_0 \setminus S_1$ are of second kind. Thus, for interactions of first kind, we have that $ B^{(1)}_\alpha \times B^{(2)}_\alpha \times B^{(3)}_\alpha \times \mathcal S_t \cap S_0 \subset S_1$, and we can apply Lemma \ref{cor:coro}.

All in all, we have partitioned $S_0$ so that: \begin{itemize}
\item $S_0$ is covered by $S_M$ for $M$ being a dyadic number between $\beta = O(1)$ and $\frac{N_{\rm{min}}}{C_2} = O(N_{\rm{min}})$
\item $S_M$ is covered by $O(N_{\rm{min}}/M)$ sets of the type $B_\alpha^{(1)} \times B_\alpha^{(2)} \times B_\alpha^{(3)} \times \mathcal S_t$.
\item Each of the boxes $B^{(i)}_\alpha$ appears less than $C_3$ times in the $M$ covering.
\item Each of the sets $B_\alpha^{(1)} \times B_\alpha^{(2)} \times B_\alpha^{(3)} \times \mathcal S_t$ of the partition is either of type 1 or type 2. If it is of type 2 then it satisfies the hypothesis of Proposition \ref{prop:refined} after possibly rearranging $\hat u (\nu, k, \mu)$, $\hat v (\zeta, m, \eta )$, $\hat w (\xi, n, \tau )$. In the type 1 case, it must happen that all the points satisfy $\Delta \gtrsim N_{\rm{min}}^2 M$ for every $\sigma \in B_\alpha^{(1)} \times B_\alpha^{(2)} \times B_\alpha^{(3)} \times \mathcal S_t \cap S_0$ (by Lemma \ref{cor:coro}). 
\end{itemize}

\subsection{Case \ref{case:Neq}} \label{subsec:det4}

Now let us show that instead of showing \eqref{eq:maindet11}, it suffices to show 
\begin{align} \begin{split} \label{eq:maindet12}
\frac{N_{\rm{min}}^{1/2} }{ L_1^{\frac{1+\delta}{2} } L_2^{\frac{1+\delta}{2} } L_3^{ \frac12 - 3\delta}} &\int_{S_\mathrm{good}} \left| \widehat{P_{B^{(1)}_\alpha} Q_{L_1} u} (\nu, k, \mu) \widehat{ P_{B^{(2)}_\alpha}  Q_{L_2} v} (\zeta, m, \eta)  \widehat{ P_{B^{(3)}_\alpha}  Q_{L_3} w} (\xi, m, \tau) \right| d\sigma \\
& \les_{s, \delta, \eps} N_{\rm{min}}^{\eps / 2} \| u \|_{L^2} \| v \|_{L^2} \| w \|_{L^2},
\end{split} \end{align}
for any dyadic $M$ between $\beta$ and $\frac{N_{\rm{min}}}{C_2}$ and any $\alpha \in \Upsilon_M$. Assume \eqref{eq:maindet12}. Clearly, we can put $P_{B_\alpha^{(i)}}$ projections also on the right hand side, since substituting $u$ with $P_{B^{(1)}_\alpha} u$ does not change the left hand side. Using also that $S_M$ is covered by $B_\alpha^{(1)} \times B_\alpha^{(2)} \times B_\alpha^{(3)} \times \mathcal S_t$ for $\alpha \in \Upsilon_M$, and adding up we get that
\begin{align*}
\frac{N_{\rm{min}}^{1-s}}{ L_1^{\frac{1+\delta}{2} } L_2^{\frac{1+\delta}{2} } L_3^{ \frac12 - 3\delta}}& \int_{S_M} \left| \widehat{ Q_{L_1} u} (\nu, k, \mu) \widehat{   Q_{L_2} v} (\zeta, m, \eta)  \widehat{  Q_{L_3} w} (\xi, m, \tau) \right| d\sigma \les_{s, \delta} \frac{N_{\rm{min}}^{\eps / 2}}{N_{\rm{min}}^{s-1/2}} \sum_{\alpha \in \Upsilon_M} \| P_{B^{(1)}_\alpha}  u \|_{L^2} \| P_{B^{(2)}_\alpha} v \|_{L^2} \|P_{B^{(3)}_\alpha}  w \|_{L^2} \\ 
&\les \frac{ N_{\rm{min}}^{\eps/2} }{N_{\rm{min}}^{s-1/2}} \| u \|_{L^2} \left( \sum_{\alpha \in \Upsilon_M } \| P_{B^{(2)}_\alpha} v \|_{L^2}^2 \right)^{1/2}  \left( \sum_{\alpha \in \Upsilon_M } \|P_{B^{(3)}_\alpha}  w \|_{L^2}^2  \right)^{1/2}\les \frac{ N_{\rm{min}}^{\eps/2} }{N_{\rm{min}}^{s-1/2}} \| u \|_{L^2} \| v \|_{L^2} \| w \|_{L^2}, 
\end{align*}
where in the last inequality we used that $B_\alpha^{(i)}$ covers at most $10C_3$ times each point, as $\alpha \in \Upsilon_M$. Lastly, adding up over the $O(\log N_{\rm{min}})$ possibilities for $M$ between $\beta$ and $\frac{N_{\rm{min}}}{C_2}$, we obtain that 
\begin{equation} 
\frac{N_{\rm{min}}^{1-s}}{ L_1^{\frac{1+\delta}{2} } L_2^{\frac{1+\delta}{2} } L_3^{ \frac12 - 3\delta}} \int_{S_0} \left| \widehat{ Q_{L_1} u} (\nu, k, \mu) \widehat{  Q_{L_2} v} (\zeta, m, \eta)  \widehat{   Q_{L_3} w} (\xi, m, \tau) \right| d\sigma \les_{s, \delta, \eps} \frac{N_{\rm{min}}^{\eps}}{N_{\rm{min}}^{s-1/2}} \| u \|_{L^2} \| v \|_{L^2} \| w \|_{L^2},
\end{equation}
which implies \eqref{eq:maindet11} in case \ref{case:Neq}, since $\sigma$ corresponding to that case has to be in $S_0$. Thus, we have reduced case \ref{case:Neq} to show \eqref{eq:maindet12}.

\subsubsection{Interactions of first kind}
Recall that in interactions of first kind we have that any point of $B^{ (1) }_\alpha \times B^{(2)}_\alpha \times B^{(3)}_\alpha \times \mathcal S_t \cap S_0$ satisfies the hypothesis of Lemma \ref{cor:coro} and therefore $|\Delta| \gtrsim N_{\rm{min}}^2 M$. Note moreover that
\begin{equation*}
|\Delta | = |\Delta + \mu + \eta + \tau | \leq | \mu - \phi(\nu, k) | + | \eta - \phi (\zeta, m) | + | \tau - \phi (\xi, n) | \leq 2L_1 + 2L_2 + 2L_3 \leq 6L_{\rm{max}},
\end{equation*}
and we obtain that $L_{\rm{max}} \gtrsim N_{\rm{min}}^2M$. 

Undoing Fourier Transform (recall $\tilde u = \mathcal F^{-1} | \hat u |$) and using \eqref{eq:refined1} from Proposition \ref{prop:refined}, we get
\begin{align} \nonumber
\Xi &:=\frac{N_{\rm{min}}^{1/2}}{ L_1^{\frac{1+\delta}{2} } L_2^{\frac{1+\delta}{2} } L_3^{ \frac12 - 3\delta}} \int_{S_\mathrm{good}} \left| \widehat{P_{B^{(1)}_\alpha} Q_{L_1} u} (\nu, k, \mu) \widehat{ P_{B^{(2)}_\alpha}  Q_{L_2} v} (\zeta, m, \eta)  \widehat{ P_{B^{(3)}_\alpha}  Q_{L_3} w} (\xi, m, \tau) \right| d\sigma  \\
 &= \frac{N_{\rm{min}}^{1/2}}{ L_1^{\frac{1+\delta}{2} } L_2^{\frac{1+\delta}{2} } L_3^{ \frac12 - 3\delta}}  \int_{\mathbb R} \int_{\mathbb R \times \mathbb T}  P^1_\alpha Q_{L_1} \tilde u (x, t)  P^2_\alpha Q_{L_2} \tilde v (x, t) P^3_\alpha Q_{L_3} \tilde w(x, t)  dx dt \nonumber \\ 
 &\les \frac{N_{\rm{min}}^{1/2} L_{\rm{max}}^{3\delta} }{ L_{\rm{min}}^{\frac{1}{2} } L_{\rm{med}}^{\frac{1}{2} } L_{\rm{max}}^{ \frac12}} M \mathrm{min} \{ L_1, L_2 \}^{1/2} \| u \|_{L^2}\| v \|_{L^2}\| w \|_{L^2}. \label{eq:sandia}
\end{align}
Now, in \eqref{eq:sandia}, we use that $L_{\rm{max}} \gtrsim N_{\rm{min}}^2 M$ and that $\min \{ L_1, L_2 \} \leq L_{\rm{med}}$, to obtain
\begin{align*}
\Xi &\les \frac{N_{\rm{min}}^{1/2} L_{\rm{max}}^{3\delta} }{ M^{1/2} N_{\rm{min}} }  M \| u \|_{L^2}\| v \|_{L^2}\| w \|_{L^2} \les (M/N_{\rm{min}})^{1/2} \cdot N_{\rm{min}}^{ 9 \delta} \| u \|_{L^2}\| v \|_{L^2}\| w \|_{L^2} \\
&\les N_{\rm{min}}^{\eps /2}  \| u \|_{L^2}\| v \|_{L^2}\| w \|_{L^2} ,
\end{align*}
where used $M \les N_{\rm{min}}$, $L_{\rm{max}} \les N_{\rm{min}}^3$ and $\delta \ll \eps$. This concludes the proof of \eqref{eq:maindet12} for interactions of first kind.

\subsubsection{Interactions of second kind}
For interactions of second kind, we can use \eqref{eq:refined2} from Proposition \ref{prop:refined} for some pair of the frequencies (the one such that $|\theta_i - \theta_j | \geq M$). Undoing Fourier Transform (recall $\tilde u = \mathcal F^{-1} | \hat u |$) and using \eqref{eq:refined1}, we get
\begin{align*} \nonumber
\Xi &:=\frac{N_{\rm{min}}^{1/2}}{ L_1^{\frac{1+\delta}{2} } L_2^{\frac{1+\delta}{2} } L_3^{ \frac12 - 3\delta}} \int_{S_\mathrm{good}} \left| \widehat{P_{B^{(1)}_\alpha} Q_{L_1} u} (\nu, k, \mu) \widehat{ P_{B^{(2)}_\alpha}  Q_{L_2} v} (\zeta, m, \eta)  \widehat{ P_{B^{(3)}_\alpha}  Q_{L_3} w} (\xi, m, \tau) \right| d\sigma  \\
 &= \frac{N_{\rm{min}}^{1/2}}{ L_1^{\frac{1+\delta}{2} } L_2^{\frac{1+\delta}{2} } L_3^{ \frac12 - 3\delta}}  \int_{\mathbb R} \int_{\mathbb R \times \mathbb T}  P^1_\alpha Q_{L_1} \tilde u (x, t)  P^2_\alpha Q_{L_2} \tilde v (x, t) P^3_\alpha Q_{L_3} \tilde w(x, t)  dx dt \nonumber \\ 
 &\les \frac{N_{\rm{min}}^{1/2} L_{\rm{max}}^{3\delta} }{ L_{1}^{\frac{1}{2} } L_{2}^{\frac{1}{2} } L_{3}^{ \frac12}}  \frac{L_i^{1/2} L_j^{1/2}}{N_{\rm{min}}^{1/2}} \| u \|_{L^2}\| v \|_{L^2}\| w \|_{L^2} \nonumber \\ 
  &\les  N_{\rm{min}}^{9\delta}   \| u \|_{L^2}\| v \|_{L^2}\| w \|_{L^2}.
\end{align*}
This proves \eqref{eq:maindet12} in the case of interactions of second kind, since $\delta \ll \eps$. Thus, we the conclude the proof of \eqref{eq:maindet11} for case \ref{case:Neq}, and therefore we have proven \eqref{eq:maindet1} from Proposition \ref{prop:maindet}.

\subsection{Bad interactions} \label{subsec:det5}
The objective of this section is to deal with the bad interactions, showing \eqref{eq:maindet2} and \eqref{eq:maindet3} in Proposition \ref{prop:maindet}. Let us recall that $S_{\rm{bad}}$ is the region where
\begin{align} \label{eq:ravioli2}
|\nu | &\leq \frac{1}{10}, \qquad |\zeta + k/2 |,  | \xi - k/2 | \leq \frac{1}{10}, \qquad \mbox{ and } \qquad m = n = -k/2, \\
&\mbox{ and } \qquad | \mu - \phi(\nu, k_2)|, | \eta - \phi(\zeta, m_2) |, | \tau - \phi(\xi, n_2) | \leq 2048|k|^3 \notag
\end{align}
up to rearrangement of $(\nu, k)$, $(\zeta, m)$ and $(\xi, n)$.

Instead of showing \eqref{eq:maindet2}--\eqref{eq:maindet3}, recall that we can show its frequency localized versions \eqref{eq:maindet21}--\eqref{eq:maindet31} from Remark \ref{rem:freq_loc}. Let us note that in $S_{\rm{bad}}$ we have that $\frac{ \langle \xi \rangle }{\langle  \nu \rangle \cdot \langle \zeta \rangle } \lesssim 1$, $N_i \approx N_{\rm{min}} \approx N$ and $L_\mathrm{max} \leq 2048N_{\rm{min}}^3 \les N^3 $. Using that, we see that in order to show \eqref{eq:maindet21} it suffices to show:
\begin{align}\label{eq:maindet22}
& \frac{N^{-s + 9\delta} }{ L_1^{\frac12} L_2^{\frac12}L_3^{\frac12} } \int_{S_\mathrm{bad}} \left|  |\nu|^{1/2} | \zeta |^{1/2} | \xi |^{1/2}  \widehat{ P_{N_1} Q_{L_1} u }(\nu, k, \mu) \widehat{ P_{N_2} Q_{L_2} v} (\zeta, m, \eta) \widehat{ P_{N_3} Q_{L_3} w}(\xi, n, \tau) d\sigma \right|  \\
 \qquad &\les \frac{N^{\eps}}{N^{s-1/2}} \| u  \|_{L^2} \| u \|_{L^2} \| w \|_{L^2},
\end{align}
for all $s > 1/2$ and $\delta > 0$ sufficiently small. Similarly, instead of showing \eqref{eq:maindet31}, it suffices to show
\begin{equation} \label{eq:maindet32}
  \frac{N^{1-s + 9\delta} }{ L_1^{\frac12} L_2^{\frac12}L_3^{\frac12} }  \int_{S_\mathrm{bad}} \left| \widehat{ P_{N_1} Q_{L_1} u }(\nu, k, \mu) \widehat{ P_{N_2} Q_{L_2} v} (\zeta, m, \eta) \widehat{ P_{N_3} Q_{L_3} w}(\xi, n, \tau)   d\sigma \right|  \les \frac{N^\eps}{N^{s-1/2}}\| u  \|_{L^2} \| v \|_{L^2} \| w \|_{L^2},
\end{equation}
for $s > 3/4$ and $\delta > 0$ sufficiently small.

Now, since both estimates \eqref{eq:maindet22}--\eqref{eq:maindet32} are symmetric with respect to rearrangement of $u, v, w$, and the region $S_{\mathrm{bad}}$ is also symmetric with respect to rearrangement of $(\nu, k)$, $(\zeta, m)$ and $(\xi, n)$, we can assume without loss of generality that we are exactly in the situation of \eqref{eq:ravioli2}, with no rearrangement needed there. Using that, we can further reduce \eqref{eq:maindet22}--\eqref{eq:maindet32} to show
\begin{equation} \label{eq:ohcapitan}
\frac{N^{1/2}}{ L_1^{\frac12} L_2^{\frac12}L_3^{\frac12} } \mathcal I_i \les N^{\eps / 2}\| u  \|_{L^2} \| v \|_{L^2} \| w \|_{L^2},
\end{equation}
where 
\begin{align*}
\mathcal I_i &= \int_{\mathbb{R}^2} d\mu d\eta \int_{-\frac{1}{10}}^{\frac{1}{10}}  \int_{-\frac{1}{10}}^{\frac{1}{10}} d\nu d\omega \sum_{N_{\rm{min}} \leq k < 8N_{\rm{min}}} \\
&\qquad |\nu|^{\gamma_i} \left| \widehat{Q_{L_1} u}(\nu, k, \mu) \widehat{Q_{L_2} v} (\omega -k/2, -k/2, \eta) \widehat{Q_{L_3} w}(-\nu - \omega + k/2, -k/2, -\mu-\eta) \right|,
\end{align*}
and $\gamma_i = 1/2$ for $i=1$ (corresponding to \eqref{eq:maindet22}) and $\gamma_i = 0$ for $i=2$ (corresponding to \eqref{eq:maindet32}). Here, we are using the change of variables $\omega - k/2 = \zeta$, so that $|\omega | \leq \frac{1}{10}$ by \eqref{eq:ravioli2}. We use Cauchy-Schwarz in the $\mathbb{R}^2$ integral and obtain that 
\begin{align*}
\mathcal I_i &\les \int_{-\frac{1}{10}}^{\frac{1}{10}}  \int_{-\frac{1}{10}}^{\frac{1}{10}}  \sum_{N_{\rm{min}} \leq k < 8N_{\rm{min}}} 
\Bigg[
\left( \int_{\mathbb R^2}  |\widehat{Q_{L_2} v} (\omega -k/2, -k/2, \eta) |^2  | \widehat{Q_{L_3} w}(-\nu - \omega + k/2, -k/2, -\mu-\eta)  |^2  d\mu  d\eta\right)^{1/2} \\ 
&\quad \left( \int_{\mathbb R^2} |\widehat{Q_{L_1} u}(\nu, k, \mu) |^2   \mathbbm{1}_{| \eta - \phi ( \omega -k/2, -k/2 ) | \leq 2L_2} \mathbbm{1}_{ | \eta + \mu + \phi ( -\nu - \omega + k/2, -k/2 ) | \leq 2L_3}d\mu  d\eta \right)^{1/2}  |\nu|^{\gamma_i} d\nu d\omega
\Bigg]
\end{align*}
where we used the restrictions coming from $|\eta - \phi(\zeta,m) | \leq 2L_2$ and $|\tau - \phi (\xi, n) | \leq 2L_3$. We see that in the last integral, for every $\mu \in \mathbb R$, we have that $\eta$ ranges over an interval of size smaller than $\min \{ 4L_2, 4L_3 \} \leq 4L_{\rm{med}}$. Therefore, we have
\begin{equation} \label{eq:castilloif}
\mathcal I_i \les \int_{-\frac{1}{10}}^{\frac{1}{10}}  \int_{-\frac{1}{10}}^{\frac{1}{10}}  \sum_{N_{\rm{min}} \leq k < 8N_{\rm{min}}} L_{\rm{med}}^{1/2} |\nu|^{\gamma_i} \| \hat u(\nu, k, \mu) \|_{L^2_\mu} \| \hat v (\omega-k/2, -k/2, \eta ) \|_{L^2_\eta} \| \hat w (-\nu-\omega + k/2, -k/2, \tau )\|_{L^2_\tau}  d\nu d\omega
\end{equation}

We compute $\Delta$ in $k, \nu, \omega$ variables for elements of $S_\mathrm{bad}$:
\begin{align*}
\Delta &= \nu (\nu^2 + k^2) + \zeta ( \zeta^2 + (k/2)^2) + \xi (\xi^2 + (k/2)^2) \\ 
&= \nu^3 + \zeta^3 - (\nu + \zeta)^3 + \nu k^2 + \zeta (k/2)^2 + \xi (k/2)^2 \\
&= -3 \nu \zeta (\nu + \zeta) + \frac34 \nu k^2 = -3 \nu (-k/2 + \omega) (\nu - k/2 + \omega) + \frac34 \nu k^2 \\
&= -\frac34 \nu k^2 - 3 \nu \omega (-k/2) - 3 \nu (-k/2) (\nu + \omega) + \frac34 \nu k^2 + O(| \nu |^3 + | \omega |^3 ) \\
&=  3 \nu \omega k + \frac32 k \nu^2 + O(| \nu |^3 + | \omega |^3 ).
\end{align*}
Given that $|\nu|, |\omega| \les 1$ and $\langle \Delta \rangle \approx \langle \Delta + 1 \rangle$, we get that
\begin{equation} \label{eq:cacahuete}
\langle \Delta \rangle \approx  \langle k \nu \left( \omega + \frac{\nu}{2} \right) \rangle .
\end{equation} 

Now, we treat $\mathcal I_i$ in the two cases $i=1$, $i=2$; corresponding to the proof of \eqref{eq:maindet2} and \eqref{eq:maindet3}.

\subsubsection{Bounds on $\mathcal I_1$}

We do dyadic partition of the square $[-1/10, 1/10]^2$ again with respect to $( \omega + \frac{\nu}{2} )$, truncated at $\Gamma = N^{-1}$. Formally, let $\Gamma$ be a power of two (of negative exponent) with $N^{-1} \leq \Gamma \leq 1$ and such that $\Gamma \leq  \left| \omega + \frac{\nu}{2} \right| < 2 \Gamma$, or in the case $\Gamma = N^{-1}$ simply $\left| \omega + \frac{\nu}{2} \right| < 2N^{-1}$. Let $A_\Gamma$ be the points $(\nu, \omega ) \in [-1/10, 1/10]^2$ with that associated value of $\Gamma$. From \eqref{eq:cacahuete}, we get that
\begin{equation} \label{eq:cacahuete2}
\langle \Delta \rangle \gtrsim N |\nu| \Gamma.
\end{equation}
This clearly holds for $\Gamma > N^{-1}$ because $\Gamma \approx  \left| \omega + \frac{\nu}{2} \right|$, but note it also holds $\Gamma = N^{-1}$ because the right hand side above is smaller than $1$. Combining \eqref{eq:cacahuete2} with
\begin{equation} \label{eq:cacahuete3}
\Delta = \Delta - \mu - \eta - \tau \leq | \phi(\nu, k) -\mu| + |\phi(\zeta, m) - \eta | + |\phi(\xi, n) - \tau | \leq 6 L_{\rm{max}}
\end{equation}
we deduce
\begin{equation} \label{eq:mar}
|\nu|^{1/2}\les \frac{L_{\rm{max}}^{1/2}}{N^{1/2} \Gamma^{1/2}}.
\end{equation}

Now, we split the integrals $\mathcal I_1 = \sum_{\Gamma} \sum_k \mathcal I_1^{k, \Gamma}$ according to the values of $k$ and $\Gamma$. From \eqref{eq:mar}, we have the bound 
\begin{align}
\mathcal I_1^{k, \Gamma} &\les   \frac{ L_{\rm{max}}^{1/2} L_{\rm{med}}^{1/2} }{N^{1/2} \Gamma^{1/2} } \int_{A_\Gamma} \| \hat u(\nu, k, \mu) \|_{L^2_\mu} \| \hat v (\omega-k/2, -k/2, \eta ) \|_{L^2_\eta} \| \hat w (-\nu-\omega + k/2, -k/2, \tau )\|_{L^2_\tau}  d\nu d\omega \nonumber \\ 
 &\les   \frac{ L_{\rm{max}}^{1/2} L_{\rm{med}}^{1/2} }{N^{1/2} \Gamma^{1/2} } \left( \int_{-\frac{1}{10}}^{-\frac{1}{10}} \| \hat u(\nu, k, \mu) \|_{L^2_\mu} \int_{-\frac{1}{10}}^{-\frac{1}{10}}  \mathbbm{1}_{(\nu, \omega) \in A_\Gamma}   d\omega d\nu \right)^{1/2} \\
 &\qquad \cdot \left( \int_{A_\Gamma} \| \hat v (\omega-k/2, -k/2, \eta ) \|_{L^2_\eta}^2 \| \hat w (-\nu-\omega + k/2, -k/2, \tau )\|_{L^2_\tau}^2  d\nu d\omega \right)^{1/2} \nonumber
\end{align}
Now, observe that for any fixed $\nu \in [-1/10, 1/10]$, the set of $\omega$ satisfying $|\omega + \nu/2| \leq 2\Gamma$ is an interval of size $4\Gamma$. Therefore, the integral of $\mathbbm{1}_{(\nu, \omega) \in A_\Gamma}$ along the $d\omega$ direction is bounded by $4\Gamma$. Therefore, we obtain
\begin{align*}
\mathcal I_1^{k, \Gamma} \les  \frac{ L_{\rm{max}}^{1/2} L_{\rm{med}}^{1/2} }{N^{1/2}} \| \hat u(\nu, k, \mu) \|_{L^2_{\nu, \mu}} \| \hat v (\zeta, -k/2, \eta ) \|_{L^2_{\zeta, \eta}} \| \hat w (\xi, -k/2, \tau )\|_{L^2_{\xi, \tau}}
\end{align*}
Lastly, adding with respect to $\Gamma$ and $k$, and noting that there are $O(\log (N) ) $ possible values of $\Gamma$, we get
\begin{align*}
\mathcal I_1 &= \sum_{k = N}^{8N} \sum_{\Gamma} \mathcal I_1^{k, \Gamma} \les \frac{ L_{\rm{max}}^{1/2} L_{\rm{med}}^{1/2} \log (N)}{N^{1/2}} \sum_{k = N_{\rm{min}}}^{8N_{\rm{min}}}  \| \hat u(\nu, k, \mu) \|_{L^2_{\nu, \mu}} \| \hat v (\zeta, -k/2, \eta ) \|_{L^2_{\zeta, \eta}} \| \hat w (\xi, -k/2, \tau )\|_{L^2_{\xi, \tau}} \\ 
&\les \frac{ L_{\rm{max}}^{1/2} L_{\rm{med}}^{1/2} \log (N)}{N^{1/2}} \left( \sum_{k = N_{\rm{min}}}^{8N_{\rm{min}}}  \| \hat u(\nu, k, \mu) \|_{L^2_{\nu, \mu}}^2 \right)^{1/2} \left( \sum_{k = N_{\rm{min}}}^{8N_{\rm{min}}} \| \hat v (\zeta, -k/2, \eta ) \|_{L^2_{\zeta, \eta}}^2 \right)^{1/2} \max_{N_{\rm{min}}\leq k \leq 8N_{\rm{min}}}\| \hat w (\xi, -k/2, \tau )\|_{L^2_{\xi, \tau}}\\
&\les \frac{ L_{\rm{max}}^{1/2} L_{\rm{med}}^{1/2} \log (N)}{N^{1/2}} \| u \|_{L^2} \| v \|_{L^2} \| w \|_{L^2}
\end{align*}
This shows \eqref{eq:ohcapitan} for $\mathcal I_1$, for $s > 1/2$ and $\delta > 0$ sufficiently small, concluding the proof of \eqref{eq:maindet2}.

\subsubsection{Bounds on $\mathcal I_2$}
The procedure is similar to before, but we need to do a different partition of the region. The reason is that we no longer have the factor $|\nu|^{1/2}$, so the problematic region is the case where $\nu (\omega + \nu/2)$ is small (rather than $\omega + \nu/2$ being small as before). Thus, we let $\Theta$ to be the dyadic value such that $\Theta \leq |\nu (\omega + \nu/2)| \leq 2 \Theta$ truncated at $\Theta = N^{-1}$ (so we only ask for $| \nu (\omega + \nu/2) | \leq 2 N^{-1}$ in the case $\Theta = N^{-1}$). We divide the square $[-1/10, 1/10]$ accordingly in regions $B_\Theta$. We also divide the integral $\mathcal I_2 = \sum_k \sum_\Theta \mathcal I_2^{k, \Theta}$ according to the values of $k$ and $\Theta$.

Let us also note that the measure of $\{ \nu \mbox{ s.t. } (\nu, \omega) \in B_\Theta \}$ is bounded by $10\Theta^{1/2}$ uniformly in $\omega$. The reason is that, if $| \nu (\omega + \nu / 2) | \leq 2\Theta$, it must happen that either $| \nu | \leq (2\Theta)^{1/2}$ or that $|\omega + \nu/2| \leq (2\Theta)^{1/2}$. From \eqref{eq:castilloif}, using Cauchy-Schwarz, we have:
\begin{align}
\mathcal I_2^{k, \Theta} &\les \int_{B_\Theta}   L_{\rm{med}}^{1/2} \| \hat u(\nu, k, \mu) \|_{L^2_\mu} \| \hat v (\omega-k/2, -k/2, \eta ) \|_{L^2_\eta} \| \hat w (-\nu-\omega + k/2, -k/2, \tau )\|_{L^2_\tau}  d\nu d\omega \nonumber \\ 
&\les   L_{\rm{med}}^{1/2} 
\left( \int_{-\frac{1}{10}}^{\frac{1}{10}}\int_{-\frac{1}{10}}^{\frac{1}{10}}  \| \hat v (\omega-k/2, -k/2, \eta ) \|_{L^2_\eta}^2 \mathbbm{1}_{B_\Theta} (\nu, \omega) d\nu d\omega \right)^{1/2} \nonumber \\
&\qquad \cdot \left( \int_{B_\Theta} \| \hat u(\nu, k, \mu) \|_{L^2_\mu}^2 \| \hat w (-\nu-\omega + k/2, -k/2, \tau )\|_{L^2_\tau}^2  d\nu d\omega  \right)^{1/2} \nonumber \\ 
&\les L_{\rm{med}}^{1/2} \Theta^{1/4} \| \hat u (\nu, k, \mu)\|_{L^2_{\nu, \mu}} \| \hat v (\zeta, -k/2, \eta)\|_{L^2_{\zeta, \eta}} \| \hat w (\xi, -k/2, \tau)\|_{L^2_{\xi, \tau}} \label{eq:cala}
\end{align}

Now, from \eqref{eq:cacahuete} that $\langle \Delta \rangle \approx N \Theta$, and using \eqref{eq:cacahuete3}, we get that $\Theta \les \frac{L_{\rm{max}}}{N}$. Using that in \eqref{eq:cala}, together with the fact that there are $O(\log (N))$ possible values for $\Theta$ (takes dyadic values between $N^{-1}$ and $1$) we get
\begin{align*}
\mathcal I_2 &= \sum_{k = N_{\rm{min}}}^{8N_{\rm{min}}} \sum_{\Theta} \mc I_2^{k, \Theta} \les \frac{ L_{\rm{med}}^{1/2} L_{\rm{max}}^{1/4} \log (N) }{N^{1/4}} \sum_{k = N_{\rm{min}}}^{8N_{\rm{min}}} \| \hat u (\nu, k, \mu)\|_{L^2_{\nu, \mu}} \| \hat v (\zeta, -k/2, \eta)\|_{L^2_{\zeta, \eta}} \| \hat w (\xi, -k/2, \tau)\|_{L^2_{\xi, \tau}} \\
&\les \frac{ L_{\rm{med}}^{1/2} L_{\rm{max}}^{1/4} \log (N) }{N^{1/4}} \left( \sum_{k = N_{\rm{min}}}^{8N_{\rm{min}}} \| \hat u (\nu, k, \mu)\|_{L^2_{\nu, \mu}}^2 \right)^{1/2} \left( \sum_{k=N_{\rm{min}}}^{8N_{\rm{min}}} \| \hat v (\zeta, -k/2, \eta)\|_{L^2_{\zeta, \eta}}^2 \right) \max_{N_{\rm{min}} \leq k \leq 8N_{\rm{min}}} \| \hat w (\xi, -k/2, \tau)\|_{L^2_{\xi, \tau}} \\
&\les \frac{ L_{\rm{med}}^{1/2} L_{\rm{max}}^{1/4} \log (N) }{N^{1/4}} \| u \|_{L^2} \| v \|_{L^2} \| w \|_{L^2}.
\end{align*}
We see that \eqref{eq:ohcapitan} is satisfied for $\mathcal I_2$ as long as $s > 3/4$ and $\delta > 0$ is sufficiently small. This concludes the proof of \eqref{eq:maindet3}

\subsection{Proof of Theorem \ref{th:det_fail}} \label{subsec:det_fail} \label{subsec:det6}

The counterexamples to the inequalities in \eqref{eq:one1} are based on the bad interactions that we have seen in the previous subsection. We fix some $N$ large. We will work with $k = N$, $m = n = -N/2$. We will also assume functions that are close to linear solutions, meaning they are supported on $| \mu - \phi_1 | \leq C$ and $| \eta - \phi-2 | \leq C$ for some constant $C$, independent of $N$, to be fixed later. We recall the variable $\omega = \zeta + N/2$, and define $\rho = \xi - N/2$. We will work in the regime $|\nu| , |\omega |, \leq 1$, so that $|\rho | = |-\nu - \omega | \leq 2$. We define  
\begin{equation*}
\hat u (\nu, k, \mu) = \mathbbm{1}_{k = N} \mathbbm{1}_{R_1}(\nu) \mathbbm{1}_{|\mu - \phi_1| \leq C}, \qquad \hat v \left( \frac{-N}{2} + \omega , m, \eta \right) = \mathbbm{1}_{m = -N/2} \mathbbm{1}_{R_2}(\omega)  \mathbbm{1}_{| \eta - \phi_2 | \leq C}.
\end{equation*}
From now on, we assume $k = N$ and $m = -N/2$ and remove those conditions from the definitions of $u, v$. We take the regions $R_1, R_2 \subset [-1, 1]$ to be fixed later. 

We have that
\begin{equation*}
-i \widehat{ \p_x (u v) } \left( \xi, -n, -\tau \right) = -\xi  \int_{R_1}  \mathbbm{1}_{R_2} (-\rho-\nu) \int_{\phi_1 - 1}^{\phi_1 + 1}  \mathbbm{1}_{| -\mu-\tau - \phi_2 | \leq C} d\mu d\nu
\end{equation*}
where $\rho = N/2 - \xi$. We also notice that the right hand side is positive since $\xi = -N/2 -\rho \approx - N/2$. It is clear that when $|\mu - \phi_1| \leq C/2$ and $|-\phi_1 - \phi_2 - \tau| \leq C/2$ the condition on the last integral is satisfied. Therefore, we obtain the lower bound
\begin{align*}
-i \widehat{ \p_x (u v) } (-\rho-N/2, -n, -\tau) &\geq \frac{N}{4}   \int_{R_1}  \mathbbm{1}_{R_2} (-\rho-\nu)  \mathbbm{1}_{|- \tau - \phi_1 - \phi_2 | \leq C} d\nu
\end{align*}
Now, we focus on evaluating the left hand side for $| \tilde \tau | < 1$. In that case, the rightmost indicator function in the integral above is bounded by $\mathbbm{1}_{| \phi_1 + \phi_2 + \phi_3 | \leq C-1}$.  We also recall \eqref{eq:cacahuete}, we have: $\langle \Delta \rangle \approx  \langle  N \nu  ( \omega + \nu/2) \rangle $. We now fix $C$ sufficiently large so that the last condition inside the integral is implied by $ N |\nu|  | \omega + \nu/2| \leq 1$. Therefore, we obtain 
\begin{align*}
-i \widehat{ \p_x (u v) } (-(N/2 + \rho), -n_2, -\tau) &\geq
\frac{N}{4} \mathbbm{1}_{| \tilde \tau | \leq 1} 
\int_{R_1} \mathbbm{1}_{R_2} ( -\rho-\nu) \mathbbm{1}_{| \nu  | \cdot | \rho + \nu / 2| \leq 1/N} d\nu
\end{align*}
In particular, taking $R_1 = [-N^{-1/2}, N^{-1/2}]$ and $R_2 = [-N^{-1/2}, N^{-1/2}]$, we obtain that the integral above is greater than $N^{-1/2}$ for $| \rho | \leq N^{-1/2}$. That yields $\| \p_x (u v) \|_{X^{s, b}} \gtrsim N^{-1/4} N^{s} N^{1/2} = N^{s+1/4}$, while $u, v$ have norms with product $ N^{2s} |R_1|^{1/2} |R_2|^{1/2} = N^{2s-1/2}$. We see that the inequality $N^{s+1/4} \lesssim N^{2s-1/2}$ breaks for $s < 3/4$ by taking $N$ sufficiently large. That concludes the counterexample for $X^{s, b}$ spaces

In the case of $Y^{s, b}$ spaces, we take $R_1 = \frac12 + [-\frac{1}{2N}, \frac{1}{2N}]$ and $R_2 = -\frac14 + [-\frac{1}{N}, \frac{1}{N}]$. That way, $\| u \|_{Y^{s, b}} \| v \|_{Y^{s, b}} \approx N^{2s-1}$. We also see that when $\rho \in \frac14 + [-\frac{1}{2N}, \frac{1}{2N}]$ and $\nu \in R_1$, we have both that $-\rho - \nu \in R_2$ and $|\nu| \cdot |\rho + \nu / 2 | \leq 1$. Therefore, if, we have that
$$
-i \widehat{ \p_x (u v) } (-(N/2 + \rho), -n_2, -\tau) \geq
\frac{1}{4} \mathbbm{1}_{| \tilde \tau | \leq 1}  \mathbbm{1}_{\rho \in \frac14 + [-\frac{1}{2N}, \frac{1}{2N}]}.
$$
This implies $\| \p_x (uv ) \|_{Y^{s, b}} \gtrsim N^{-1/2+s}$. Since $N^{-1/2+s} \les N^{2s-1}$ only holds for $s \geq 1/2$, we deduce that second inequality from \eqref{eq:one1} must fail for $s < 1/2$.

\section{Proof of Theorem \ref{th:probabilistic}} \label{sec:prob}

Let us recall from the introduction that the proof of Theorem \ref{th:probabilistic} follows from the inequalities in \eqref{eq:beck1}. Before describing the organization of the section, let us split the inequalities from \eqref{eq:beck1} according to the different contributions to the $Z^{s, b}$ norm.

First, similarly to the deterministic case, we choose $b = \frac12 + \delta$, and recall that $\delta$ will be sufficiently small depending of $\alpha$, $s$. We start by removing some of the time localizations in \eqref{eq:beck1}. Indeed, notice that the estimates in \eqref{eq:beck1} can be deduced from
\begin{align} \label{eq:beck15} \begin{split}
		&\left\| \p_{x_1}(v w) \right\|_{Z^{s, -\frac12+2\delta}} \les  \| v \|_{Z^{s, \frac12 + \delta}} \| w \|_{Z^{s, \frac12 + \delta}}, \qquad
		\left\|  \p_{x_1}(v S(t)u_{0, r}^\omega )  \right\|_{Z^{s, -\frac12+2\delta}_T} \les \| v \|_{Z^{s, \frac12 + \delta}_T}, \\
		&\left\|  \p_{x_1}( (S(t)u_{0, r}^\omega)^2 )  \right\|_{Z^{s, -\frac12+2\delta}_T} \les 1, 
		\qquad \forall v, w \in Z^{s, \frac12 + \delta}_T,
\end{split} \end{align}
by applying \eqref{eq:beck15} to some $v', w' \in Z^{s, \frac12 + \delta}$, extensions of $v, w$, such that $\| v' \|_{Z^{s, \frac12 + \delta}} \leq 2 \| v\|_{Z^{s, \frac12}_T}$ and $\| w' \|_{Z^{s, \frac12 + \delta}} \leq 2 \| w \|_{Z^{s, \frac12 + \delta}_T}$ (such extensions exists by definition of the space, see \eqref{eq:localized_spaces}).

First, let us focus on the first inequality in \eqref{eq:beck15}. Using the definition of $Z^{s, b}$ norm, we see that for $s > 1/2$ the inequality follows from the deterministic estimate (equation \eqref{eq:one1}) combined with the following $L^\infty$ estimate:
	\begin{equation} \label{eq:Linfty_1}
		\| \langle \tau - \phi (\xi, n_2) \rangle^{-\frac12+2\delta} \widehat{ \langle \p_x \rangle (u v ) } \|_{L^\infty_{\xi, n_2} (L^2_\tau )} \les \| u \|_{Z^{s, \frac12 + \delta}} \| v \|_{Z^{s, \frac12 + \delta}}.
	\end{equation} 
	We recall that $\langle \p_{x} \rangle$ is the operator with Fourier multiplier $\langle \xi \rangle$ which already includes the weight $\frac{\langle \xi \rangle}{| \xi |}$ (applied to a derivative $\p_x$). We stress that our notation $L^\infty_{\xi, n_2} (L^2_\tau )$ means that the $L^2_\tau$ norm is performed first, and the $L^\infty_{\xi, n_2}$ afterwards. 

Now, we focus on the other two estimates from \eqref{eq:beck15}, expressing those estimates in terms of $u_0^\omega$, before taking the real part. Using that $\phi (\xi, n_2)$ is odd with respect to the first variable, we have that $S(t) \widebar{ u_{0}^\omega } = \widebar{ S(t) u_{0}^\omega }$. Therefore
\begin{align*}
	\p_{x_1} ((S(t) u_{0, r}^\omega)^2) &= \frac12\p_{x_1} \left[ (S(t) u_{0}^\omega)^2 \right] + \p_{x_1} \left[ (S(t) u_{0}^\omega) (\widebar{ S(t) u_{0}^\omega } ) \right] +  \frac12 \p_{x_1} \left[ (\widebar{ S(t) u_{0}^\omega })^2 \right],\\
	\p_{x_1} (vS(t)u_{0, r}^\omega) &= \p_{x_1} (v S(t)u_{0}^\omega ) + \p_{x_1} \left( v S(t) \widebar{ u_{0}^\omega } \right).
\end{align*}
so that we can reduce the last two estimates in \eqref{eq:beck15} to the estimates:
\begin{align} \begin{split} \label{eq:beck25}
		\left\|  \p_{x_1}(v S(t)u_{0}^\omega )  \right\|_{Z^{s, -\frac12+2\delta}_T} &\les \| v \|_{Z^{s, b}}\qquad \forall v \in Z^{s, \frac12 + \delta}, \\
		\left\|  \p_{x_1}( (S(t)u_{0}^\omega)^2 )  \right\|_{Z^{s, -\frac12+2\delta}_T} &\les 1, \qquad 
		\left\|  \p_{x_1}( S(t)u_{0}^\omega \widebar{ S(t)u_{0}^\omega } )  \right\|_{Z^{s, -\frac12+2\delta}_T} \les 1.
\end{split} \end{align}
With respect to the first estimate, the main part is the one corresponding to the $Y^{s, b}$ part of the $Z^{s, b}$ norm, which is proved in the following proposition.

\begin{proposition} \label{prop:mixed} Fix $\alpha < \frac{25}{26}$ and $s > 1/2$. There exists $\delta > 0$ sufficiently small such that
	\begin{equation}
		\left\|  \p_{x_1}(v S(t)u_{0}^\omega )  \right\|_{Y^{s, -\frac12+2\delta}_T} \les \| v \|_{Z^{s, \frac12 +\delta}}\qquad \forall v \in Z^{s, \frac12 + \delta}
	\end{equation}
	for almost every $\omega$. The implicit constant in $\les$ is allowed to depend on our parameters and $\omega$, but not on $v$.
\end{proposition}
Combining this Proposition with the $L^\infty$ estimate:
\begin{equation} \label{eq:Linfty_2}
	\left\| \langle \tau - \phi (\xi, n_2) \rangle^{-\frac12+2\delta} \mathcal F_{x, t} \left[ \chi (t) \langle \p_{x_1} \rangle (v S(t)u_{0}^\omega )  \right] \right\|_{L^\infty_{\xi, n_2} (L^2_\tau )} \les \| v \|_{Z^{s, \frac12 + \delta}}\qquad \forall v \in Z^{s, \frac12 + \delta},
\end{equation}
 we deduce the first estimate of \eqref{eq:beck25}. Here, $\chi(t)$ is a cut-off function with value $1$ on $[0, T]$.

With respect to the bounds on the second line of \eqref{eq:beck25}, we also note by the Fourier definition of the $Z^{s, \frac12 + \delta}$ spaces that: 
$$\| \p_x F \|_{Z^{s, \frac12 + \delta}} \les \|  \p_x  F \|_{Y^{s, \frac12 + \delta}} + \| \langle \tau - \phi(\xi, n_2) \rangle^b \widehat{ \langle \p_x \rangle F} \|_{L^\infty_{\xi, n_2} (L^2_\tau )} 
\les \|  \langle \p_x \rangle F \|_{X^{s, \frac12 + \delta}} + \| \langle \tau - \phi(\xi, n_2) \rangle^{\frac12 + \delta} \widehat{ \langle \p_x \rangle F} \|_{L^\infty_{\xi, n_2} (L^2_\tau )} $$

Thus, the second line of \eqref{eq:beck25} reduces to 
\begin{equation} \label{eq:beck3}
	\left\|  \langle \p_{x_1} \rangle ( (S(t)u_{0}^\omega)^2 )  \right\|_{X^{s, -\frac12+2\delta}_T} \les 1, \qquad 
	\left\|  \langle \p_{x_1} \rangle ( S(t)u_{0}^\omega \widebar{ S(t)u_{0}^\omega } )  \right\|_{X^{s, -\frac12+2\delta}_T} \les 1,
\end{equation}
and
\begin{align} \begin{split} \label{eq:Linfty_3}
	\left\| \langle \tau - \phi (\xi, n_2) \rangle^{-\frac12+2\delta} \mathcal F_{x, t} \left[  \langle \p_{x_1} \rangle ( (\chi (t) S(t)u_{0}^\omega)^2 )  \right] \right\|_{L^\infty_{\xi, n_2} (L^2_\tau )} \les 1, \\
	\left\| \langle \tau - \phi (\xi, n_2) \rangle^{-\frac12+2\delta} \mathcal F_{x, t} \left[  \langle \p_{x_1} \rangle ( \chi (t) S(t)u_{0}^\omega \widebar{ \chi(t) S(t)u_{0}^\omega } )  \right] \right\|_{L^\infty_{\xi, n_2} (L^2_\tau )} \les 1.
\end{split} \end{align}

Lastly, we note that \eqref{eq:beck3} holds almost surely as long as it holds in $L^2_\omega$. The argument is the following one. If $\mathbb E_\omega ( F(\omega )^2 ) \leq C$, Markov inequality ensures that $F(\omega )^2 \leq \frac{C}{\gamma}$ unless in an exceptional set of probability at most $\gamma$. In particular, the probability of the set of $\omega$ where $F(\omega )\les_\omega 1$ has to be $1$, since otherwise we would get contradiction for some sufficiently small $\gamma$. 

We now state the needed estimates as Propositions

\begin{proposition} \label{prop:inhomogeneus} Fix $\alpha < \frac{25}{26}$ and $s > 1/2$. There exists $\delta > 0$ sufficiently small such that
	\begin{equation}
		\mathbb E_\omega \left( \left\|  \langle \p_{x_1} \rangle ( (S(t)u_{0}^\omega)^2 )  \right\|_{X^{s, -\frac12+2\delta}_T}^2 \right) \les 1, \qquad \mbox{ and } \qquad 
		\mathbb E_\omega \left( \left\|  \langle \p_{x_1} \rangle ( S(t)u_{0}^\omega \widebar{ S(t)u_{0}^\omega } )  \right\|_{X^{s, -\frac12+2\delta}_T}^2 \right) \les 1.
	\end{equation}
\end{proposition}

\begin{proposition} \label{prop:Linfty} Fix $\alpha < \frac{25}{26}$ and $s > 1/2$. There exists $\delta > 0$ sufficiently small such that. We have the $L^\infty$ estimates \eqref{eq:Linfty_1}, \eqref{eq:Linfty_2} and \eqref{eq:Linfty_3}.
\end{proposition}

All in all, we have reduced the proof of Theorem \ref{th:probabilistic} to show Propositions \ref{prop:mixed}, \ref{prop:inhomogeneus} and \ref{prop:Linfty}. The organization of this Section is as follows. Each subsection of this section corresponds to the proof of one of those Propositions, starting with Proposition \ref{prop:inhomogeneus}, then Proposition \ref{prop:mixed} and lastly Proposition \ref{prop:Linfty}.

\subsection{Proof of Proposition \ref{prop:inhomogeneus}}
\label{subsec:inho}

The objective of this section is to show Proposition \ref{prop:inhomogeneus}. Recall that the notation here is slightly different from Section \ref{sec:deterministic}, since $k, m, n, \in \mathbb Z^2$, where the first component refers to the integer frequency, while we keep $\nu, \zeta, \xi$ for the real first component of the frequency. We also recall $\tilde \nu = \nu - k_1 \in [-1, 1]$ (and analogously $\tilde \zeta$, $\tilde \xi$), so that when we are looking at $P_k u_0$ the frequency can be expressed as $(\nu, k_2) = (\tilde \nu, 0) + k \in \mathbb R \times \mathbb Z$. We have
\begin{align*}
\widehat{ S(t) u_0^\omega } (\zeta, m_2, \nu) =  \frac{g_m(\omega)}{\langle m \rangle^\alpha } z_m (\tilde \zeta) \delta (\eta - \phi_2) 
\end{align*}
where we recall that $\{g_m^\omega \}_{m\in \mathbb Z^2}$ are iid Gaussians, $\phi_2 = \phi (\zeta, m_2)$ and we have defined \begin{equation} \label{eq:uracilo} z_m (\tilde \zeta) = \langle m \rangle^\alpha \cdot \widehat{P_m u_0} (m_1 + \tilde \zeta, m_2). \end{equation} We also recall that $P_m$ was defined in the introduction as a projection to frequencies on the segment from $(m_1-1, m_2)$ to $(m_1 + 1, m_2)$ such that $f = \sum_{m \in \mathbb Z^2} P_m f$. From our generic assumption on $u_0$, we obtain that $\| z_m \|_{L^\infty} = c_m \les 1$ uniformly on $m$.

 We have
\begin{align*}
&\mathcal F_{x, t} \left[ \langle \p_{x_1} \rangle ((S(t)u_0^\omega)^2) \right] (- \xi, -n_2, -\tau) \\
&\qquad = -i  \langle \xi \rangle \int_{\bar S^{n, \xi}} \frac{g_k(\omega) }{\langle k \rangle^\alpha} \frac{g_m(\omega)}{\langle m \rangle^\alpha}
z_m (\tilde \zeta) z_k(\tilde \nu) \int_{\mb R} \delta (-\tau - \eta - \phi_1) \delta ( \eta - \phi_2 ) d\eta d \bar \sigma_{n, \xi} \\
&\qquad = 
-i \langle \xi \rangle \int_{\bar S^{n, \xi}} \frac{g_k(\omega)}{\langle k \rangle^\alpha} \frac{g_m(\omega)}{\langle m \rangle^\alpha}
z_m (\tilde \zeta) z_k (\tilde \nu)  \delta ( -\tau - \phi_1 - \phi_2) d \bar \sigma_{n, \xi}.
\end{align*}
 We are also denoting $\phi_1 = \phi ( \nu, k_2) = \phi (- \xi -  \zeta , -n_2-m_2)$.

Since we will be looking at a norm localised in time, we can multiply by any function which is $1$ on $[0, T]$, as the smooth cutoff $\chi (t)$. We take $\chi$ to be even and denote as $h(\tau)$ its Fourier Transform (which is real, even and decays exponentially fast). We have
\begin{align*}
\mathcal F_{x, t} \left[ \chi (t) \langle \p_{x_1} \rangle ((S(t)u_0^\omega)^2) \right] (- \xi, -n_2,  -\tau)  = 
\langle  \xi \rangle \int_{\bar S^{n, \xi}}  \frac{g_k(\omega)}{\langle k \rangle^\alpha} \frac{g_m(\omega)}{\langle m \rangle^\alpha}
  z_k(\tilde \nu) z_m(\tilde \zeta )h ( -\tau - \phi_1 - \phi_2)  d \bar \sigma_{n, \xi}. 
\end{align*}

The $X^{s, -\frac12+2\delta}$ norm squared of $\langle \p_{x_1} \rangle ( \chi (t) (S(t)u_0^\omega)^2)$ is given by
\begin{align}
&\| \langle \p_{x_1} \rangle( \chi(t) (S(t)u_0^\omega)^2) \|_{X^{s, -\frac12+2\delta}}^2 \notag \\
&\qquad  =\int_\tau \sum_{n_2} \int_{ \xi}
\frac{ \langle n \rangle^{2s} \cdot \langle \xi \rangle^2 }{ \langle -\tau + \phi_3 \rangle^{1-4\delta} }
 \left| \int_{\bar S^{n, \xi}}  \frac{g_k(\omega) g_m (\omega)}{\langle k \rangle^\alpha \cdot \langle m \rangle ^\alpha }
z_k(\tilde \nu) z_m(\tilde \zeta) h( -\tau - \phi_1 - \phi_2 ) d \bar \sigma_{n, \xi} \right|^2 \notag \\
&\qquad  \les
\int_{\tilde \tau} \frac{1}{\langle \tilde \tau  \rangle^{1 - 4\delta} } \sum_{n_2} \int_{ \xi} \langle n \rangle^{2s} \langle n_1 \rangle^2
 \underbrace{ \left| \int_{\bar S^{n, \xi}} \frac{g_k(\omega) g_m (\omega)}{\langle k \rangle^\alpha \cdot \langle m \rangle ^\alpha }
 z_k(\tilde \nu) z_m(\tilde \zeta) h( \tilde \tau + \Delta )  d \bar \sigma_{n, \xi} \right|^2 }_{\mathcal I}, \label{eq:trudeau}
\end{align}
where we recall $-\Delta = \phi_1 + \phi_2 + \phi_3$ and that $h$ is even.

In order to show Proposition \ref{prop:inhomogeneus}, it suffices to bound \eqref{eq:trudeau}, which we rewrite as:
\begin{equation} \label{eq:marin0}
\mathbb E_{\omega} \left( \| \langle \p_{x_1} \rangle( \chi(t) (S(t)u_0^\omega)^2) \|_{X^{s, -\frac12 + 2\delta}}^2 \right) 
\les \int_{\tilde \tau} \frac{1}{\langle \tilde \tau  \rangle^{1 - 4\delta} } \sum_{n_2} \int_{ \xi} \langle n \rangle^{2s} \langle n_1 \rangle^2
\mathbb E_\omega (\mathcal I_{n, \xi} )
\end{equation}
where
\begin{align} \label{eq:marin} \begin{split}
 \mathbb E_\omega  (\mathcal I_{n, \xi}) &= \int_{\bar S^{n, \xi}}  \int_{\bar S^{n, \xi}}  \frac{\mathbb E_\omega (g_k (\omega ) g_m (\omega ) \overline{ g_{k'} (\omega ) g_{m'} (\omega ) } ) }{\langle k \rangle^\alpha \langle m \rangle^\alpha \langle k' \rangle^\alpha \langle m' \rangle^\alpha } 
h(\tilde \tau + \Delta  ) z_k(\tilde \nu) z_m(\tilde{\zeta}) h(\tilde \tau + \Delta'  ) \overline{ z_{k'}(\tilde \nu') z_{m'}(\tilde{\zeta}') }  d \bar \sigma_{n, \xi}  d \bar \sigma_{n, \xi}'
\end{split} \end{align}
Here $n_2, \xi$ are fixed and $\nu, \zeta, m, k, \Delta$ are obtained from $\sigma_{n, \xi} \in \bar S^{n, \xi}$ as usual, while $\nu', \zeta', m', k', \Delta'$ are obtained from $\sigma_{n, \xi}' \in \bar S^{n, \xi}$.

Now, when studying $\mathbb E_\omega (\mathcal I_{n, \xi})$, there are three possibilities on which the expectation of the product of Gaussians is not zero, corresponding to the three possible pairings among $\{ k, m, k', m' \}$. The three possibilities are $(k, m) = (k', m')$, $(k, m) = (m', k')$ and $(k, k') = -(m, m')$ (those are not exclusive, some special cases fall on more than one possibility).

 \begin{remark} 
 
In any case where $(k, k') = -(m, m')$, since $k + m + n = (-\tilde \nu - \tilde \zeta - \tilde \xi, 0)$, we have that $|n| \leq 3$, which allows an easy treatment of this interaction that we show now. Let $\Upsilon$ be the contribution of \eqref{eq:marin0} from the case $\langle n \rangle \leq C$ for some universal constant $C$ (in particular, this includes $(k, k') = -(m, m')$). We use that $\| z_k \|_{L^\infty} \leq 1$. Moreover, since $2-2(b+\delta) = 1 - 4\delta$ multiply and divide by $\langle \tilde \tau \rangle^{5\delta}$ for integrability purposes. We get:
\begin{align*}
\Upsilon &\les \int_{\mathbb R} \frac{d \tilde \tau}{\langle \tilde \tau  \rangle^{1+\delta} } \sum_{n_2 = -C}^C \int_{-C}^C d\xi
 \left| \int_{\bar S^{n, \xi}} \frac{d \sigma_{n, \xi}}{ \langle m \rangle^{2\alpha} }
 |h( \tilde \tau + \Delta )| \langle \tilde \tau \rangle^{5\delta} \right|^2 
 \\
 &\les
\left( \int_{\mathbb R} \frac{d \tilde \tau}{\langle \tilde \tau \rangle^{1+\delta}} \right) \sup_{\tilde \tau} \sum_{n_2 = -C}^C \int_{-C}^C d\xi \left| \int_{\bar S^{n, \xi}} \frac{|h(\tilde \tau + \Delta )| \langle \tilde \tau \rangle^{5\delta} }{\langle m \rangle^{2\alpha}} d\bar \sigma_{n, \xi}\right|^2 \\
&\les 
\sup_{\tilde \tau} \left(  \sum_{m_2 \in \mathbb Z} \frac{1}{\langle m_2 \rangle^{2\alpha} } \sum_{n_2 = -C}^C \int_{-C}^C d\xi  \int_{\bar S^{n, m_2, \xi}} |h(\tilde \tau + \Delta )| \langle \tilde \tau \rangle^{5\delta}  d \bar \sigma_{n, m_2, \xi}\right)^2 \\
&\les 
\sup_{\tilde \tau, n_2, m_2} \frac{1}{\langle m_2 \rangle^{20\delta}} \left(  \int_{-C}^C d\xi  \int_{\bar S^{n, m_2, \xi}} \frac{1}{\langle \tilde \tau + \Delta \rangle^2 } \langle \Delta \rangle^{5\delta}  d \bar \sigma_{n, m_2, \xi}\right)^2 \\
&\les 
\sup_{\tilde \tau, n_2, m_2} \frac{1}{\langle m_2 \rangle^{20\delta}} \left(  \int_{\mathbb R} d\zeta \langle m \rangle^{15\delta}  \int_{\bar S^{n_2, m_2, \zeta}} \mathbbm{1}_{-C \leq \xi \leq C } \frac{ d \bar \sigma_{n_2, m_2, \zeta} }{\langle \tilde \tau + \Delta \rangle^2 }   \right)^2,
\end{align*}
where we used $|h(x)| \les \langle x \rangle^{-2}$ (decays fast), $\alpha > 1/2$, and $|\Delta | \les \langle m \rangle^3$ (given that $\langle n \rangle \les 1$, we have that $\langle m \rangle$ is the size of the maximum frequency, up to a constant).

Now, we consider two cases. If $|\zeta| < 10C$, then both integrals above run over regions of size $O(C)$, and $N  \les \langle m_2 \rangle$, so we simply bound $\Upsilon \les \langle m_2 \rangle^{-5\delta} \les 1$.

If $|\zeta | \geq 10C$, using that $|n_2|, |\xi| \leq C$, we know that $$\theta_1 - \theta_3 = \sqrt{3(\xi + \zeta )^2 + (n_2 + m_2)^2} - \sqrt{3\xi^2 + n_2^2} \approx N_{\rm{max}}.$$ Thus, from the following Corollary \ref{cor:integral}, we see that the double integral is bounded by $\int d\zeta \frac{1}{N_{\mathrm{max}}^2} \les \frac{1}{N_{\rm{max}}}$, and we get that $\Upsilon \les N^{15\delta - 1} \les 1$.

Thus, from now on, we may assume implicitly that $N_3$ is sufficiently large (in particular, $N$ is sufficiently large).
\end{remark}

We now continue our discussion with the bounds on \eqref{eq:marin} in the case where $(k, m) = (k', m')$. The case $(k, m) = (m', k')$ is analogous since $\mathbb E_\omega \mathcal I_{n, \xi}$ is invariant under the change $(m', \tilde \zeta') \leftrightarrow (k', \tilde \nu' )$, so we can consider $(k, m) = (k', m')$ without loss of generality. Thus
\begin{align}
\mathbb E_\omega & \left( \| \langle \p_{x_1} \rangle ( (S(t)u_0^\omega)^2) \|_{X^{s, -\frac12+2\delta}}^2 \right) \notag \\
&\qquad \notag \les 
1+ \int_{\tilde \tau} \frac{1}{\langle \tilde \tau  \rangle^{1-4\delta} } \sum_{n_2} \int_{\xi} \langle n \rangle^{2s} \langle n_1 \rangle^2
 \sum_{m} \frac{1}{\langle k \rangle^{2\alpha} \cdot \langle m \rangle ^{2\alpha} }
\left| \int_{\bar S^{n, m, \xi}} h( \tilde \tau + \Delta )  z_k(\tilde \nu) z_m( \tilde \zeta) d \bar \sigma_{n, m, \xi}  \right|^2 \notag \\
&\qquad \les 
1+ \int_{\tilde \tau} \frac{1}{\langle \tilde \tau  \rangle^{1-4\delta} } \sum_{n_2} \int_{\xi} \langle n \rangle^{2s} \langle n_1 \rangle^2
 \sum_{m} \frac{1}{\langle k \rangle^{2\alpha} \cdot \langle m \rangle ^{2\alpha} }
\left| \int_{\bar S^{n, m, \xi}} h( \tilde \tau + \Delta )   d \bar \sigma_{n, m, \xi}  \right|^2 
\label{eq:madagascar}
\end{align}
where we used the fact that $\| z_k \|_{L^\infty} \les 1$ uniformly on $k$. Now, we will divide the expression above in dyadic frequencies.

Now, let us recall the dyadic parameters $M_3, M_2$, that measure $| \theta_1 - \theta_2|$ and $| \theta_1 - \theta_3|$ respectively. We also recall $N_{\ast i}^{1/2} / N \approx M_{i, \rm{min}} \leq M_i \les N$. In particular, there are $O( \log (N) )$ possible values for each $M_i$. We will consider $\iota$ to be a tuple containing all the $N_i$, $N_{\ast i}$, $M_i$. Similarly the $\bar S_\iota$ sets, (or the restricted versions like $\bar S_\iota^{n, m, \xi}$) correspond to the restriction of the original $\bar S$ set to those tuples that are consistent with the dyadic values dictated by $\iota$. 

Splitting \eqref{eq:madagascar} in terms of $\iota$:
\begin{align} \label{eq:lesotho}
\| \langle \p_{x_1} \rangle ( (S(t)u_0^\omega)^2) \|_{X^{s, -\frac12+2\delta}}^2 
&\les 
1+ \sum_{\iota} \frac{N_{\ast 3}^2 N_3^{2s} \log (N)}{N_1^{2 \alpha} N_2^{2 \alpha}} \int_{\tilde \tau} \frac{1}{\langle\tilde \tau  \rangle^{1-4\delta} } \sum_n \int_{\tilde{\xi}} \sum_{m} \left( \int_{\bar S^{n, m, \xi}_\iota} | h( \tilde \tau + \Delta )| d \bar \sigma_{n, m, \xi} \right)^2.
\end{align}
Note that $k, m, n, \tilde{\xi}$ need to run over the sets described $N_i, N_{\ast i}, M_i$ as well (not just $\tilde \zeta, \tilde \nu$) since otherwise we would have $\bar S^{n, m, \xi}_\iota = \emptyset$, making the contribution zero.

\begin{lemma} \label{lemma:integral} Fix some dyadic numbers given by the tuple $\iota = (N_i, N_{\ast i}, M_i)_{i=1}^3$, and some $\lambda \in \mathbb R$, $\delta > 0$. Assume that either $M_3 > M_{3, \rm{min}}$ or $N_{\ast 3} \neq 1$. For any fixed $\xi, n_2, m_2$, we have:
\begin{equation}  \label{eq:sweden}
\int_{\bar S^{n_2, m_2, \xi}_\iota} \frac{1}{\langle \lambda - \Delta ) \rangle^{1+\delta}}  d \bar \sigma_{n_2, m_2, \xi} \les_\delta \frac{1}{NM_3} 
\end{equation}
where the implicit constant inside $\les$ is independent of $\lambda$. Note that the same bound holds integrating over $\bar S^{n, m_2, \xi}$, given that $\bar S^{n, m_2, \xi} \subset \bar S^{n_2, m_2, \xi}$.
\end{lemma}
\begin{proof}
	We have that $n_1 \in  \{ \lfloor \xi \rfloor, \lceil \xi \rceil \}$, so let us fix $n_1$ to be one of those two values (which will be treated in the same way), and show the bound for $n_1$ fixed.  
	
 We can write $\Delta$ in terms of $n_2, m_2, \xi, \zeta$ as $$\Delta = \zeta (\zeta^2 + m_2^2) + \xi ( \xi^2 + n_2^2 ) - (\zeta + \xi) \left( (\zeta + \xi)^2 + (n_2 + m_2)^2 \right).$$ Thus, one can write $\p_\zeta \Delta$ in terms of $\xi, n_2, m_2$ as:
\begin{equation*}
\p_{\zeta} \Delta = 3 \zeta^2 + m_2^2 - 3 (\xi + \zeta)^2 - (n_2+m_2)^2 = \theta_2^2 - \theta_1^2.
\end{equation*}
We can clearly assume $\bar S^{n, m_2, \xi}_\iota \neq \emptyset$, and in such case we must have we have $\theta_1 + \theta_2 \gtrsim N_1 + N_2 \gtrsim N_{\rm{max}}$. Since $\theta_1 + \theta_2 \les N$, we get that \begin{equation} \label{eq:eren}
| \p_{\zeta} \Delta | \approx N | \theta_1 - \theta_2|.\end{equation}

Now, we use the coordinate $\zeta$ to parametrize $\bar \sigma_{n, m_2, \xi} \in \bar S_\iota^{n, m_2, \xi}$. We can obtain the other frequencies from $\bar \sigma_{n, m_2, \xi}$ noting that $k_2 = -n_2-m_2$, $\nu = - \xi - \zeta$, and $k_1 \in \{ \lfloor \zeta \rfloor, \lceil \zeta \rceil \}$. The two options for $k_1$ are treated in the same way, so let us show the Lemma assuming $k_1 = \lfloor \zeta \rfloor$. We define the interval $I_\iota^{n, m_2, \xi}$ to be the interval where $\zeta$ ranges as $\bar \sigma_{n, m_2, \xi}$ ranges over $\bar S^{n, m_2, \xi}_\iota$. The proof of \eqref{eq:sweden} can be reduced to prove
\begin{equation}
\int_{I_\iota^{n, m_2, \xi}} | h (\lambda - \Delta ) | d\zeta \les \frac{1}{NM_3}
\end{equation}

Let us start with the case $M_3 > M_{3, \rm{min}} \approx N_{\ast 3}^{1/2}/N$, on which \eqref{eq:eren} yields $| \p_{\zeta} \Delta | \approx N M_3$. For any fixed $a, b \in \mathbb R$, $a < b$, we estimate the measure of the subset of $I_\iota^{n, m_2, \xi}$ where $\Delta \in (a, b)$. Call that set $\mathcal S_{(a, b)}$. Moreover, since $\p_{\zeta} \Delta$ will change signs at most two times (it is a second degree polynomial), we can divide $\mathbb R$ in three intervals $I_1$, $I_2$, $I_3$ such that $\Delta$ is monotone with respect to $\zeta$. Without loss of generality let us assume $\Delta$ is increasing on $I_1$ and let us estimate $| \mathcal S_{(a, b)} \cap I_1 |$. We know that $\Delta$ is increasing in $I_1$ and its derivative is $\gtrsim NM_3$ in $\mathcal S_{(a, b)} \cap I_1$. Thus, letting $\ell, r$ be the left and right endpoints of $\mathcal S_{(a, b)} \cap I_1$, we have that $\Delta |_{\zeta = r} - \Delta |_{\zeta = \ell} \gtrsim NM_3 | \mathcal S_{(a, b)} \cap I_1 |$. However, that difference is at most $b-a$, since $\Delta \in (a, b)$ when $\zeta \in \mathcal S_{(a, b)}$. Therefore, we conclude that $| \mathcal S_{(a, b)} \cap I_1 | \les \frac{b-a}{NM_3}$. Reasoning analogously for $I_2, I_3$, we conclude
\begin{equation} \label{eq:Sestimate}
|\mathcal S_{(a, b)} | \les \frac{b-a}{NM_3}
\end{equation}
Now, using that $|h(x)| \les \langle x \rangle^{-2}$ (it decays rapidly), we can write
\begin{align} \label{eq:afterSestimate}
\int_{I_\iota^{n, m_2, \xi}} \langle h(\lambda - \Delta ) \rangle^{-1-\delta} d \zeta 
&\les  \int_{\mathcal S_{(\lambda-1, \lambda+1)}}  \frac{d \zeta}{\langle \lambda - \Delta \rangle^{1+\delta}} + 
\sum_{j=0}^\infty \int_{\mathcal S_{(\lambda+2^j, \lambda+2^{j+1} )}} \frac{d \zeta}{\langle \lambda - \Delta \rangle^{1+\delta}}\\
&\quad + 
\sum_{j=0}^\infty \int_{\mathcal S_{(\lambda-2^{j} , \lambda-2^{j+1} )}} \frac{d \zeta}{\langle \lambda - \Delta \rangle^{1+\delta}} \notag \\
&\les | \mathcal S_{(\lambda-1, \lambda+1)} | + \sum_{j=0}^\infty \frac{|\mathcal S_{(\lambda+2^j, \lambda+2^{j+1} )} | }{ 2^{(1+\delta )j} }
+ \sum_{j=0}^\infty \frac{|\mathcal S_{( \lambda-2^{j+1} , \lambda-2^j)} | }{ 2^{(1 +\delta )j} } \notag \\
&\les \frac{1}{NM_3} + \sum_{j=0}^\infty \frac{2^j}{NM_3 2^{(1+\delta )j}} \les_\delta \frac{1}{NM_3} \notag
\end{align}
where we used \eqref{eq:Sestimate} to get the first expression of the last line. 

Now, we need to treat the case $M_3 = M_{3,\rm{min}}$, in which we have $N_{\ast 3} > 1$. Let us estimate $\mathcal S_{(a, b)}$ for $N_{\ast 3}>1$, with the same definition of $\mathcal S_{(a, b)}$ as in the previous case. Let us work on a region $I_1$ where both $\Delta$ and $\p_{\zeta} \Delta$ are monotonic (there are $\les 1$ such regions since $\Delta$ is a third degree polynomial). Note that we have$|\p_{\zeta}^2 \Delta| = |6 \zeta + 6 (\xi - \zeta)| = |6\xi | \geq 6$, since $N_{\ast 3}  >1$.

Now, let $\ell, r$ be the left and right endpoints of $\mathcal S_{(a, b)} \cap I_1$ respectively. In particular, the signs of $\p_{\zeta} \Delta$ and $\p_{\zeta}^2 \Delta$ are constant in $(\ell, r) \subset I_1$. If $\p_{\zeta} \Delta$ has the same sign as $\p_{\zeta}^2 \Delta$, we have that for $x \in I_1 \cap \mathcal S_{(a, b)}$,
\begin{equation*}
\Big| \p_{\zeta} \Delta_{|\zeta = x} \Big| \geq \left| \p_{\zeta} \Delta_{|\zeta = \ell} + \int_{\ell}^x \p_{\zeta}^2 \Delta d \zeta \right| \geq \left| \int_{\ell}^x \p_{\zeta}^2 \Delta d \zeta\right| \gtrsim N_{\ast 3} (x-\ell)
\end{equation*}
If $\p_{\zeta} \Delta$ has different sign from $\p_{\zeta}^2 \Delta$, we can bound
\begin{equation*}
\Big| \p_{\zeta} \Delta_{|\zeta = x} \Big| \geq \left| \p_{\zeta} \Delta_{|\zeta = r} - \int_{x}^{r} \p_{\zeta}^2 \Delta d \zeta \right| \geq \left| \int_{x}^{r} \p_{\zeta}^2 \Delta d\zeta \right| \gtrsim N_{\ast 3} (r-x)
\end{equation*}
Both the integrals of $x-\ell$ and $r-x$ between $[\ell, r]$ are given by $\frac{(r-\ell)^2}{2}$. Therefore, independently of the choice of signs and using that $\p_{\zeta} \Delta$ has constant sign:
\begin{equation*}
\left| \Delta_{|\zeta = r} - \Delta_{|\zeta = \ell} \right| = \left| \int_{\ell}^{r} \p_{\zeta} \Delta d\zeta \right| \geq \frac{(r-\ell)^2}{2}N_{\ast 3} \gtrsim | \mathcal S_{(a, b)}\cap I_1 |^2 N_{\ast 3}
\end{equation*}
and using that $\Delta \in [a, b]$ for $\zeta \in [\ell, r]$, we get the estimate
\begin{equation*} 
| \mathcal S_{(a, b)} \cap I_1| \les (b-a)^{1/2} N_{\ast 3}^{-1/2}.
\end{equation*}
The estimate follows analogously for any other $I_i$ (and recall there are $\les 1$ of them), so we get
\begin{equation}\label{eq:Sestimate2}
| \mathcal S_{(a, b)} | \les (b-a)^{1/2} N_{\ast 3}^{-1/2}.
\end{equation}

Using \eqref{eq:Sestimate2} instead of \eqref{eq:Sestimate} in the reasoning performed at \eqref{eq:afterSestimate}, we conclude the desired bound $$\int_{I_{N_i, M_i}} \langle \lambda + \Delta) \rangle^{-1-\delta} d\zeta \les N_{\ast 3}^{-1/2} \approx \frac{1}{NM_3}$$ where we are using that $M_3 = M_{3,\rm{min}} \approx N_{\ast 3}^{1/2} /N.$
\end{proof}

Now, we state a modified version of the Lemma where the integral is performed in other variables.

\begin{corollary} \label{cor:integral} 
Fix some dyadic numbers given by the tuple $\iota = (N_i, N_{\ast i}, M_i)_{i=1}^3$, and some $\lambda \in \mathbb R$. Assume that either $M_2 > M_{2, \rm{min}}$ or $N_{\ast 2} \neq 1$. For any fixed $n_2, m_2, \zeta$, we have:
\begin{equation}  \label{eq:sweden2}
\int_{\bar S^{n_2, m_2, \zeta}_\iota} \frac{ d \bar \sigma_{n_2, m_2, \zeta} }{ \langle \lambda - \Delta \rangle^{1+\delta}  } \les \frac{1}{NM_2} 
\end{equation}
where the implicit constant inside $\les$ is independent of $\lambda$. Analogously to \ref{lemma:integral}, we stress that $\bar S^{n_2, m, \zeta} \subset \bar S^{n_2, m_2, \zeta}$ so the bound also holds integrating over that second set. 
\end{corollary}
\begin{proof} The proof follows directly from \eqref{eq:sweden2} after swapping the names of the variables as follows $$(n, \xi, \tilde \xi, N_3, N_{\ast 3}, M_3) \leftrightarrow (m, \zeta, \tilde \zeta, N_2, N_{\ast 2}, M_2)$$
\end{proof}


Now, we state a slightly modified version of Lemma \ref{lemma:integral}. In this case, we restrict the interval to a unit interval (integrating in $\tilde{\zeta}$ instead of $\zeta$) and we show that in that case, we can remove the assumption that $M_3 > M_{3, \rm{min}}$ or $N_{\ast 3} \neq 1$.

\begin{corollary} \label{cor:finland} 
Fix some dyadic numbers given by the tuple $\iota = (N_i, N_{\ast i}, M_i)_{i=1}^3$, and some $\lambda \in \mathbb R$. We have:
\begin{equation}  \label{eq:finland}
\int_{\bar S^{n, m, \xi}_\iota}  \frac{  d \bar \sigma_{n, m, \xi} }{\langle \lambda - \Delta \rangle^{1+\delta}} \les_\delta \frac{1}{NM_3} 
\end{equation}
and
\begin{equation}  \label{eq:finland2}
\int_{\bar S^{n, m, \zeta}_\iota}  \frac{d \bar \sigma_{n, m, \zeta}  }{\langle \lambda - \Delta \rangle^{1+\delta} } \les_\delta \frac{1}{NM_2} 
\end{equation}
where the implicit constant inside $\les$ is independent of $\lambda$. 
\end{corollary}
\begin{proof} Let us just show \eqref{eq:finland}, being the argument analogous for \eqref{eq:finland2}. In the case where $M_3 > M_{3, \rm{min}}$ or $N_{\ast 3} > 1$ this follows directly from Lemma \ref{lemma:integral}, noting that \eqref{eq:finland} corresponds to the integral in \eqref{eq:sweden2} restricted to the interval $[m_2, m_2+1)$. Otherwise, we have $NM_3 = N M_{3, \rm{min}} \approx  NN_{\ast 3}^{1/2}/N = N_{\ast 3}^{1/2} = 1$ so the right hand side of \eqref{eq:finland} is $1$. Then \eqref{eq:finland} follows from the fact that $\zeta \in [m_2, m_2+1)$ so we are integrating over a $O(1)$ measure set.

\end{proof}

Using Corollary \ref{cor:finland} to estimate the integral on \eqref{eq:lesotho}, we get
\begin{align*}
\mathbb E_\omega \left( \| \langle \p_{x_1} \rangle ( (S(t)u_0^\omega)^2) \|_{X^{s, -\frac12+2\delta}}^2  \right)
&\les 1+ \sum_{\iota} \frac{N_{\ast 3}^2 N_3^{2s} \log (N)}{N_1^{2 \alpha} N_2^{2 \alpha}} \int_{\tilde \tau}  \sum_n \int_{\tilde{\xi}} \sum_{m} \int_{\bar S^{n, m, \xi}_\iota} \frac{ | h( \tilde \tau + \Delta )| d \bar \sigma_{n, m, \xi} }{ 
\langle\tilde \tau  \rangle^{1-4\delta}  } \frac{1}{NM_3} \\
&\les 1+ \sum_{N_i} \sum_{M_3, M_2} \frac{N_3^{2s} N_{\ast 3}^2 \log (N)}{N_1^{2 \alpha} N_2^{2 \alpha}} \cdot \frac{1}{NM_3} 
 \int_{\bar S_\iota} \int_{\tilde \tau}  \frac{|h( \tilde \tau + \Delta )| }{\langle \tilde \tau  \rangle^{1-4\delta} } d\bar \sigma
\end{align*}
Using the rapid decay of $h$, we have that
\begin{equation*}
 \int_{\tilde \tau} \frac{|h( \tilde \tau + \Delta )| }{\langle \tilde \tau  \rangle^{1-4\delta} }
  \les \int_{\tilde \tau} \frac{1}{\langle \tilde \tau  \rangle^{1-4\delta} \langle \tilde \tau + \Delta \rangle^{100} }
  \les \frac{1}{\langle \Delta \rangle^{1-4\delta}}.
\end{equation*}
Thus, it suffices to bound
\begin{equation*}
\sum_{\iota}  \frac{N_3^{2s}N_{\ast 3}^2 \log (N)}{N_1^{2 \alpha} N_2^{2 \alpha} } \cdot \frac{1}{NM_3}
 \int_{\bar S_\iota} \frac{d \bar \sigma}{\langle \Delta  \rangle^{1-4\delta} }.
\end{equation*}

Using that $\langle \Delta \rangle \les N^3$ and, it suffices to show
\begin{equation} \label{eq:Xi}
\Xi_{\iota} :=  \frac{N_3^{2s} N_{\ast 3}^2  }{N_1^{2 \alpha} N_2^{2 \alpha}} \cdot \frac{1}{NM_3}
 \int_{\bar S_\iota} \frac{d\bar \sigma}{\langle \Delta  \rangle^{1+\delta} } \les N^{-\eps}
\end{equation}
for every $\iota$ and some $\eps$ that will be as small as we want, but with $\delta \ll \eps$. The tuple $\iota$ will encode dyadic numbers $N_i$, $N_{\ast i}$, $M_i$. We will now study this quantity according to different cases. We denote by $c$ some small constant independent of the frequencies to be fixed later.

\begin{enumerate}
\item Case $\mc C_{1}$: $M_3 \geq N/1000$ (large derivative in $ \zeta$ direction)
\item Case $\mc C_2$: $M_3 < N/1000$ and $M_2 \geq N/1000$ (large derivative in $\xi$ direction)
\item Case $\mc C_3$: $c\leq \max \{ M_2, M_3 \} \les N/1000$ (some derivative is $\gtrsim 1$)
\item Case $\mc C_4$: $M_2, M_3 < c$ and $| \Delta | \geq N^2$ (very small derivatives but large $\Delta$)
\item Case $\mc C_5$: $M_2, M_3 < c$, $| \Delta | \leq N^2$ (equillateral triangle configuration).

\end{enumerate}

\subsubsection{$\mc C_1$. Case $M_3 \geq N / 1000$ }
Given that $N_3 \les N$ and $\max \{ N_1, N_2 \} \gtrsim N$, we have
\begin{align*}
\Xi_{\iota} &\les
   \frac{  N^{2s} N_{\ast 3}^2 }{N^{2\alpha} \min \{ N_1 , N_2\}^{2\alpha} N^2}  \sum_{n_2, m_2} \int_\xi \int_{\bar S_\iota^{n_2, m_2, \xi}} \frac{ d\bar \sigma_{n_2, m_2, \xi} }{\langle \Delta \rangle^{1+\delta}}  \les 
   \frac{  N^{2s+2} }{N^{2\alpha} \min \{ N_1 , N_2\}^{2\alpha} N^2 } \sum_{n_2} \sum_{m_2} \int_{\xi} \frac{\mathbbm{1}_{{\bar S_\iota^{n_2, m_2, \xi}} \neq \emptyset } }{N^2}.
\end{align*}
where we applied Lemma \ref{lemma:integral} and used that $M_3 \gtrsim N$. Now, the condition $\mathbbm{1}_{{\bar S_\iota^{n_2, m_2, \xi}} \neq \emptyset }$ implies that $|m_2| \les N_2$ and $|-m_2-n_2| = |k_2| \les N_1$. Once $n_2$ is fixed, those conditions define an interval of size smaller than $2\min \{ N_1 , N_2\}$ where $m_2$ needs to lie. Both $n_2$ and $\xi$ range over intervals of length $\les N$. From the previous estimate, we obtain
\begin{equation*}
\Xi_\iota \les  
   \frac{  N^{2s} N_{\ast 3}^2 }{N^{2\alpha}}   \frac{\min \{ N_1 , N_2\}}{ \min \{ N_1 , N_2\}^{2\alpha} } \frac{N^2}{N^4} \les 
   N^{2(s-\alpha)} \min \{ N_1 , N_2\}^{2(1/2-\alpha)} \les N^{-\eps}
\end{equation*}
using that $\alpha > s$ (recall $s = (1/2)^+$).

\subsubsection{Remark for all remaining cases}

For all the rest of subcases, we will have that $M_3 \leq N/1000$. In particular, this means that $| \theta_1 - \theta_2 | \leq N/500$. Given that $\theta_1 + \theta_2 > | ( \nu, k_2) | + | (\zeta, m_2) |$ is bigger than the maximum frequency, we have that $\theta_1 + \theta_2 > N$. Combining this with $| \theta_1 - \theta_2 | \leq N/500$, we get \begin{equation} \label{eq:boundtheta}
\theta_1 \geq 4N/10, \qquad \mbox{ and } \qquad \theta_2 \geq 4N/10.
\end{equation}
 Finally, since $N_j > \theta_j /\sqrt 3$, we have that $$N_1, N_2 \gtrsim N.$$ We will use this for all the remaining subcases. Thus, we will have to show
 \begin{equation} \label{eq:Xi2}
 	\Xi_{\iota} =  \frac{N_3^{2s} N_{\ast 3}^2  }{N^{4 \alpha}} \cdot \frac{1}{NM_3}
 	\int_{\bar S_\iota} \frac{d\bar \sigma}{\langle \Delta  \rangle^{1+\delta} } \les N^{-\eps}.
 \end{equation}

\subsubsection{$\mc C_2$. Case $M_3 < N/1000$ but $M_2 \geq N/1000$} 
Using Corollary \ref{cor:integral} and $M_2 \gtrsim N$, we get that
\begin{align*}
\Xi_{\iota} &\les 
   \frac{  N^{2s} N_{\ast 3}^2 }{N^{4\alpha} } \frac{1}{NM_3} \sum_{n_2, m_2} \int_{\zeta} \int_{\bar S^{n_2, m_2, \zeta}} \frac{ d \bar \sigma_{n_2, m_2, \zeta} }{\langle \Delta \rangle^{1+\delta}}
   \les  
    N^{2s - 4 \alpha} \frac{N_{\ast 3}^2}{N_{\ast 3}^{1/2}} \sum_{m_2, n_2} \int_{\zeta}  \frac{\mathbbm{1}_{\bar S^{n_2, m_2, \zeta} \neq \emptyset}}{N^2}, 
\end{align*}
where we used that $M_3 \geq M_{3, \mathrm{min}} = N_{\ast 3}^{1/2} / N$. Given that $m_2, n_2, \zeta$ all range over intervals of length $\les N$, we conclude that $\Xi_\iota \les N^{2s+5/2-4\alpha}$ which is smaller than $N^{-\eps}$ as long as $\alpha > \frac78$, for $s-1/2$ sufficiently small depending on $\alpha$.

\subsubsection{Remark for all remaining cases}

Let us recall Lemma \ref{lemma:localization} from the deterministic part, and take $w_i$ to be $(\sqrt{3}\nu, k_2)$, $\sqrt{3} \zeta, m_2)$ and $(\sqrt{3} \xi, n_2)$ respectively. All our following cases will satisfy that $M_2, M_3 \leq N/1000$ and therefore, we are under the hypothesis of Lemma \ref{lemma:localization} for $\mathcal E = 2(M_2+M_3)$. As a consequence
\begin{equation} \label{eq:localization_inhomo}
	( \sqrt{3} \nu, k_2) \in R_{\pm 2\pi/3} (\sqrt{3} \xi, n_2) + B  \qquad \mbox{and} \qquad ( \sqrt{3} \zeta, m_2) \in R_{\mp 2\pi/3} (\sqrt{3} \xi, n_2) + B,
\end{equation}
where we recall that $R_\varphi$ refers to the rotation of angle $\varphi$ and $B$ is a ball of radius $O(\mathcal E) = O(M_2 + M_3)$.

\subsubsection{$\mc C_3$. Case $c\leq \max \{ M_2, M_3 \} \leq N/1000$ . }

We have
\begin{align} \label{eq:mauritius}
\Xi_{\iota} &\les 
   \frac{  N^{2s} N_{3\ast}^2 }{N^{4\alpha} }  \frac{1}{M_3N}  \int_{\bar S_\iota} \frac{1}{\langle \Delta \rangle^{1+\delta}} .
\end{align}

We consider two subcases depending if $M_3 \geq M_2$ or not. If $M_3 \geq M_2$, we use Lemma \ref{lemma:integral} and obtain:
\begin{align*}
 \int_{\bar S_\iota} \frac{d \bar \sigma}{\langle \Delta \rangle^{1+\delta}} \les  \sum_{n_2} \int_{\xi} \sum_{m_2} \int_{\bar S^{n_2, m_2, \xi}_\iota}\frac{d \bar \sigma_{n_2, m_2, \xi}}{\langle \Delta \rangle^{1+\delta}} \les \sum_{n_2} \int_{\xi} \sum_{m_2} \frac{ \mathbbm{1}_{ \bar S^{n_2, m_2, \xi}_\iota \neq \emptyset } }{NM_3} \les \frac{N^2 M_3}{NM_3} = N 
\end{align*}
In the last inequality we used \eqref{eq:localization_inhomo} to justify that $m$ is in some box of side $\les \max \{ M_3 , M_2 \} = M_3 \geq c$ for fixed $(\xi, n_2)$.

Let us now consider the case where $M_3 < M_2$ (in particular, $M_2 > c$). A similar reasoning using Corollary \ref{cor:integral} and \eqref{eq:localization_inhomo} yields
\begin{align*}
 \int_{\bar S_\iota} \frac{1}{\langle \Delta \rangle^{1+\delta}} \les  \sum_{m_2} \int_{\zeta} \sum_{n_2} \int_{\bar S_\iota^{n_2, m_2, \zeta}} \frac{1}{\langle \Delta \rangle^{1+\delta}} \les \sum_{m_2} \int_{\zeta} \sum_{n_2}  \frac{ 
 \mathbbm{1}_{ \bar S_\iota^{n_2, m_2, \zeta} \neq \emptyset }  }{NM_2} \les \frac{N^2 M_2}{NM_2} = N
\end{align*}

All in all, we get from \eqref{eq:mauritius} that
\begin{align*}
\Xi_\iota &\les  
   \frac{  N^{2s} N_{\ast 3}^2 }{N^{4\alpha} }  \frac{1 }{M_3N}  N \les N^{2s+1-4\alpha} N_{\ast 3}^2  \frac{1 }{N_{\ast 3}^{1/2}}
 \les N^{2s+5/2-4\alpha},\end{align*}
 using that $M_3\geq M_{3, \mathrm{min}} \approx N_{3\ast}^{1/2}/N$. We get $\Xi_\iota \les N^{-\eps}$ as long as $\alpha > \frac78$.


\subsubsection{$\mc C_4$. Case $M_2, M_3 < c$ and $|\Delta | \geq N^2$}

Under $M_2, M_3 < c$, we have that \eqref{eq:localization_inhomo} implies $| \bar S_\iota | \les N^2$, since once $(\xi, n_2)$ is fixed, we know that $(\zeta, m_2)$ is in a circle of measure $O(1)$ from \eqref{eq:localization_inhomo}.

Since our division according to $\iota$ does not contemplate the value of $\Delta$, we call $\Xi_{\iota}^{(1)}$ to the contribution to $\Xi_\iota$ that comes from the region where $| \Delta| \geq N^2$.  Using that lower bound on $\Delta$, together with $| \bar S_\iota | \les N^2$ and $M_3 \geq M_{3, \mathrm{min}}$, we have:
\begin{align} 
\Xi_{\iota}^{(1)} & \les  \frac{N^{2s} N_{\ast 3}^2 }{N^{4\alpha}} \frac{1}{M_3 N}  \int_{\bar S_\iota} \frac{ \mathbbm{1}_{| \Delta | \geq N^2} d \bar \sigma }{\langle \Delta \rangle } 
\les  \frac{N^{2s} N_{\ast 3}^2 }{N^{4\alpha}} \frac{1}{N_{\ast 3}^{1/2} }   \frac{| \bar S_\iota |}{N^2} \les N^{3/2+2s-4\alpha},
\end{align}
which is smaller than $N^{-\eps}$ for any $\alpha > \frac58$ provided $s-1/2, \eps$ are sufficiently small.

\subsubsection{$\mc C_5$. Case $M_3, M_2 \leq 1$, and $| \Delta | \leq N^2$}

We refer to $\Xi_\iota^{(2)}$ as the contribution to $\Xi_\iota$ coming from $| \Delta | < N^2$. We also note that $M_2, M_3 < c$ implies $|\theta_3 - \theta_1|, | \theta_2-\theta_1| \les 1$. Taking $c$ sufficiently small to beat the implicit constants on Lemma \ref{cor:coro} (but $c = O(1)$), we see that Lemma \ref{cor:coro} yields contradiction with $|\Delta| \leq N^2$ if $\min \{ | \xi |, | \zeta |, |  \nu |\} \gg 1 $, because we would be satisfying the hypothesis of Lemma \ref{cor:coro} but not the conclusion. Thus, we fix such value of $c$, and we have that $$\min \{ | \xi |, | \zeta |, |  \nu |\} \les 1.$$ 

Now, by \eqref{eq:localization_inhomo}, we have that $\int_{\bar S_\iota} \mathbbm{1}_{| \Delta | \geq N^2} d\bar \sigma \les N$. For example, if $\zeta$ attains the minimum, it runs over an $O(1)$ interval, $m_2$ runs over a $O(N)$ set; and after fixing those, \eqref{eq:localization_inhomo} gives us balls of measure $O(1)$ for the other variables. The same argument is clearly also valid if $\nu$ or $\xi$ attain the minimum.

Thus, we can simply bound:
\begin{equation*}
	\Xi_\iota^{(2)} \les \frac{N^{2s} N_{\ast 3}^2 }{N^{4\alpha}} \frac{1}{NM_3} \int_{\bar S_\iota} \mathbbm{1}_{| \Delta | \leq N^2} \les N^{2s - 4\alpha} N_{\ast 3}^2 \frac{1}{N_{\ast 3}^{1/2}} N \leq N^{2s + 5/2 - 4\alpha}
\end{equation*}
which is smaller than $N^{-\eps}$ for $\alpha > \frac78$ as long as $s-1/2$ and $\eps$ are sufficiently small.

\subsection{Proof of Proposition \ref{prop:mixed}} \label{subsec:mixed}

The objective of this section is to show Proposition \ref{prop:mixed}. Let us recall from \eqref{eq:uracilo} that
\begin{equation*}
\widehat{S(t)u_0^\omega}(\nu, k_2, \mu) = \sum_{k_1 = \lfloor \nu \rfloor}^{\lceil \nu \rceil} \frac{g_k^\omega z_k (\tilde \nu )}{ \langle k \rangle^\alpha} \delta (\mu - \phi_1).
\end{equation*}
Moreover if we apply the time localization $\chi (t)$ whose Fourier Transform $h$ decays rapidly, we obtain
\begin{equation} \label{eq:onett}
\mathcal F \left[\chi(t) S(t)u_0^\omega \right] ( \nu, k_2, \mu) = \sum_{k_1 = \lfloor \nu \rfloor}^{\lceil \nu \rceil} \frac{g_k^\omega}{ \langle k \rangle^\alpha} h (\mu - \phi_1).
\end{equation}

For any $v \in Z^{s, \frac12 + \delta}$, we may write 
\begin{equation} \label{eq:two}
\hat v (\zeta , m_2, \eta ) =  \frac{1}{ \langle m \rangle^s} \frac{1}{\langle \eta - \phi_2 \rangle^{\frac12 +\delta} } \hat v_{\mathrm{ren}} (\zeta, m_2, \eta) 
\end{equation}
so that $v_{\mathrm{ren}}$ is a renormalized version of $\hat v$ satisfying 
\begin{align} \begin{split} \label{eq:marcus}
	\| v \|_{Z^{s, \frac12 + \delta }} &\geq \| v \|_{Y^{s, \frac12 + \delta}} = \left\| \frac{\langle \zeta \rangle}{|\zeta|} \hat v_{\rm{ren}} \right\|_{L^2} \geq \| \hat v_{\rm{ren}} \|_{L^2} \\
	\| v \|_{Z^{s, \frac12 + \delta}} &=
	\left\| \frac{\langle \zeta \rangle }{|\zeta |}\hat v (\zeta, m_2, \eta) \cdot \langle \eta - \phi (\zeta, m_2) \rangle^{\frac12 + \delta} \right\|_{L^\infty_{\zeta, m_2}(L^2_\eta)} =  \left\| \frac{\langle \zeta \rangle }{|\zeta |} \hat v_{\rm{ren}} (\zeta, m_2, \eta) \langle m \rangle^{-s} \right\|_{L^\infty_{\zeta, m_2} (L^2_\eta)} 
\end{split} \end{align}
 Then, we write
\begin{equation*}
\hat v_{\mathrm{ren}} (\zeta, m_2, \eta ) = \psi (\eta - \phi_2) b_{\eta - \phi_2} (\zeta, m_2)
\end{equation*}
where $\| b_{\tilde \eta} \|_{L^2} = 1$ for all $\tilde \eta$ and therefore $\| \hat v_{\mathrm{ren}} \|_{L^2} = \| \psi \|_{L^2}$. That is, we are decomposing $v_{\rm{ren}}$ as a sum of linear solutions localized at each curve $\eta - \phi_2 = \lambda$, each with amplitude $\psi (\lambda )$. The decomposition in terms of the original $v$ is written as
\begin{equation*}
\psi (\tilde \eta ) = \langle \tilde \eta\rangle^{\frac12 + \delta} \left( \sum_{m_2} \int_{\zeta} 
 \langle m \rangle^{2s} | \hat v (\zeta, m_2, \tilde \eta + \phi_2) |^2 \right)^{1/2} \quad \mbox{ and } 
\quad 
b_{\tilde \eta }(\zeta, m_2) = \frac{\langle m \rangle^s 
	 \hat v (\zeta , m_2, \tilde \eta + \phi (\zeta, m_2) )}{ \left( \sum_{k_2} \int_{\nu} \langle k \rangle^{2s} 
	  | \hat v (\nu, k_2, \tilde \eta + \phi (\nu, k_2) |^2 \right)^{1/2} }.
\end{equation*}
With such decomposition $\| b_{\tilde \eta} \|_{L^2} = 1$ for all $\tilde \eta \in \mathbb R$ and $\| \psi (\tilde \eta) \|_{L^2}^2 = \| v \|_{Y^{s, b}}^2$.

Just as in Subsection \ref{subsec:inho}, we will work with the $S$, $\bar S$ sets and their restrictions like $S^{n, \xi}$, $\bar S^{n, \xi}$ describing the possible tuples of frequencies. Now, from equations \eqref{eq:onett}--\eqref{eq:two}:
\begin{align*}
\mathcal F \left[ (\chi (t) S(t)u_0^\omega v) \right] (-\xi, -n_2, -\tau)
 = i \xi \int_{\bar S^{n, \xi}} \frac{1}{\langle k \rangle^\alpha \langle m \rangle^s}  \int_\eta
 g_{k}^\omega z_k (\tilde \nu)  h (-\tau - \eta - \phi_1) \frac{\psi (\eta - \phi_2)}{\langle \eta - \phi_2 \rangle^{\frac12 + \delta} } b_{\eta - \phi_2} (\zeta, m_2) 
\end{align*}
Now, if we look at the $Y^{s,-\frac12+2\delta}$ norm, we have 
\begin{align*}
\left\| \p_x  (\chi (t) S(t)u_0^\omega v) \right\|_{Y^{s, -\frac12+2\delta}}^2 &\les
\sum_{n_2} \int_\xi \langle n \rangle^{2s} | \xi |^2
\int_{\tau} \frac{1}{\langle \tau - \phi_3 \rangle^{1-4\delta}} \\
&\qquad \cdot
\left|    \int_{\bar S^{n, \xi}}\int_{\eta} \frac{g_{k}^\omega z_k (\tilde \nu)}{ \langle k \rangle^{2\alpha} \langle m \rangle^{2s} }
h (-\tau - \eta - \phi_1) \frac{\psi (\eta - \phi_2)}{\langle \eta - \phi_2 \rangle^{\frac12 + \delta} } b_{\eta - \phi_2} (\zeta, m_2)    d\bar \sigma_{n, \xi }\right|^2
\end{align*}
Doing the changes of variables $\tilde \eta = \eta - \phi_2 = \eta - \phi(\zeta, m_2)$ and $\tilde \tau = \tau - \phi_3 = \tau - \phi (\xi, n_2)$, we have that:
\begin{align}
&\left\|  \p_x  (\chi (t) S(t)u_0^\omega v) \right\|_{Y^{s, -\frac12+2\delta}}^2 \notag \\
&\qquad \les
\sum_{n_2} \int_\xi  \int_{\tilde \tau}  
\int_{\tau} \frac{ \langle n \rangle^{2s} | \xi |^2 }{\langle \tilde \tau \rangle^{1-4\delta}}
\left|   \int_{\bar S^{n, \xi} }\int_{\tilde \eta} g_{k}^\omega  z_k (\tilde \nu)
h (-\tilde \tau - \tilde \eta - \Delta) \frac{\psi (\tilde \eta )}{\langle \tilde \eta \rangle^{\frac12 + \delta} } b_{\tilde \eta } (\zeta, m_2)  \right|^2 \label{eq:fruta}
\end{align}
where we used $\Delta = \phi_1 + \phi_2 + \phi_3$. Now, we decompose according to our dyadic partition $\iota$, where $\iota$ includes all the dyadic numbers $N_i$, $N_{\ast i}$, $M_i$ but in this case we do not include $M_1$. We obtain that
\begin{equation}
\mathbb E_{\omega} \left\|  \p_x  (\chi (t) S(t)u_0^\omega v) \right\|_{Y^{s,-\frac12+2\delta}}^2 \leq \sum_{\iota} \log (N)^{10} \mathbb E_{\omega} \Theta_\iota
\end{equation}
where
\begin{equation} \label{eq:Theta}
\Theta_\iota = \frac{N_3^{2s} N_{\ast 3}^2}{N_1^{2\alpha} N_2^{2s}} 
\sum_{n_2} \int_{\xi} \int_{\tilde \tau} \frac{1}{\langle \tilde  \tau  \rangle^{1-4\delta}}
\left|   \int_{\bar S^{n, \xi}_\iota }\int_{\tilde \eta} g_{k}^\omega z_k (\tilde \nu)
h (-\tilde \tau - \tilde \eta - \Delta) \frac{\psi (\tilde \eta )}{\langle \tilde \eta \rangle^{\frac12 + \delta} } b_{\tilde \eta } (\zeta, m_2)  d\bar \sigma_{n, \xi} \right|^2
\end{equation}
The factor $\log (N)^{10}$ accounts for all the dyadic partitions with respect to $N_i$, $N_{\ast i}$ and $M_i$

Therefore, in order to show our Proposition \ref{prop:mixed}, it suffices to show that
\begin{equation} \label{eq:objective_mixed}
\mathbb E_\omega \Theta_\iota \les N^{-\eps} \| v \|_{Z^{s, \frac12 + \delta}}, \qquad \forall v \in Y^{s, \frac12 + \delta}
\end{equation}
with a constant that is independent of $v$.

Using that $\langle \tilde \tau + \tilde \eta \rangle \les  \langle \Delta \rangle  \cdot \langle \tilde \tau + \tilde \eta + \Delta \rangle$, we have
\begin{equation*}
\Theta_\iota \les \frac{N_3^{2s} N_{\ast 3}^2}{N_1^{2\alpha} N_2^{2s}} 
\sum_{n_2} \int_{\xi} \int_{\tilde \tau} \frac{1}{\langle \tilde  \tau  \rangle^{1-4\delta}}
\left|   \int_{\bar S^{n, \xi}_\iota }\int_{\tilde \eta} g_{k}^\omega z_k (\tilde \nu)
h (-\tilde \tau - \tilde \eta - \Delta) \langle \tilde \tau + \tilde \eta + \Delta \rangle^{3\delta} \frac{\psi (\tilde \eta ) \cdot  \langle \Delta \rangle^{3\delta} }{\langle \tilde \eta \rangle^{1/2+\delta} \langle \tilde \eta + \tilde \tau \rangle^{3\delta} } b_{\tilde \eta } (\zeta, m_2)  d\bar \sigma_{n, \xi} \right|^2.
\end{equation*}
Now, recall that $h$ decays arbitrarily fast, so the term $ \langle \tilde \tau + \tilde \eta + \Delta \rangle^{3\delta}$ is superfluous. One can define $\tilde h (x) = h(x) \langle x \rangle^{3\delta}$ and it would satisfy the same properties as $h$ (real, even, and exponential decay), and use it as $h$. With an abuse of notation, we denote the new $h$ by $h$ as well. With respect to $\langle \Delta \rangle^{3\delta}$ we simply bound it by $N^{9\delta}$, since $| \Delta | \les N^3$. Thus, we have
\begin{equation} \label{eq:preCS}
	\Theta_\iota \les N^{9\delta}_{\rm{max}}  \frac{N_3^{2s} N_{\ast 3}^2}{N_1^{2\alpha} N_2^{2s}} 
	\sum_{n_2} \int_{\xi} \int_{\tilde \tau} \frac{1}{\langle \tilde  \tau  \rangle^{1-4\delta}}
	\left|   \int_{\bar S^{n, \xi}_\iota }\int_{\tilde \eta} g_{k}^\omega z_k (\tilde \nu)
	h (-\tilde \tau - \tilde \eta - \Delta)  \frac{\psi (\tilde \eta )  }{\langle \tilde \eta \rangle^{1/2+\delta} \langle \tilde \eta + \tilde \tau \rangle^{3\delta} } b_{\tilde \eta } (\zeta, m_2)  d\bar \sigma_{n, \xi} \right|^2
\end{equation}

We take expectation and do Cauchy-Schwarz in $\tilde \eta$:
\begin{align*} 
	\mathbb E_\omega (\Theta_\iota)  &\les \| \psi \|_{L^2}^2 \cdot N_{\rm{max}}^{9\delta}
	\frac{N_3^{2s} N_{\ast 3}^2}{N_1^{2\alpha} N_2^{2s}} 
	\int_{\tilde \eta}  \int_{\tilde \tau} \frac{1}{\langle \tilde  \tau  \rangle^{1-4\delta} \cdot \langle \tilde \tau + \tilde \eta \rangle^{6\delta} \cdot\langle \tilde \eta \rangle^{2b}} \cdot \\
& \qquad \mathbb{E}_\omega	 \left[ \sum_{n_2} \int_{\xi}  
	\left|    \int_{\bar S_\iota^{n_2, \xi}} g_{k}^\omega z_k (\tilde \nu)
	h (-\tilde \tau - \tilde \eta - \Delta)  b_{\tilde \eta } (\zeta, m_2)  d\bar \sigma_{n, \xi}\right|^2 \right]
\end{align*}
Now recall $\| \psi \|_{L^2} \approx \| v \|_{Y^{s, b}} \les \| v \|_{Z^{s, b}}$. Moreover, since the integral in $\tilde \tau, \tilde \eta$ converges to some value depending on $\delta$, we have:
\begin{align} \begin{split}
	\mathbb E_{\omega} \Theta_\iota \les
\| v \|_{Z^{s, b}}^2 N_{\rm{max}}^{9\delta} 	\frac{N_3^{2s} N_{\ast 3}^2}{N_1^{2\alpha} N_2^{2s}} 
	\sup_{\lambda, \tilde \eta}
	\mathbb E_{\omega} \left[ 
	\sum_{n_2} \int_{\xi} 
	\left|   \int_{\bar S^{n, \xi}_\iota}  g_{k}^\omega z_k (\tilde \nu)
	h (\lambda  - \Delta)  b_{\tilde \eta } (\zeta, m_2)   d \bar \sigma_{n, \xi} \right|^2 \right] 
\end{split} \end{align}

We note that the exponential decay of the Gaussians guarantees that $|g_k| \leq N_1^{\delta}$ for all $k\in \mathbb Z^2$ with $N_1 \leq |k| \leq 2N_1$, except for a set of probability $O(1/N_1)$. Since $\sum_{N_1} 1/N_1 < \infty$ (given that $N_1$ is dyadic), Borel-Cantelli allows us to assume $|g_k| \les N_1^{\delta}$ except in a measure-zero set, that we remove. Therefore, putting absolute values on the inner integral, we see that 

\begin{equation} \label{eq:marcus2}
	\mathbb E_\omega \Theta_\iota \les \| v \|_{Z^{s, b}}^2 N_{\rm{max}}^{10\delta}\frac{N_3^{2s} N_{\ast 3}^2}{N_1^{2\alpha} N_2^{2s}} 
	\sup_{\lambda, \tilde \eta}
	\sum_{n_2} \int_{\xi} 
	\left|   \int_{\bar S^{n, \xi}_\iota}  
	h (\lambda  - \Delta)  b_{\tilde \eta } (\zeta, m_2)  d \bar \sigma_{n, \xi} \right|^2  
\end{equation}
In particular, from \eqref{eq:objective_mixed}, it suffices to show
\begin{equation}
\label{eq:objective_def}
\Theta_\iota' := \frac{N_3^{2s} N_{\ast 3}^2}{N_1^{2\alpha} N_2^{2s}} 
\sup_{\lambda, \tilde \eta}
\sum_{n_2} \int_{\xi} 
\left|   \int_{\bar S^{n, \xi}_\iota}  
h (\lambda  - \Delta)  b_{\tilde \eta } (\zeta, m_2)  d \bar \sigma_{n, \xi} \right|^2 \les N^{-\eps},
\end{equation}
for some $\eps$ sufficiently small (that is allowed to depend on $s-1/2$ and $\alpha$, but will be much larger than $\delta$). We
recall $N := N_{\rm{max}}$ and $\| b_{\tilde \eta} \|_{L^2} = 1$ for every $\tilde \eta \in \mathbb R$.

 Now, we divide in four cases: \begin{itemize}
	\item Case $\mathcal C_{\mathrm{bad}}$ that contains the cases $M_2, M_3 \leq 1/2$.
	\item Case $\mathcal C_{\mathrm{high}}$ where $N_1 \geq \frac18 N_{\rm{max}}$ and $\max \{ M_2 ,  M_3 \} \geq 1$. The subcase $\mathcal C_{\rm{res}}$ (see below) is dealt separately.
    \item Case $\mathcal C_{\mathrm{res}}$. A resonant subcase of $\mathcal C_{\rm{high}}$ where we have that $M_2 \leq N^{-12/13}$ and $N_2 \leq 100 N^{2/13}$. This is the specific contribution that needs the $L^\infty$-based norms.
    \item Case $\mathcal C_{\mathrm{low}}$, where $N_1 \leq \frac{1}{16} N_{\rm{max}}$ (which implies $N_1 = N_{\rm{min}}$ and $M_2, M_3 \approx N$).
\end{itemize}

\subsubsection{Case $\mathcal{C}_{\rm{bad}}$}
Let us recall that we include here all cases where $M_2, M_3 < 1/2$. In particular, we are under the situation where $|\theta_i - \theta_j | \les 1$ for all $i, j$. Moreover, using the same argument as in $\mathcal C_{\rm{high}}$, we just need to show \eqref{eq:sky}, which we recall below:  
\begin{equation} \label{eq:sky2}
	\Theta_\iota' := \frac{N_3^{2s} N_{\ast 3}^2}{N_1^{2\alpha} N_2^{2s}} 
	\sup_{\lambda, \tilde \eta}  
	\sum_{n_2} \int_{\xi} 
	\left|   \int_{\bar S^{n, \xi}_\iota}  
	h (\lambda  - \Delta)  |b_{\tilde \eta } (\zeta, m_2)|  d \bar \sigma_{n, \xi} \right|^2  \les N_1^{-\eps / 2}
\end{equation}

Now, since $M_2, M_3 \leq 1$, Lemma \ref{lemma:localization} gives us that $(\sqrt{3}\xi, n_2)$ and $(\sqrt{3}\zeta, m_2)$ live in circles of radius $O(1)$ determined by each other (by rotations). In particular, $| \bar S^{n, \xi} | \leq 1$, so we have that
\begin{align*}
	\Theta_\iota' &\les \frac{N_3^{2s} N_{\ast 3}^2}{N_1^{2\alpha} N_2^{2s}} 
	\sup_{\lambda, \tilde \eta} 
	\sum_{n_2} \int_{\xi} 
	\int_{\bar S^{n, \xi}_\iota}  
	h (\lambda  - \Delta)  |b_{\tilde \eta } (\zeta, m_2)|^2  d \bar \sigma_{n, \xi} \\
	&\les
	\frac{N_3^{2s} N_{\ast 3}^2}{N_1^{2\alpha} N_2^{2s}} 
	\sup_{\lambda, \tilde \eta} \sum_{m_2} \int_\zeta |b_{\tilde \eta } (\zeta, m_2)|^2 
	\int_{\bar S^{m, \zeta}_\iota}  
	h (\lambda  - \Delta)   d \bar \sigma_{n, \xi} \\
	&\les  
	\frac{N_{\ast 3}^2}{N^{2\alpha}} \sup_{\lambda, m, \zeta} 
	\int_{\bar S^{m, \zeta}_\iota}  
	h (\lambda  - \Delta)   d \bar \sigma_{n, \xi},
\end{align*}
where in the last line we also used $N_1\approx N_2\approx N_3 \approx N$ given that $M_2, M_3 < 1$.

Now, if $N_{\ast 3} = 1$, we can simply bound the later integral by $1$ using that $| \bar S^{m, \zeta}| \les 1$ since $(\sqrt{3} \xi, n_2)$ lives in a $O(1)$-sized ball determined by $\zeta, m$. Thus, we conclude \eqref{eq:sky2} easily in that case.

Thus, we can assume $N_{\ast 3} \neq 1$ and therefore we can apply Lemma \ref{lemma:integral}. We have that
\begin{equation*}
	\int_{\bar S^{m, \zeta}_\iota}  
	h (\lambda  - \Delta)   d \bar \sigma_{n, \xi} \les \sum_{n_2} 	\int_{\bar S^{m, \zeta, n_2}_\iota}  
	h (\lambda  - \Delta)   d \bar \sigma_{n, \xi} \les 
	\sum_{n_2} \frac{   \mathbbm{1}_{ \bar S^{m, \zeta, n_2}_\iota \neq \emptyset }  }{NM_3} \leq \frac{1}{N_{\ast 3}^{1/2}},
\end{equation*}
where in the last inequality we used $M_3 \gtrsim M_{3, \rm{min}} \approx N_{\ast 3}^{1/2} / N$ and that $n_2$ runs over at most $O(1)$ possible values given that $\zeta, m_2$ are fixed.

Thus, we get that
\begin{equation*}
	\Theta_\iota' \les N^{3/2 - 2\alpha}
\end{equation*}
which satisfies \eqref{eq:sky2} given that $\alpha > \frac34$.

\subsubsection{Case $\mathcal C_{\rm{high}}$}

Here, we study the case where $N_1 \geq \frac18 N_{\mathrm{max}}$. In this case, we can use deterministic bounds. From \eqref{eq:objective_def}, \eqref{eq:objective_mixed}, it suffices to show:
\begin{equation} \label{eq:sky}
	\Theta_\iota' := \frac{N_3^{2s} N_{\ast 3}^2}{N_1^{2\alpha} N_2^{2s}} 
	\sup_{\lambda, \tilde \eta}  
	\sum_{n_2} \int_{\xi} 
	\left|   \int_{\bar S^{n, \xi}_\iota}  
	h (\lambda  - \Delta)  |b_{\tilde \eta } (\zeta, m_2)|  d \bar \sigma_{n, \xi} \right|^2  \les N^{-\eps}
\end{equation}

\textbf{Generic case: $(M_3, N_{\ast 3}) \neq (M_{3, \rm{min}}, 1)$ and $M_2 \geq N^{-12/13}$. } \\

We note that $N_1 \approx N$ due to the hypothesis of $\mathcal C_{\rm{high}}$. Moreover, we note by Lemma \ref{lemma:localization} that $(\zeta, m_2)$ lies on a certain box of sizes $C \max \{ M_3 ,  M_2 \} \times C \max \{ M_3 ,  M_2 \}$ once $(\xi, n_2)$ are fixed. Due to $\mathcal C_{\rm{high}}$, we have $\max \{ M_2,  M_3\} \geq 1$, since cases with $\max \{ M_2,  M_3 \} < 1$ are included in $\mathcal C_{\rm{bad}}$

Since $(M_3, N_{\ast, 3}) \neq (M_{3, \rm{min}}, 1)$, we can apply Lemma \ref{lemma:integral}, and we have:
\begin{align} 
	\left| \int_{\bar S^{n, \xi}_\iota} h(\lambda - \Delta) | b_{\tilde \eta} (\zeta, m_2) | d\bar \sigma_{n, \xi}\right|^2  &\les \int_{\bar S^{n, \xi}_\iota }  | b_{\tilde \eta}(\zeta, m_2)|^2 h(\lambda - \Delta) d\bar \sigma_{n, \xi} \cdot \int_{\bar S^{n, \xi}_\iota } h(\lambda - \Delta) d\bar \sigma_{n, \xi} \nonumber \\
	&\les
	\int_{\bar S^{n, \xi}_\iota }  | b_{\tilde \eta}(\zeta, m_2)|^2 h(\lambda - \Delta) d\bar \sigma_{n, \xi} \sum_{m_2} \frac{ \mathbbm{1}_{\bar S^{n, \xi}_\iota \neq \emptyset } }{NM_3} \nonumber \\
	&\les
	\int_{\bar S^{n, \xi}_\iota }  | b_{\tilde \eta}(\zeta, m_2)|^2 h(\lambda - \Delta) d\bar \sigma_{n, \xi}  \min \left\{ \frac{ \max \{ M_2 , M_3 \}  }{ NM_3 }, \frac{N_2}{NM_3} \right\} \label{eq:pinkfloyd1},
\end{align}
where we use that $m_2$ ranges over at most $O(N_2)$ values (as well as over at most $\max \{ M_2, M_3 \}$ values). Interchanging the integrals and using the same argument (but with Corollary \ref{cor:integral} this time), we also have
\begin{equation}
	\sum_{n_2} \int_\xi \int_{\bar S_\iota^{n, \xi}} h(\lambda - \Delta) | b_{\tilde \eta} (\zeta, m_2) |^2 d\bar \sigma_{n, \xi} \les  \min \left\{ \frac{\max \{ M_2 , M_3 \} }{NM_2}, \frac{  N_3  }{NM_2} \right\} \label{eq:pinkfloyd2}
\end{equation}
where we also used that $\| b_{\tilde \eta} \|_{L^2} = 1$ for all $\tilde \eta$. Corollary \ref{cor:integral} also requires $(M_{2}, N_{\ast 2}) \neq (M_{2, \rm{min}}, 1)$ but this is ensured by $M_2 \geq N^{-23/24}$ (since $M_{2, \rm{min}} = N_{\ast 2}^{1/2}/N$)

Now, from \eqref{eq:pinkfloyd1}--\eqref{eq:pinkfloyd2}, picking the terms $\frac{N_2}{NM_3}$ and $\frac{\max \{ M_2 , M_3 \} }{NM_2}$ in the minima, we deduce
\begin{equation} \label{eq:pinkfloyd3}
	\Theta_\iota' \les \frac{N^{2s} N_{\ast 3}^2 }{N^{2\alpha}} N_2^{1-2s}  \cdot \frac{1}{N^2\min \{ M_2 , M_3\}} = N_2^{1-2s} \frac{ N^{2s-2\alpha} N_{\ast 3} }{N\min \{ M_2 , M_3\}}. 
\end{equation}
In the case $M_3 \leq M_2$, we use that $NM_3 \geq N_{\ast 3}^{1/2}$, and since $\alpha > \frac34$, we obtain the desired inequality.

In the case $M_2 \leq M_3$, we use that $NM_2 \geq N^{1/13}$, and provided that $\alpha > 25/26$, we obtain the desired bound for $s-1/2$, $\eps$ sufficiently small.

\textbf{Singular case: $(M_3, N_{\ast 3}) = (M_{3, \rm{min}}, 1)$} \\

The treatment is similar as in the previous case, but instead of using \eqref{eq:pinkfloyd1}, we use
\begin{equation*}
	\left| \int_{\bar S^{n, \xi}_\iota} h(\lambda - \Delta) | b_{\tilde \eta} (\zeta, m_2) | d\bar \sigma_{n, \xi} \right|^2  \les N_2^2 \int_{\bar S^{n, \xi}_\iota }  | b_{\tilde \eta}(\zeta, m_2)|^2 h(\lambda - \Delta) d\bar \sigma_{n, \xi} 
\end{equation*}
which follows trivially since $|\bar S^{n, \xi}_\iota | \les N_2^2$. Now, we continue using \eqref{eq:pinkfloyd2} with the bound $\frac{ \max \{ M_2 , M_3\} }{ NM_2 }$. Since $M_3 \ll 1$, we must have that $M_2 \geq 1$ (from $C_{\rm{high}}$ hypothesis), so that  $\frac{ \max \{ M_2 , M_3\} }{ NM_2 } = \frac{1}{N}$. Thus,
\begin{equation*}
	\Theta_\iota' \les \frac{N_3^{2s}}{N_1^{2\alpha} N_2^{2s}} \frac{1}{N} \cdot N_2^{2} \les N^{-2\alpha + 2s} N_2^{1-2s}
\end{equation*}
which concludes the case since $\alpha > s$ and $s-1/2$ is sufficiently small.

\textbf{Case $M_2 \leq N^{-12/13}$ and $N_2 \geq 100N^{2/13}$.} \\

It only remains to treat the case of $\mathcal C_{\rm{high}}$ where $M_2 \leq N^{-12/13}$. Here, we can assume that  $N_2 \geq N^{2/13}$, since otherwise the case is included in $C_{\rm{res}}$.

We start noting that $M_2 \geq M_{2, \rm{min}} = N_{\ast 2}^{1/2} /N$ so that the restriction $M_2 \geq N^{-12/13}$ implies that $N_{\ast 2}\leq N^{2/13}$. Since $\langle \zeta \rangle \leq 2N_{\ast 2} \leq 2N^{2/13}$ but $N_2 \geq 100N^{2/13}$, we have that $\langle m_2 \rangle \approx N_2$. That is, the frequency $(\zeta, m_2)$ is approximately vertical.

We also note that since $\theta_1^2 - \theta_3^2 \approx NM_2$:
\begin{equation} \label{eq:rolling}
	O(NM_2) = 3(\xi + \zeta)^2 + (n_2 + m_2)^2 - 3\xi^2 - n_2^2 = O(NN_{\ast 2}) + m_2(m_2 + 2n_2)
\end{equation}
In particular, since $N_{\ast 2} \geq 1 \gg M_2$, the error that domintaes is $O(NN_{\ast 2 })$ and we have that $m_2 (m_2 + 2n_2) = O(NN_{\ast 2 })$. 

Now, we combine that with the fact that $\langle m_2 \rangle \approx N_2$. We also assume $m_2 \neq 0$ here, deferring $m_2 = 0 $ to the end of the subsection. Since $m_2 \neq 0$:
\begin{equation}\label{eq:rolling2}
	m_2 + 2n_2 = O \left( \frac{N N_{\ast 2} }{N_2}\right)
\end{equation} 
In particular, we see that once $m_2$ is fixed, $n_2$ only ranges over an interval of size $\frac{N N_{\ast 2}}{N_2}$. We will use this to improve our bounds. 

We start using \eqref{eq:pinkfloyd1} without any modification. However, we note that we can improve \eqref{eq:pinkfloyd2} using the bound above for the quantity of possible values $n_2$ can take (when $m_2$ is fixed). We have
\begin{equation}
	\sum_{n_2} \int_\xi \int_{\bar S_\iota^{n, \xi}} h(\lambda - \Delta) | b_{\tilde \eta} (\zeta, m_2) |^2 d\bar \sigma_{n, \xi} \les \frac{N N_{\ast 2}}{N_2} \cdot \frac{1}{NM_2} \les \frac{N N_{\ast 2}^{1/2}}{N_2} \label{eq:beatles}
\end{equation} 
where in the last inequality we used that $NM_2\geq N_{\ast 2}^{1/2}$. Now, combining this with \eqref{eq:pinkfloyd1}, we see
\begin{equation*}
	\Theta_\iota' \les \left( \frac{N_3}{N_2} \right)^{2s} \frac{N_{\ast 3}^2 }{N^{2\alpha}} \frac{N N_{\ast 2}^{1/2} }{N_2}  \min \left\{ \frac{1}{ N}, \frac{N_2}{N M_3}\right\}  \les \frac{N^{2s+2-2\alpha}}{N_2^{2s}}  N_{\ast 2}^{1/2}  \min \left\{ \frac{1}{N_2}, \frac{1}{M_3} \right\}
\end{equation*}
Now, we note that $\max \{ N_2 , M_3 \} = N$. The reason is that if $N_2 \ll N_{\rm{max}}$, then $N_1 \geq 8N_2$ and we have $M_3 \gtrsim N$.

Combining that with $N_{2} \gtrsim N^{2/13}$, we get that $$ \Theta_\iota' \les \frac{N^{2s+1-2\alpha}}{N_2^{s}}   \les N^{2s+1-2\alpha-2/13}$$
so we see that the desired bound \eqref{eq:sky} is satisfied for $\alpha > 12/13$ provided that $s-1/2$ and $\eps$ are sufficiently small depending on $\alpha$.

Lastly, it remains to deal with the possibility of $m_2 = 0$. In that case, \eqref{eq:rolling} yields $O(NM_2) = (\xi + 2\zeta ) \xi$. Thus, one of the factors needs to be smaller than $(NM_2)^{1/2} \leq N^{1/26}$. Since $|\zeta|  \leq N^{2/13}$, we obtain $N_{\ast 3} \approx \langle  \xi \rangle \leq N^{2/13}$ independently of which $|\xi|$ or $| \xi + 2\zeta |$ is smaller. Using $N_{\ast 3} \leq N^{2/13}$ in \eqref{eq:pinkfloyd3}, we conclude our bound since $\alpha > \frac12 + \frac{1}{13}$, and $s-1/2, \eps$ are sufficiently small.

\subsubsection{Case $\mathcal C_{\rm{res}}$}

Since we have that $\max \{ M_2, M_3 \} \geq 1$ and $M_2 \leq N^{-12/13}$, we note that $M_3 \geq 1$.

We also notice that $M_2 \geq M_{2,\rm{min}} = N_{\ast 2}^{1/2}/N$, which implies that $N_{\ast 2} \leq N^{2/13}$.

Here, we use the $L^\infty$ part of our $Z^{s, b}$ norm. The reason why this is helpful is twofold. First of all, since $N_2 \leq N^{2/13}$ is small, we do not lose much from considering the $L^\infty$ part of the norm, even though it has the disadvantage of being in a lower regularity (yields worse control in terms of $N_2$). Secondly, we are in a situation where $M_2 \leq N^{-12/13}\ll 1$. This imposes a very small interval where $\zeta$ has to live (fixed $\xi, n, m_2$), in order to satisfy that condition. In particular, it is not difficult to show that $\zeta$ needs to live on an interval of size $\varepsilon = O(N^{-9/13})$. Therefore, we expect $L^\infty$ control to be much better than $L^2$ control, since a function with $L^2$-mass equal to 1 could still be of size $\varepsilon^{-1/2}$ on that interval, if it is exactly concentrated there, while a function with $L^\infty$ norm equal to $1$ is simply bounded by $1$.



Instead of bounding the simplified quantity $\Theta_\iota'$, we will bound $\Theta_\iota$ directly (which we recall from \eqref{eq:Theta}). We start by using $L^\infty$ bound on $z$ and $\hat v_{\rm{ren}}$ and the fact that $|g_k^\omega| \les N_1^{\eps/2}$ for all $k$ almost surely (due to the Borel-Cantelli argument given just before \eqref{eq:marcus2}). Recalling $\hat v_{\rm{ren}} = \psi (\tilde \eta) b_{\tilde \eta} (\zeta, m_2)$, we can bound:
\begin{align} \label{eq:maquiavelo}
	\Theta_\iota &\les  \frac{N_3^{2s} N_{\ast 3}^2}{N_1^{2\alpha - \eps/2} N_2^{2s}} \sum_{n_2} \int_\xi \int_{\tilde \tau} \frac{1}{\langle \tilde \tau \rangle^{1-4\delta}} \left|  \int_{\bar S^{n, \xi}_\iota} \left( \int_{\tilde \eta}  \frac{  h(-\tilde \tau - \tilde \eta - \Delta)  }{ \langle \tilde \eta \rangle^{2b} } \right)^{1/2} \| \hat v_{\rm{ren}} (\zeta, m_2, \eta) \|_{L^2_\eta}  d\bar \sigma_{n, \xi} \right|^2
	\end{align}


Now, recall that estimate \eqref{eq:marcus}, coming from $L^\infty$ part of the $Z^{s, b}$. Using it, we have
\begin{align} 
	\Theta_\iota &\les  N^{\eps/2} \frac{N_3^{2s} N_{\ast 3}^2}{N_1^{2\alpha}} \sum_{n_2} \int_\xi \int_{\tilde \tau} \frac{1}{\langle \tilde \tau \rangle^{1-4\delta}} \left|  \int_{\bar S^{n, \xi}_\iota} \frac{1}{\langle \tilde \tau + \Delta \rangle^b} \cdot \frac{| \zeta |}{\langle \zeta \rangle} \frac{ \| \hat v_{\rm{ren}} (\zeta, m_2, \eta)  \|_{L^2_\eta} \cdot \langle \zeta \rangle }{N_2^s | \zeta |}  d\bar \sigma_{n, \xi} \right|^2 \nonumber \\
	&\les \| v \|_{Z^{s, b}}^2
	N^{\eps/2} \cdot N^{2s+2-2\alpha} \sum_{n_2} \int_\xi   \int_{\bar S^{n, \xi}_\iota} \int_{\tilde \tau} \frac{1}{\langle \tilde \tau \rangle^{1-4\delta}} \frac{1}{\langle \tilde \tau + \Delta \rangle^{2b}} \cdot \frac{| \zeta |^2}{\langle \zeta \rangle^2} d\bar \sigma_{n, \xi} \cdot | \bar S^{n, \xi}| \nonumber \\
		&\les \| v \|_{Z^{s, b}}^2
	N^{\eps/2} \cdot N^{2s+2-2\alpha} \sum_{n_2, j} \int_\xi   \int_{\bar S^{n, \xi}_{\iota, j}}  \frac{1}{\langle  \Delta \rangle^{1-4\delta}} \cdot \frac{| \zeta |^2}{\langle \zeta \rangle^2}  d\bar \sigma_{n, \xi} \cdot |\bar S_{\iota, j}^{n, \xi}| \label{eq:ringo} 
\end{align}
We have used that $N_1 \approx N_3 \approx N$, given that $N_2 \leq N^{2/13}$. We have also introduced the variable $j \in  \{1, 2\}$ and split $\bar S_{\iota}^{\xi, n} = \bar S_{\iota, 1}^{\xi, n} \cup \bar S_{\iota, 2}^{\xi, n}$ regarding the value of $\zeta$. We will specify this division later. 

Note that
\begin{equation*}
	O(NM_2) = 3 (\xi + \zeta)^2 + (n_2 + m_2)^2 - 3 \xi^2 - n_2^2 = (\sqrt{3} \xi, n_2)\cdot (\sqrt{3} \zeta, m_2) + O(N_2^2) 
\end{equation*}
A first consequence of this is that
\begin{equation} \label{eq:artreides1}
	| \bar S^{n, \xi}_{\iota} | \les  M_2 N_2 + \frac{N_2^3}{N}
	\end{equation}
since $(\sqrt{3} \zeta, m_2)$ is inside a box of size $CN_2 \times CN_2$ and moreover in the line of width $\frac{NM_2 + N_2^2}{N}$ determined by the perpendicularity condition with $(\sqrt{3} \xi, n_2)$.

A second consequence is that the computation of $\Delta$ yields
\begin{align}
	\Delta &= \xi^3 + \xi n_2^2 + \zeta^3 + \zeta m_2^2 - (\zeta + \xi)^3 - (\zeta + \xi)(n_2+m_2)^2 \nonumber \\
	&= -3\xi^2 \zeta - 2n_2m_2\xi -\zeta n_2^2 + O(NN_2^2) = -3\xi^2 \zeta + 6 \xi^2 \zeta - \zeta n_2^2 + O(NN_2^2 + N^2M_2) \nonumber \\ 
	&= \zeta (\sqrt{3} \xi - n_2) (\sqrt{3} \xi + n_2)  + O(NN_2^2 + N^2M_2). \label{eq:artreides2}
\end{align}
where we also used that $3\xi\zeta + n_2m_2 = O(N_2^2 + NM_2)$ from the previous computation. For simplicity, we define $E = N_2^2 + NM_2$, so that \eqref{eq:artreides1}--\eqref{eq:artreides2} simply read
\begin{equation} \label{eq:artreides3}
	|\bar S^{n, \xi}_{\iota} | \les \frac{EN_2}{N}, \qquad \mbox{ and } \qquad \Delta = \zeta (\sqrt{3} \xi - n_2)(\sqrt{3} \xi + n_2) +O(NE).
\end{equation}
Now, we specify the cases $j=1$ and $j=2$. Let $C$ the implicit constant on $O(NE)$ in \eqref{eq:artreides3}. \begin{itemize}
\item The  case $j=1$ corresponds to $|\zeta |\cdot |\sqrt{3} \xi - n_2| \cdot |\sqrt{3} \xi + n_2| \geq 2CNE$, so that we have $| \Delta | \approx |\zeta |\cdot |\sqrt{3} \xi - n_2| \cdot |\sqrt{3} \xi + n_2|$. 
\item The case $j=2$ corresponds to have that $|\zeta |\cdot |\sqrt{3} \xi - n_2| \cdot |\sqrt{3} \xi + n_2| < 2CNE$.
\end{itemize}

\textbf{Case $j=1$}\\

Going back to \eqref{eq:ringo} and using \eqref{eq:artreides3} we see that
\begin{align*}
	\Theta_{\iota, 1} &\les \| v \|_{Z^{s, b}}^2
	N^{\eps/2} \cdot N^{2s+2-2\alpha} \sum_{n_2} \int_\xi \int_{\bar S_\iota^{n, \xi} }\frac{| \zeta |^2 / \langle \zeta \rangle^2}{  | \zeta |^{1-4\delta} \cdot |\sqrt{3} \xi + n_2 |^{1-4\delta} \cdot |\sqrt{3} \xi - n_2 |^{1-4\delta} }    |\bar S_{\iota}^{n, \xi}| \\
    &\les \| v \|_{Z^{s, b}}^2
	N^{\eps/2} \cdot N^{2s+2-2\alpha} \sum_{n_2} \int_\xi \frac{\mathbbm{1}_{|\xi|, |n_2| \les N}}{|\sqrt{3} \xi + n_2 |^{1-4\delta} \cdot |\sqrt{3} \xi - n_2 |^{1-4\delta} }    |\bar S_{\iota}^{n, \xi}|^2 \\
	&\les_\delta 
	\| v \|_{Z^{s, b}}^2
	N^{\eps/2} \cdot N^{2s+2-2\alpha} \cdot N^{10\delta} \cdot \frac{N_2^2 E^2}{N^2} \les 	\| v \|_{Z^{s, b}}^2 N^{\eps / 2 + 10\delta} N^{2s - 2\alpha} E^2 N_2^2
	\end{align*}
    Since $E = N_2^2 + NM_2 \les N^{4/13}$, we have $E^2N_2^2 \les N^{12/13}$, and we conclude the desired bound as long as $\alpha > \frac{25}{26}$, for $\delta, \eps, s-1/2$ sufficiently small. \\
	
	\textbf{Case $j=2$} \\
	Let $\Gamma \leq 1$ denote the dyadic value of $|\zeta|$. That is, if $| \zeta | \leq 1$ we denote $\Gamma$ to be the (negative) power of $2$ closest to $| \zeta |$, while for $| \zeta | \geq 1$ we simply let $\Gamma = 1$.
	
	In this case, we have $|\sqrt{3} \xi - n_2 | \cdot | \sqrt{3} \xi + n_2 | \les EN/\Gamma$. Since $(\xi, n_2)$ is of size $N$, one of those two factors needs to be $\gtrsim N$. Thus, let us assume without loss of generality $|\sqrt{3} \xi -n_2| \les E/ \Gamma$. We note that the measure of the set of $(\xi, n_2)$ satisfying that condition is $O(EN/\Gamma)$. Therefore, going back to \eqref{eq:ringo} and using \eqref{eq:artreides3} and $\langle \Delta \rangle \les N^3$, we see that
\begin{align*}
	\Theta_{\iota, 2} &\les\| v \|_{Z^{s, b}}^2
	N^{\eps/2 + 15\delta} \cdot N^{2s+2-2\alpha} \sum_{n_2, \Gamma} \int_\xi  \Gamma^2  \int_{\bar S^{n, \xi}_{\iota, j}}  \frac{1}{\langle  \Delta \rangle^{1+\delta}}  d\bar \sigma_{n, \xi} \cdot \frac{EN_2}{N} 
    \end{align*}
Noting that $M_3 \approx N$ (due to $N_2 \leq N^{2/13}$), Lemma \ref{lemma:integral} yields
\begin{equation*}
\sum_{n_2, \Gamma} \int_\xi  \Gamma^2  \sum_{m_2} \int_{\bar S^{n, \xi, m_2}_{\iota, j}}  \frac{1}{\langle  \Delta \rangle^{1+\delta}}  d\bar \sigma_{n, \xi} \les \sum_{n_2, \Gamma} \int_\xi  \Gamma^2 \frac{N_2}{N^2} \les \sum_{\Gamma} \frac{EN}{\Gamma} \Gamma^2 \frac{N_2}{N^2} \les \frac{E N_2}{N}.
\end{equation*}
Therefore, we get that
    \begin{align*} 
	\Theta_{\iota, 2} &\les
	\| v \|_{Z^{s, b}}^2
	N^{\eps/2 + 15\delta} \cdot N^{2s-2\alpha}  E^2 N_2^2 
	\end{align*}
	
Since $E = N_2^2 + NM_2 \les N^{4/13}$, we have $E^2N_2^2 \les N^{12/13}$, and we conclude the desired bound as long as $\alpha > \frac{25}{26}$, for $\delta, \eps, s-1/2$ sufficiently small. \\

\subsubsection{ Case $\mathcal{C}_\mathrm{low}$ }

Let us recall that this case corresponds to $N_1 \leq \frac{1}{8}N$. In such case, we have that $\theta_2, \theta_3 \geq 2 \theta_1$ and therefore $M_2, M_3 \gtrsim N$. We proceed analogously as in \eqref{eq:pinkfloyd1}--\eqref{eq:pinkfloyd2}, but in this case we have that $m_2$ ranges over $O(N_1)$ values when $n_2$ is fixed, since $|k_2| = |m_2+n_2| \les N_1$. Reciprocally, $n_2$ ranges over $O(N_1)$ values for $m_2$ fixed. Thus,  \eqref{eq:pinkfloyd1}--\eqref{eq:pinkfloyd2}, in this case become
\begin{equation*}
		\left| \int_{\bar S^{n, \xi}_\iota} h(\lambda - \Delta) | b_{\tilde \eta} (\zeta, m_2) | d\bar \sigma_{n, \xi}\right|^2 \les
		\int_{\bar S^{n, \xi}_\iota }  | b_{\tilde \eta}(\zeta, m_2)|^2 h(\lambda - \Delta) d\bar \sigma_{n, \xi} \cdot \frac{N_1}{N^2} 
\end{equation*}
and
\begin{equation*}
		\sum_{n_2} \int_\xi \int_{\bar S_\iota^{n, \xi}} h(\lambda - \Delta) | b_{\tilde \eta} (\zeta, m_2) |^2 d\bar \sigma_{n, \xi} \les  \frac{N_1}{N^2}.
\end{equation*}
Therefore, we obtain
\begin{equation*}
	\Theta_\iota' \les \frac{N^{2s} N_{\ast 3}^2 }{N^{2s} N_1^{2\alpha}} \cdot \frac{N_1^2}{N^4} \leq \frac{N_1^{-2\alpha+1}}{N},
\end{equation*}
concluding this case since $\alpha > 1/2$.

\subsection{Proof of Proposition \ref{prop:Linfty}}
\label{subsec:Linfty}

In order to prove the $L^\infty$ estimates it will be useful to have some estimates on $A_{\xi, n_2, \tau}$, which was introduced in the proof of Proposition \ref{prop:refined}. The setting here has some minor modifications, so we give a short proof of the bound again, at the risk of some small repetition. We fix some dyadic numbers $N_i, L_i, M_3$, and recall the set $A_{\xi, n_2, \tau}$ is defined as the set of $(\nu, k_2, \mu)$ such that, for those $\xi, n_2, \tau$, the conditions imposed by $N_i$, $L_i$ and $M_3$ are satisfied. Then, we have the following bound
\begin{lemma} \label{lemma:Abounds} We have that
	\begin{equation} \label{eq:firstAbound}
	| A_{n, \xi, \tau} | \leq L_1 L_2 \frac{\min \{ N_1 , N_2\} }{ N_{\rm{max}} M_3 },
	\end{equation}
	which simplifies to 
	\begin{equation} \label{eq:secondAbound}
		| A_{n, \xi, \tau} | \leq L_1 L_2 \frac{  \min \{ N_1 , N_2\} }{N_{\rm{max}}^2 }
	\end{equation}
	in the case $N_1 \geq 4 N_2$ or viceversa. 
	
	We also have the trivial bound
    \begin{equation} \label{eq:trivialAbound}
| A_{n, \xi, \tau} | \leq \min \{ L_1, L_2 \} \min \{ N_1, N_2 \}^2,
    \end{equation}
    and the following alternative bound
	\begin{equation} \label{eq:thirdAbound}
		| A_{n, \xi, \tau} | \leq \min \{ L_1, L_2 \} \frac{M_3}{N_3} N_{\rm{max}}^2
	\end{equation}
	Moreover, in this last bound, we can include in $A_{n, \xi, \tau}$ all the $(\nu, k_2, \mu)$ such that $|\theta_1 - \theta_2| \leq M_3$, instead of just the ones such that $|\theta_1 - \theta_2 | \approx M_3$.
\end{lemma} 
 \begin{proof} Since the proof of \eqref{eq:firstAbound} is very similar to Proposition \ref{prop:refined}, we just perform the main steps (see that Proposition for a detailed treatment). First, $|A_{n, \xi, \tau} | \leq \min \{ L_1, L_2\} | \tilde A_{n, \xi, \tau, \mu} |$, where now we have also fixed $\mu$ and $\tilde A$ is just the set of the possible $(\nu, k_2)$. Then, we observe $k_2$ ranges over $\min \{ N_1, N_2 \}$ possible values due to the restrictions $|k_2| \leq N_1$ and $|-n_2-k_2| = |m_2| \leq N_2$. Then, one looks at the size of the interval of possible $\nu$, when $n_2, \xi, k_2$ are fixed, taking into account that $\tilde \Delta := \Delta - \phi (\xi, n_2) = -\tilde \eta - \tilde \mu - \tau$ ranges over an interval of size $4\max \{ L_1, L_2 \}$ (because $\tau$ is fixed and $\tilde \mu, \tilde \eta$ are smaller than $L_1, L_2$ respectively). A computation (done in Proposition \ref{prop:refined}) shows that $|\p_\nu \tilde \Delta| \gtrsim NM_3$, so therefore the measure of the set of possible $\nu$ is at most $\frac{ \max \{ L_1, L_2 \} }{NM_3}$. 
 
 The bound \eqref{eq:secondAbound} follows from \eqref{eq:firstAbound} noting that if $N_1 \geq 4N_2$ or viceversa, then $M_3 \approx N_{\rm{max}}$.

 The bound \eqref{eq:trivialAbound} follows from $| \tilde A_{n, \xi, \tau, \mu } | \leq \min \{ N_1, N_2 \}^2$, which itself follows $\nu$ and $k_2$ ranging over at most $O(\min \{  N_1, N_2 \} )$ values. To see that, notice $| \nu | \leq N_1$ and $| \nu + \xi | \leq N_2$, and similarly for $k_2$.
 	
Lastly, we show \eqref{eq:thirdAbound}. Here, we note
 	\begin{equation*}
 	O(N_{\rm{max}} M_3) = 	(3\nu^2 + k_2^2) - (3(-\xi - \nu)^2 + (-n_2-k_2)^2 ) = (6 \xi, 2n_2) \cdot (\nu, k_2) - (3\xi^2 + n_2^2).
 	\end{equation*}
 	Thus, dividing by $2\theta_3 = 2\sqrt{3\xi^2 + n_2^2} \approx N_3$, we have that
 	\begin{equation*}
 		\vec{d} \cdot (\zeta, m_2) = \frac{\theta_3}{2} + O \left( \frac{N_{\rm{max}} M_3 }{N_3} \right)
 	\end{equation*}
 	where $\vec{d}$ is a unitary vector in the direction of $(3\xi, n_2)$. For $\xi, n_2$ fixed, we see that $(\zeta, m_2)$ lies in some line perpendicular to $\vec{d}$ of width $\frac{M_3 N_{\rm{max}} }{N_3}$. Since $(\zeta ,m_2)$ also lies on a big box of size $N_{\rm{max}} \times N_{\rm{max}}$, we can bound $|\tilde A_{\xi, n_2, \tau, \mu}| \les N_{\rm{max}}^2 \frac{M_3}{N_3} $ just because the intersection of a box of size $\ell \times \ell$ with a line of width $w$ has area $O(w \ell)$. We also remark that the proof just used $| \theta_1 - \theta_2 | \les M_3$, so this bound works when bounding all such frequencies, and not just $| \theta_1 - \theta_2 | \approx M_3$.
 \end{proof}

 \subsubsection{Proof of \eqref{eq:Linfty_1}}
Let us recall estimate \eqref{eq:Linfty_1}:
\begin{equation*}
\Lambda_1 = 	\| \langle \tau - \phi (\xi, n_2) \rangle^{-\frac12+2\delta} \widehat{ \langle \p_x \rangle (u v ) } \|_{L^\infty_{\xi, n_2} (L^2_\tau)} \les \| u \|_{Z^{s, \frac12 + \delta}} \| v \|_{Z^{s, \frac12 + \delta}}.
\end{equation*}

We partition the left term with respect to dyadic frequencies $N_i, N_{\ast i}, L_i$ at the cost of at most introducing factors of $\log (N_{\rm{max}})$ (see the discussion after \eqref{eq:maindet11} for a precise description of this procedure). We decompose in two cases: \begin{enumerate}
	\item The case where we either have comparable frequencies ($N_1\leq 4 N_2 \leq 16N_1$) or $L_{\rm{max}} \geq 1000N^3$.
	\item The case where $L_{\rm{max}} \les N^3$, and without loss of generality, $N_1 \geq 4N_2$.
\end{enumerate} 

In the first case, we renormalize $$\hat u_{\rm{ren}} (\nu, k_2, \mu) = \langle (\nu, k_2) \rangle^{-s} \cdot \langle \mu - \phi (\nu, k_2) \rangle^{-\frac12 - \delta} \hat u (\nu, k_2, \mu)$$ and analogously $\hat v_{\rm{ren}}$. In particular, we have $\| \hat u_{\rm{ren}} \|_{L^2} = \|  u \|_{X^{s, b}} \leq \|  u \|_{Z^{s, b}}$. Letting $\Lambda_{1, \iota}^{(1)}$ be the contribution to $\Lambda_1$ from this first case, with dyadic numbers $\iota = (N_i, N_{\ast i}, L_i)_{i=1}^3$ we have:
\begin{equation*}
	\Lambda_{1, \iota}^{(1)} \les \frac{N_{3 \ast }}{L_3^{\frac12-2\delta} N_1^s N_2^s L_1^{\frac12 + \delta} L_2^{\frac12 + \delta} } \sup_{n, \xi}  \left\| \int_\mu  \sum_{k_2} \int_\nu  | \hat u_{\rm{ren}} (\nu, k_2, \mu) \hat v_{\rm{ren}} (-\xi-\nu, -n_2-k_2, -\tau - \mu) |     \right\|_{L^2_\tau}
\end{equation*}
We note that $\tau$ ranges over an interval of size $O(L_3)$ since $| \tau - \phi (\xi, n_2) | \approx L_3$. Therefore, doing Cauchy-Schwarz, we see that
\begin{equation*}
	L_{\rm{max}}^{\delta} \Lambda_{1, \iota}^{(1)} \les \frac{N_{3 \ast }}{N_1^s N_2^s L_1^{1/2} L_2^{1/2} L_3^{1/2 - 3\delta}} L_3^{1/2}  \| \hat u_{\rm{ren}} \ast \hat v_{\rm{ren}} \|_{L^\infty_{n, \xi, \tau}}
    \les 
    \frac{N_{3 \ast }}{N_1^s N_2^s L_1^{1/2} L_2^{1/2} } L_3^{3\delta}  \| \hat u_{\rm{ren}} \|_{L^2} \| \hat v_{\rm{ren}} \|_{L^2}
\end{equation*}
 Moreover, we notice that $|\tilde \tau + \tilde \eta + \tilde \mu |=| \Delta |  \les N^3$, so if $L_1 \gg N^3$, then $\max \{ L_2, L_3\} \gtrsim L_1$. Therefore, we always have that $$ \frac{L_3^{3\delta}}{L_2^{1/2} L_1^{1/2}} \les \frac{N^{9\delta}}{L_{\rm{med}}^{1/2 - 3\delta} }. $$ 
In the subcase $N_1 \approx N_2$ we can bound
\begin{equation*}
	L_{\rm{max}}^{\delta} \Lambda_{1, \iota}^{(1)} \les \frac{N^{1+9\delta}}{N^{2s} }  \les N^{-\delta} \| u \|_{Z^{s, b}} \| v \|_{Z^{s, b}}
\end{equation*}
taking $\delta$ sufficiently small depending on $s-1/2 > 0$. With respect to the case $L_{\rm{max}} \geq 1000N^3$, we note that $| \tilde \mu + \tilde \nu + \tilde \tau | = | \Delta | \leq 100N^2 $implies $L_{\rm{med}} \approx L_{\rm{max}}$. Therefore, we have \begin{equation*}
	L_{\rm{max}}^{\delta} \Lambda_{1, \iota}^{(1)} \les \frac{N_{\ast 3}}{N_1^s N_2^s} \frac{N^{9\delta}}{L_{\rm{max}}^{1/2-3\delta}} \| u \|_{Z^{s, b}} \| v \|_{Z^{s, b}} 
    \les \frac{N_{\ast 3}}{N^s} \frac{N^{18\delta}}{N^{3/2}} \| u \|_{Z^{s, b}} \| v \|_{Z^{s, b}} \les N^{-\delta} \| u \|_{Z^{s, b}} \| v \|_{Z^{s, b}}
\end{equation*}
for $\delta$ sufficiently small.

Now, we deal with the second case where $N_1 \geq 4N_2$ and $L_{\rm{max}} \les N^3$. We renormalize $u$ in the same way as in the previous case, but now we renormalize $v$ according to the $L^\infty$ part of the $Z^{s, b}$ norm, that is $v_{\rm{ren}} (\zeta, m_2, \eta) = \langle \eta - \phi (\zeta, m_2) \rangle^{-\frac12 - \delta} \hat v (\zeta, m_2, \eta)$. In particular, we have that $\| \hat v_{\rm{ren}} \|_{L^\infty_{\zeta, m_2} (L^2_\eta)} \les \| v \|_{Z^{s, b}}$. 

We recall $A_{\xi, n_2,  \tau}$ is the set of $(\zeta, m_2, \eta)$ described by $\xi, n_2, \tau$ and our dyadic numbers $N_i$, $M_i$, $L_i$, $N_{\ast i}$. Letting $\Lambda_{1, \iota}^{(2)}$ correspond to the contribution of our second case to $\Lambda_1$, we have:
\begin{align*}
L_{\rm{max}}^{\delta} \Lambda_{1, \iota}^{(2)} &\les \frac{N_{3 \ast}}{L_3^{1/2-3\delta} N_1^s L_1^{1/2} L_2^{1/2} } \sup_{n, \xi} \left( \int_{\tau} \left|  \sum_{k_2} \int_\nu \int_\mu | \hat u_{\rm{ren}} (\nu, k_2, \mu) \hat v_{\rm{ren}} (-\xi-\nu, -n_2-k_2, -\tau - \mu)  | \right|^2  \right)^{1/2}
\\
&\les
  \sup_{n, \xi, \tau'}  \frac{N_{3 \ast}}{L_3^{1/2-3\delta}} \frac{ |A_{\xi, n_2, \tau'} |^{1/2} }{N_1^s L_1^{1/2} L_2^{1/2}} \left( \sum_{k_2} \int_\nu \int_{\tau} \int_{\mu } | \hat u_{\rm{ren}} (\nu, k_2, \mu) |^2 | \hat v_{\rm{ren}} (-\xi-\nu, -n_2-k_2, -\tau - \mu) |^2 \right)^{1/2} \\
  &\les 
  \sup_{n, \xi, \tau'}  N_{3 \ast}^{1-s} L_3^{3\delta} \frac{ |A_{\xi, n_2, \tau'} |^{1/2} }{L_1^{1/2} L_2^{1/2} L_3^{1/2}} \| \hat v_{\rm{ren}} \|_{L^\infty_{\zeta, m_2} (L^2_\eta)} \| \hat u_{\rm{ren}} \|_{L^2}
\end{align*} 
Now, we bound $|A_{\xi, n_2, \tau'}|$ using \eqref{eq:secondAbound}. Combining that with $N_1 \gtrsim N_{\rm{max}}$ from $N_1 \geq 4N_2$, we get that
\begin{align*}
	L_{\rm{max}}^{\delta} \Lambda_{1, \iota}^{(2)} &\les  \sup_{\xi, n_2} \frac{N_{3\ast}^{1-s}}{L_1^{1/2} L_2^{1/2} L_3^{1/2}} L_{\rm{max}}^{3\delta} \frac{L_1^{1/2} L_{2}^{1/2}  }{ N^{1/2} } \|  u \|_{Z^{s, b}} \| v \|_{Z^{s, b}} \\
	&\les  N^{9\delta} \cdot N^{1/2-s} \|  u \|_{Z^{s, b}} \|  v \|_{Z^{s, b}} \les N^{-\delta} \|  u \|_{Z^{s, b}} \|  v \|_{Z^{s, b}}
\end{align*}
where we used that $L_{\rm{max}} \les N^3$ and that $\delta$ is sufficiently small with respect to $s-1/2 > 0$. 

Thus, we have shown that for all dyadic partitions $\iota$, we have $\Lambda_{1, \iota} \les N^{-\delta} L_{\rm{max}}^{-\delta} \| u \|_{Z^{s, b}} \| v \|_{Z^{s, b}}$. Summing over all the partitions, this concludes the proof of \eqref{eq:Linfty_1}, since $\sum_{N_i, L_i} L_{\rm{max}}^{-\delta} N^{-\delta} \les 1$.

\subsubsection{Proof of \eqref{eq:Linfty_2}}

Let us recall the estimate \eqref{eq:Linfty_2}:
\begin{equation*}
	\Lambda_{2} := \left\| \langle \tau - \phi (\xi, n_2) \rangle^{-\frac12 +2\delta} \mathcal F_{x, t} \left[ \langle  \p_{x_1} \rangle (v  \chi (t) S(t)u_{0}^\omega )  \right] \right\|_{L^{\infty}} \les \| v \|_{Z^{s, b}}\qquad \forall v \in Z^{s, b}.
\end{equation*}
Let us start noticing that $\mathcal F_{x, t} [\chi (t) S(t) u_0^\omega ] (\nu, k_2, \mu) =\hat u_0^\omega (\nu, k_2) \hat \chi (\mu - \phi (\nu, k_2)).$

We follow a similar approach to the one in \eqref{eq:Linfty_1}. In particular, we define $$\hat v_{\rm{ren}} (\zeta, m_2, \eta) = \langle m \rangle^{-s} \langle \eta - \phi (\zeta, m_2) \rangle^{-\frac12 - \delta} \hat v(\zeta, m_2, \eta)$$
and
$$\hat u_{\rm{ren}} (\nu, k_2, \mu) = \langle k \rangle^{-\alpha} \hat \chi (\mu - \phi (\nu, k_2)) \hat u_0^\omega (\nu, k_2).$$
We notice $$\| \hat v_{\rm{ren}} \|_{L^2} \les \| v \|_{X^{s, b}}, \qquad \mbox{ and }\qquad \| \langle \mu - \phi (\nu, k_2) \rangle^a \hat u_{\rm{ren}} (\nu, k_2, \mu) \|_{L^\infty_{\nu, k_2} (L^2_\mu)} \approx \| \langle k \rangle^{-\alpha} \hat u_0^\omega \|_{L^\infty} \les 1,$$ for any exponent $a$, given that $\hat \chi$ decays exponentially fast. The fact that we can choose any $a > 0$ corresponds to the fact that linear solutions with initial data on $H^s$ remain unaffected by the $b$ exponent on $X^{s, b}$ spaces. In other words, the weight $\langle \mu - \phi (\nu, k_2) \rangle$ is inocuous, since they are concentrated precisely on the set $\mu - \phi (\nu, k_2) = 0$.


We start the bounds very exactly in the same way as \eqref{eq:Linfty_1}, case 2: we do a dyadic decomposition in frequencies and a $L^\infty-L^2$ H\"older type approach. We have
\begin{align} \label{eq:cortisol}
	L_{\rm{max}}^{\delta} \Lambda_{2, \iota} &\les \sup_{n, \xi, \tau'} N_{3\ast} \frac{L_3^{3\delta}}{L_3^{1/2}} \frac{|A_{\xi, n_2, \tau'}|^{1/2}}{N_1^\alpha N_2^s L_1^{1/2} L_2^{1/2}} \| \langle \mu - \phi (\nu, k_2) \rangle^{1/2 + \delta} \hat u_{\rm{ren}} \|_{L^\infty_{\nu, k_2} (L^2_\mu)} \| \hat v_{\rm{ren}} \|_{L_2} 
\end{align}
where the only difference with respect to \eqref{eq:Linfty_1} case 2 is the different exponents in $N_1$ and $N_2$ due to the different renormalizations of $u, v$. 

We first consider the case $L_{\rm{max}} \geq N_{\rm{max}}^2 $, where we use \eqref{eq:trivialAbound} $|A_{\xi, n_2, \tau'}| \leq L_{\rm{med}} \min \{ N_1, N_2 \}^2$. Defining $\Lambda_{2, \iota}^{(1)}$ to be this contribution to $\Lambda_2$, we conclude that $$ L_{\rm{max}}^\delta \Lambda_{2, \iota}^{(1)} \les \frac{N_{3\ast}}{L_{\rm{max}}^{1/2-3\delta}} \frac{\min \{ N_1, N_2\} }{N_1^\alpha N_2^s} \frac{L_{\rm{med}}^{1/2}}{L_{\rm{min}}^{1/2}L_{\rm{med}}^{1/2}} \| \hat u_{\rm{ren}} \|_{L^\infty} \| \hat v \|_{L_2} \les \| \hat u_{\rm{ren}} \|_{L^\infty} \| \hat v \|_{L_2}, $$ since $\alpha, s>1/2$ and $\delta$ is sufficiently small depending on $\alpha, s$.

The main case is $L_{\rm{max}} \leq N^2$, which we divide again in two cases. If $N_1 \geq 4N_2$ or viceversa, we can use \eqref{eq:secondAbound} to estimate $|A_{\xi, n_2, \tau'}|$, and together with $L_{\rm{max}} \leq N_{\rm{max}}^2$, \eqref{eq:cortisol} yields:
\begin{equation*}
L_{\rm{max}}^{\delta} \Lambda_{2, \iota}^{(2)} \les  \sup_{n, \xi, \tau'} N_{3\ast} N_{\rm{max}}^{6\delta} \frac{\min \{ N_1, N_2\}^{1/2} }{ N_{\rm{max}} N_1^\alpha N_2^s} \| v\|_{Z^{s, b}}
\les N_{3\ast} N_{\rm{max}}^{6\delta} \frac{1 }{ N_{\rm{max}}^{3/2}} \| v\|_{Z^{s, b}} \les N^{-\delta} \| v \|_{Z^{s, b}}
\end{equation*}  
using again $\alpha, s > 1/2$.

Lastly, we have the case $L_{\rm{max}}\leq N^2$ with $N_1 \leq 4N_2 \leq 16N_1$. In this case, we will combine \eqref{eq:firstAbound} with \eqref{eq:thirdAbound}. In particular, we need to split the estimate further with respect to the dyadic number $M_3$, and we use \eqref{eq:firstAbound} if $M_3\geq N_{\rm{max}}^{-1/2}$ and \eqref{eq:thirdAbound} for $M_3 = N_{\rm{max}}^{-1/2}$ and lower. Using $N_1 \approx N_2 \approx N_{\rm{max}}$ we can see that both bounds for $M_3\leq N_{\rm{max}}^{-1/2}$ and $M_3 \geq N_{\rm{max}}^{-1/2}$ imply:
\begin{equation*}
	|A_{\xi, n_2, \tau'} | \leq L_{\rm{max}} L_{\rm{med}} \frac{N_{\rm{max}}^{3/2}}{N_{3 \ast}}.
\end{equation*}
 Thus, denoting by $\Lambda_{2, \iota}^{(3)}$ this last contribution to $\Lambda_{2}$, we have
\begin{align*}
	L_{\rm{max}}^{\delta} \Lambda_{2, \iota}^{(3)} \les   \sup_{n, \xi, \tau'} N_{3\ast} N_{\rm{max}}^{6\delta} \frac{1 }{N_1^\alpha N_2^s} \frac{N_{\rm{max}}^{3/4}}{N_{3 \ast}^{1/2}}  \|  v \|_{Z^{s, b}} \leq 
	N^{5/4 + 6\delta - \alpha - s}_{\rm{max}}   \|  v \|_{Z^{s, b}},
\end{align*}
so we conclude $\Lambda_{2, \iota} \les N^{-\delta} L_{\rm{max}}^{-\delta} \| v \|_{Z^{s, b}}$ for $\alpha > 3/4$, using that $s > 1/2$ and $\delta$ is sufficiently small. Thus and \eqref{eq:Linfty_2} follows by summing the $\Lambda_{2, \iota}$ and noting that $\sum_\iota N^{-\delta} L_{\rm{max}}^{-\delta} \les 1$.

\subsubsection{Proof of \eqref{eq:Linfty_3}}
Let us recall estimate \eqref{eq:Linfty_3}:
\begin{align*} \begin{split} 
\Lambda_3 := \left\| \langle \tau - \phi (\xi, n_2) \rangle^{-\frac12+2\delta} \mathcal F_{x, t} \left[ \langle \p_{x_1} \rangle ( ( \chi(t) S(t)u_{0}^\omega)^2 )  \right] \right\|_{L^\infty} \les 1, \\
		\left\| \langle \tau - \phi (\xi, n_2) \rangle^{-\frac12+2\delta} \mathcal F_{x, t} \left[ \langle \p_{x_1} \rangle ( \chi (t)S(t)u_{0}^\omega \widebar{ \chi(t) \chi (t) S(t)u_{0}^\omega } )  \right] \right\|_{L^\infty} \les 1.
\end{split} \end{align*}
We show the first estimate, since the second follows in an identical way. To that end, we define the renormalized $u$ as in the proof of \eqref{eq:Linfty_2} by $\hat u_{\rm{ren}}(\nu, k_2, \mu) = \langle k \rangle^{-\alpha} \hat \chi (\mu - \phi (\nu, k_2) ) \hat u_0 (\nu, k_2) $. We recall that this guarantees $\| \langle \mu - \phi (\nu, k_2) \rangle^a \hat u_{\rm{ren}} \|_{L^\infty} \les 1$ for all $a$. In particular, we take $a = 3/2 + \delta$ here.

We proceed analogously to the proofs of \eqref{eq:Linfty_1} and \eqref{eq:Linfty_2}, doing a dyadic partition in frequency and using H\"older, but this time with a $L^\infty - L^\infty$ type of bound on the spatial frequencies. We have
\begin{equation} \label{eq:rutte}
L_{\rm{max}}^{\delta} \Lambda_{3, \iota} \les  \sup_{\xi, n_2, \tau'} \frac{N_{3\ast}}{N_1^\alpha N_2^
\alpha} \frac{ L_3^{3\delta} }{L_1^{1/2} L_2^{1/2} L_3^{1/2} } \cdot \frac{|A_{\xi, n_2, \tau}|}{L_{\rm{med}}} \| \langle \mu - \phi (\nu, k_2) \rangle^{3/2 + \delta} \hat u_{\rm{ren}} (\nu, k_2, \mu) \|_{L^\infty_{\nu, k_2} (L^2_\mu) }^2
\end{equation}
where we used that $L_1L_2 \geq L_{\rm{med}}$ in the fraction $|A_{\xi, n_2, \tau} |/ (L_1L_2)$.

First, note from \eqref{eq:rutte} that the trivial bound $|A_{\xi, n_2, \tau} | \les N_{\rm{min}}^2 L_{\rm{med}}$ already works if $L_{\rm{max}} \geq N^5$ just by assuming $\alpha > 1/2$, using that $L_{\rm{med}}$ is bounded by $L_1L_2$. Therefore, we can assume $L_{\rm{max}} \leq N^5$ in the following discussion.

We start with the case with the case of $N_1 \geq 4N_2$ (analogously, $N_2 \geq 4N_1$). Here we bound $|A_{\xi, n_2, \tau}|$ as $|A_{\xi, n_2, \tau}|^{1/2} |A_{\xi, n_2, \tau}|^{1/2}$. The first square root is bounded with \eqref{eq:trivialAbound} and the second one via \eqref{eq:secondAbound}. Denoting $\Lambda_{3, \iota}^{(1)}$ to this contribution, we have: 
\begin{equation*}
	L_{\rm{max}}^{\delta} \Lambda_{3, \iota}^{(1)} \les  \sup_{\xi, n_2, \tau'} \frac{N_{3\ast}}{N_{\rm{min}}^\alpha N_{\rm{max}}^
		\alpha} \frac{ N_{\rm{max}}^{15\delta} }{L_1^{1/2} L_2^{1/2} L_3^{1/2}} \frac{L_{\rm{med}} N_{\rm{min}}^{3/2} L_{\rm{max}}^{1/2}  }{ L_1 L_2 N_{\rm{max}} } \les  N_{\rm{max}}^{-2\alpha + 9\delta} N_{\rm{min}}^{3/2} \les 1
\end{equation*}
assuming $\alpha > 3/4$.

Now, we treat the contribution $\Lambda_{3, \iota}^{(2)}$ coming from the case $N_1 \approx N_2 \approx N_{\rm{max}}$. In such case we bound $|A_{n, \xi, \tau'}|$ by the geometric average of \eqref{eq:firstAbound} and \eqref{eq:thirdAbound}. Performing that average, we obtain
\begin{equation*}
	| A_{\xi, n_2, \tau} | \les L_{\rm{med}}L_{\rm{max}}^{1/2} \frac{N_{\rm{max}}}{N_3^{1/2}}
\end{equation*}
Plugging this into \eqref{eq:rutte}, we obtain
\begin{equation*}
	\Lambda_{3, \iota}^{(2)} \les \sup_{\xi, n_2, \tau'} \frac{N_{3\ast}}{N_1^\alpha N_2^
		\alpha} N_{\rm{max}}^{15\delta}  \frac{N_{\rm{max}}}{ N_3^{1/2} } \les N_{\rm{max}}^{3/2 + 9 \delta - 2\alpha} \les 1
\end{equation*}
using again $\alpha > 3/4$.

\printbibliography

\end{document}